\documentclass[11pt]{article}
\usepackage[T1]{fontenc}
\usepackage{amsmath,amsfonts,amssymb,theorem,color}

\usepackage{imakeidx} %% pour l'index

\usepackage{authblk} %% affil

\usepackage{etoolbox}

\usepackage{hyperref}

\usepackage{mathbbol,MnSymbol}
\theoremstyle{change} 

\usepackage{titlesec}

\newenvironment{proof}{\noindent \bf Proof:   \rm}{\hspace*{\fill}
$\square$ \medskip}

{\theorembodyfont{\itshape}
\newtheorem{lemma}{Lemma}
\newtheorem{proposition}[lemma]{Proposition}

\newtheorem{Facts}[lemma]{Facts}

\newtheorem{notation}[lemma]{Notation}
\newtheorem{corollary}[lemma]{Corollary}
\newtheorem{theorem}[lemma]{Theorem}
\newtheorem{definition}[lemma]{Definition}
\newtheorem{remark}[lemma]{Remark}

\newtheorem{example}[lemma]{Example}
}

\DeclareMathAlphabet{\mathpzc}{OT1}{pzc}{m}{it}
\usepackage[all,cmtip]{xy}%%xymatrix

\usepackage{appendix}

   \mathchardef\mhyphen="2D
   
  \newcommand{\ring}[1]{\mathsf{#1}}

\newcommand{\eq}[1][r]
   {\ar@<-3pt>@{-}[#1]
    \ar@<-1pt>@{}[#1]|<{}="gauche"
    \ar@<+0pt>@{}[#1]|-{}="milieu"
    \ar@<+1pt>@{}[#1]|>{}="droite"
    \ar@/^2pt/@{-}"gauche";"milieu"
    \ar@/_2pt/@{-}"milieu";"droite"}
 
    \DeclareMathSymbol{\Alpha}{\mathalpha}{operators}{"41}
\DeclareMathSymbol{\Beta}{\mathalpha}{operators}{"42}
\DeclareMathSymbol{\Epsilon}{\mathalpha}{operators}{"45}
\DeclareMathSymbol{\Zeta}{\mathalpha}{operators}{"5A}
\DeclareMathSymbol{\Eta}{\mathalpha}{operators}{"48}
\DeclareMathSymbol{\Iota}{\mathalpha}{operators}{"49}
\DeclareMathSymbol{\Kappa}{\mathalpha}{operators}{"4B}
\DeclareMathSymbol{\Mu}{\mathalpha}{operators}{"4D}
\DeclareMathSymbol{\Nu}{\mathalpha}{operators}{"4E}
\DeclareMathSymbol{\Omicron}{\mathalpha}{operators}{"4F}
\DeclareMathSymbol{\Rho}{\mathalpha}{operators}{"50}
\DeclareMathSymbol{\Tau}{\mathalpha}{operators}{"54}
\DeclareMathSymbol{\Chi}{\mathalpha}{operators}{"58}
\DeclareMathSymbol{\omicron}{\mathord}{letters}{"6F}

\begin{document}

\title{Topological rigidity as a monoidal equivalence}   
\author{Laurent Poinsot\footnote{Second address: CREA, 
French Air Force Academy, 
Base a\'erienne 701,
13661 Salon-de-Provence,
France.}}         
\affil{LIPN - UMR CNRS 7030\\
University Paris 13, Sorbonne Paris Cit\'e\\
93140 Villetaneuse, France\\\quad\\
           %<-------------------
laurent.poinsot@lipn.univ-paris13.fr}
 
 \date{}                                  %<-------------------

\maketitle

\begin{abstract}
A topological commutative ring  is said to be {\em rigid} when for every set $X$, the topological dual  of the $X$-fold topological product  of the ring is isomorphic to the free module  over $X$.  Examples are fields with a ring topology, discrete rings, and normed algebras. Rigidity translates into a dual equivalence between categories of free modules and of ``topologically-free'' modules and, with a suitable topological tensor product for the  latter, one proves that it lifts to  an equivalence between monoids in this category (some suitably generalized topological algebras) and coalgebras. In particular, we provide a description of its relationship with the standard duality between algebras and coalgebras, namely finite duality. %Several notions of  topological semisimplicity for these (commutative) monoids are introduced, and over an algebraically closed field $\mathbb{k}$, with a ring topology, all these notions, in addition to cosemisimplicity for coalgebras and classical Jacobson semisimplicity, are proven to coincide. Accordingly this leads to the new rigidity result that the function algebras of $\mathbb{k}$-valued maps defined on a set, with the topology of simple convergence, are the only topologically semisimple monoids. \\

%Any commutative algebra equipped with a derivation may be turned into a Lie algebra under the Wronskian bracket. This provides an entirely new sort of a universal envelope for a Lie algebra, the {\em Wronskian envelope}. The main result of this paper is the characterization of those Lie algebras which embed into their Wronskian envelope as Lie algebras of vector fields on a line. As a consequence we show that, in contrast to the classical situation, most of the free Lie algebras do not embed into their Wronskian envelope.  One also provides a differential Hopf algebra structure on the Wronskian envelope of a free Lie algebra, inherited from a cogroup structure on the free Lie algebra.\\
\end{abstract}

%16W80 Topological and ordered rings and modules
% 54H13 Topological fields, rings, etc
%46A20 Duality theory
%13J99 None of the above, but in this section (Topological rings and modules)
%16D60 Simple and semisimple modules, primitive rings and ideals
%19D23 Symmetric monoidal categories
%16N20 Jacobson radical, quasimultiplication

%\begin{footnotesize}
%\tableofcontents
%\end{footnotesize}

\section{Introduction}\label{sec:intro}

The main result of~\cite{Poinsot_Rigidity} states that given a (Hausdorff\footnote{All topologies will be assumed separated.}) topological field $(\mathbb{k},\tau)$, for every set $X$, the topological dual $((\mathbb{k},\tau)^X)'$ of the $X$-fold topological product $(\mathbb{k},\tau)^X$ is isomorphic to the vector space $\mathbb{k}^{(X)}$ of finitely-supported $\mathbb{k}$-valued maps defined on $X$ (i.e., those maps $X\xrightarrow{f}\mathbb{k}$ such that for all but finitely many members $x$ of $X$, $f(x)=0$). %Roughly speaking, this means that the topological dual of $(\mathbb{k},\tau)^X$ does not depend on the field topology $\tau$ and which only really  matters, is the fact that it has a product topology. This is essentially the main result of~\cite{Poinsot_Rigidity} of which this contribution is a continuation.

Actually this topological property of {\em rigidity} is shared by more general topological (commutative unital) rings\footnote{In this contribution, every ring is assumed commutative and unital (see Section~\ref{conv:convention1}).} than only  topological fields (a fact not noticed in~\cite{Poinsot_Rigidity}). For instance any discrete  ring is rigid in the above sense (see Lemma~\ref{lem:rigidity_of_discrete_rings}). And even if not all topological rings are rigid (see Section~\ref{sec:counterexample} for a counter-example), many of them still are (e.g., every real or complex normed commutative algebra). 

It is our intention to study in more details some consequences of the property of rigidity for arbitrary commutative rings in particular for some of their topological algebras\footnote{The results of the present contribution also serve in a subsequent paper under preparation about topological semisimplicity of commutative topological algebras.}. So far, for a topological ring $(\ring{R},\tau)$, rigidity reads as $((R,\tau)^X)'\simeq R^{(X)}$ (here, and everywhere else, $R$ stands for the canonical left $\ring{R}$-module structure on the underlying abelian group of $\ring{R}$) for each set $X$. Suitably topologized (see Section~\ref{subsec:algebraic_dual_functor}), the algebraic dual $(R^{(X)})^*$ turns out to be isomorphic to $(R,\tau)^X$. %Accordingly rigidity leads to a one-one correspondence between free $\ring{R}$-modules $R^{(X)}$ and topological $(\ring{R},\tau)$-modules of the form $(R,\tau)^{X}$, said to be  {\em topologically-free}. In this contribution one observes that topologically-free modules are particularly nice since they have topological bases which play a r\^ole similar to linear bases of free modules (or vector spaces) and make them quite easy to handle. 

More appropriately   the above correspondence may be upgraded into a  dual equivalence of categories\footnote{A {\em dual} equivalence is an equivalence between a category and the opposite  of another.} between free  and topologically-free modules, i.e., those topological modules isomorphic to some $(R,\tau)^X$ (Theorem~\ref{thm:equiv_of_cats}) under the algebraic and topological dual functors. (This extends a similar interpretation from~\cite{Poinsot_Rigidity} to the realm of arbitrary commutative rigid rings.)%A similar interpretation of the main result~\cite[Theorem~5, p.~156]{Poinsot_Rigidity} was given in the quoted reference as a dual equivalence between some categories of (topological) bimodules over a topological division ring. At this point what is new in this contribution is a similar interpretation for arbitrary (commutative) rigid rings.)

Under the rigidity assumption, the aforementioned dual equivalence enables to provide a topological tensor product $\ostar_{(\ring{R},\tau)}$ for topological\-ly-free $(\ring{R},\tau)$-modules by transporting the algebraic tensor product $\otimes_{\ring{R}}$ along the dual equivalence.  It turns out that $\ostar_{(\ring{R},\tau)}$ is (coherently) associative, commutative and unital, i.e., makes monoidal the category of topologically-free modules (Proposition~\ref{prop:ostar_is_a_coherent_monoidal_functor}). Not too surprisingly the above dual equivalence remains well-behaved, i.e., monoidal,  with respect to the (algebraic and topological) tensor products (Theorem~\ref{thm:monoidal_equivalence}). According to the theory of monoidal categories, this in turn provides a dual equivalence between monoids in the tensor category of topologically-free modules (some suitably generalized topological algebras) and coalgebras (Corollary~\ref{cor:induced_equivalence_between_top_mon_and_coalgebras}).  So there are two constructions: a {\em topological dual coalgebra} of a monoid (in the tensor category of topologically-free modules) and an {\em algebraic dual monoid} of a coalgebra, and these constructions are inverse one from the other (up to isomorphism).

There already exists a standard duality theory between algebras and coalgebras, over a field, known as {\em finite duality}  but contrary to our ``topological duality'' it is merely an adjunction, not an equivalence. One discusses how these dualities interact  
(see Section~\ref{subsub:relation_with_finite_duality}) and in particular one proves that the algebraic dual monoid of a coalgebra essentially corresponds to  its finite dual  (Section~\ref{paragraph:relation_with_algebraic_dual}), that over a discrete field, the topological dual coalgebra of a monoid  is a subcoalgebra of the finite dual coalgebra of its underlying  algebra  and furthermore that they are equal exactly when finite duality provides an equivalence of categories (Theorem~\ref{thm:findual}).

\section{Conventions, notations and basic definitions}\label{sec:conventions}

\subsection{Conventions}\label{conv:convention1}
%\begin{Convs}\label{conv:convention1}
%\begin{enumerate}
%\item\label{conv1:item1} 

Except as otherwise stipulated, all topologies are Hausdorff, and every ring is assumed  unital and commutative\footnote{To the contrary an algebra over a ring won't be assumed commutative, but  associative and unital.}. 

%\item 

For a ring $\mathsf{R}$\index{R1@$\ring{R}$}, $R$\index{R2@$R$} denotes both its underlying set and the canonical left $\ring{R}$-module structure on its underlying additive group. Likewise if $\ring{A}$ is an $\ring{R}$-algebra, then $A$ is both its underlying set and its underlying $\ring{R}$-module. The unit of a ring $\mathsf{R}$ (resp., unital algebra $\ring{A}$) is either denoted by $1_{\mathsf{R}}$\index{$1_{\ring{R}}$} (resp. $1_{\ring{A}}$). A ring map (or morphism of rings) is assumed to preserve the units. %The bilinear multiplication will be denoted $m_{\ring{R}}(x,y)$, $m_{\ring{A}}(x,y)$, $m(x,y)$, by juxtaposition $xy$ or even using an {\em ad hoc} notation.

%\item 

%\item A module $M$ over a ring $\mathsf{R}$ is actually a unital\footnote{This means that $1_R x=x$ for all $x\in M$.} left-$\mathsf{R}$-module, and when $\mathsf{R}$ is a field -- usually denoted $\mathbb{k}$\index{k@$\mathbb{k}$} --  it is called a  vector space.

%\item\label{conv1:item4} 

A product of topological spaces always has the product topology. When for each $x\in X$, all $(E_x,\tau_x)$'s are equal to the same topological space $(E,\tau)$, then the $X$-fold topological product $\prod_{x\in X}(E_x,\tau_x)$ is canonically identified with the set $E^X$ of all maps from $X$ to $E$ equipped with the topology of simple convergence, and is denoted by $(E,\tau)^X$. Under this identification, the canonical projections\index{Canonical projections} $(E,\tau)^X\xrightarrow{\pi_x}(E,\tau)$\index{$\pi_x$} are given by $\pi_x(f)=f(x)$, $x\in X$, $f\in E^X$.
\subsection{Basic definitions}
\begin{definition}\label{def:top_ring_top_field}
Let $\ring{R}$ be a ring. A (Hausdorff, following our conventions) topology  $\tau$ of (the carrier set of) the ring is called a {\em ring topology}\index{Ring topology} when addition, multiplication and additive inversion of the ring are continuous. 
By {\em topological ring}\index{Topological ring} $(\ring{R},\tau)$\index{R3@$(\ring{R},\tau)$} is meant a ring together with a ring topology $\tau$ on it\footnote{In view of Section~\ref{conv:convention1}, the multiplication of a topological ring is jointly continuous.}.  By a {\em field with a ring topology}\index{Field with a ring topology}, denoted $(\mathbb{k},\tau)$, is meant a topological ring $(\mathbb{k},\tau)$ with $\mathbb{k}$ a field.  
\end{definition}

\begin{example}
A ring $\ring{R}$ with the discrete topology $\mathsf{d}$ is a topological ring. %Actually the multiplication $R\times R\xrightarrow{m_{\ring{R}}}R$ is even uniformly continuous (since $m_{\ring{R}}(\{\, 0\,\}\times R)+m_{\ring{R}}(R\times\{\, 0\,\})=\{\, 0\,\}$).
\end{example}

Let  $(\ring{R},\tau)$ be a topological ring.
%\begin{enumerate}
%\item 
A pair $(M,\sigma)$ consisting of a (left and unital\footnote{Unital means that the scalar action of the unit of $\ring{R}$ is the identity on the module.}) $\ring{R}$-module $M$ and a topology $\sigma$ on $M$ which makes continuous the addition, opposite and scalar multiplication $R\times M\to M$, is called a {\em topological $(\ring{R},\tau)$-module}\index{Topological module}. Such a topology is referred to as a {\em $(\ring{R},\tau)$-module topology}\index{Module topology}. In particular, when $\ring{R}$ is a field $\mathbb{k}$, then this provides {\em topological $(\mathbb{k},\tau)$-vector spaces}\index{Topological vector space}. %A topological $(\mathbb{Z},\mathsf{d})$-module is called a {\em topological  abelian group}.
Given  topological $(\mathsf{R},\tau)$-modules $(M,\sigma),(N,\gamma)$, a {\em continuous $(\ring{R},\tau)$-linear map} $(M,\sigma)\xrightarrow{f}(N,\gamma)$ is a $R$-linear map $M\xrightarrow{f}N$ which is continuous. Topological $(\ring{R},\tau)$-modules and these morphisms form a category\footnote{One does not worry about size issues and in this presentation ``category'' means a locally small category while ``set'' loosely means both small and large set. One assumes that the reader is familiar with basic notions from category theory among subcategories,  (full, faithful) functors, natural transformations, natural isomorphisms, equivalence of categories, categorical product, terminal object, left/right adjoints, unit and counit of an adjunction. However some of them will be recalled in the text, essentially through footnotes. 
Of course~\cite{MacLane} is a fundamental reference for this subject.}  $\mathbf{TopMod}_{(\ring{R},\tau)}$\index{TopMod@$\mathbf{TopMod}_{(\ring{R},\tau)}$}, which is denoted $\mathbf{TopVect}_{(\mathbb{k},\tau)}$\index{TopVect@$\mathbf{TopVect}_{(\mathbb{k},\tau)}$}, when $\mathbb{k}$ is a field.

%\item 
A pair $(\ring{A},\sigma)$, with $\ring{A}$ a unital $\ring{R}$-algebra, and $\sigma$ a topology on $A$, is a {\em topological $(\ring{R},\tau)$-algebra}\index{Topological algebra}, when $\sigma$ is a module topology for the underlying $\ring{R}$-module $A$, and the multiplication of $\ring{A}$ is a bilinear (jointly) continuous map. %A topological  (unital, commutative and non-trivial) $(\mathbb{Z},\mathsf{d})$ -algebra thus is a topological ring.
Given topological $(\ring{R},\tau)$-algebras $(\ring{A},\sigma),(\ring{B},\gamma)$, a {\em continuous $(\ring{R},\tau)$-algebra map} $(\ring{A},\sigma)\xrightarrow{f}(\ring{B},\gamma)$ is a unit-preserving $\ring{R}$-algebra map $\ring{A}\xrightarrow{f}\ring{B}$ which is also continuous. Topological $(\ring{R},\tau)$-algebras with these morphisms form a category ${}_1\mathbf{TopAlg}_{(\ring{R},\tau)}$\index{TopAlg@$\mathbf{TopAlg}_{(\ring{R},\tau)}$}. One also has the full subcategory\footnote{By a {\em full subcategory} is meant a subcategory $\mathbf{D}$ of $\mathbf{C}$ such that $\mathbf{D}(C,D)=\mathbf{C}(C,D)$ for each $\mathbf{D}$-objects $C,D$, where as usually, $\mathbf{C}(C,D)$ denotes the {\em hom-set} of all $\mathbf{C}$-morphisms with domain $C$ and codomain $D$.}  ${}_{1,c}\mathbf{TopAlg}_{(\ring{R},\tau)}$\index{TopAlg4@${}_{1,c}\mathbf{TopAlg}_{(\ring{R},\tau)}$} of unital and commutative topological algebras.
%\end{enumerate}

%\begin{remark}
%For each topological ring $(\ring{R},\tau)$, one has obvious forgetful and embedding functors which make commute the following diagram.
%\begin{equation}
%\xymatrix@R=0.8pc@C=1pc{
%\ar[ddd]{}_{1}\mathbf{TopAlg}_{(\ring{R},\tau)} \ar[rrr]&&&{}_{1}\mathbf{Alg}_{\ring{R}}\ar[ddd]\\
%&\ar@{_{(}->}[lu]{}_{1,c}\mathbf{TopAlg}_{(\ring{R},\tau)}\ar[d]\ar[r]&{}_{1,c}\mathbf{Alg}_{\ring{R}}\ar@{^{(}->}[ru]\ar[d]&\\
%&\ar@{_{(}->}[ld]{}_{c}\mathbf{TopAlg}_{(\ring{R},\tau)} \ar[r]&{}_c\mathbf{Alg}_{\ring{R}}\ar@{^{(}->}[rd]&\\
%\mathbf{TopAlg}_{(\ring{R},\tau)} \ar[rrr]\ar[rd]&&& \mathbf{Alg}_{\ring{R}}\ar[ld]\\
%&\mathbf{TopMod}_{(\ring{R},\tau)} \ar[r]& \mathbf{Mod}_R&
%}
%\end{equation}
%\end{remark}

\begin{example}
$(R,\tau)$  is a topological $(\ring{R},\tau)$-module and the topological ring $(\ring{R},\tau)$ with the previous module structure, is a topological $(\ring{R},\tau)$-algebra.
\end{example}

\begin{remark}
A $\ring{R}$-module (resp. $\ring{R}$-algebra) with the discrete topology $\ring{d}$\index{d@$\ring{d}$} is a topological $(\ring{R},\mathsf{d})$-module (resp. $(\ring{R},\mathsf{d})$-algebra).
\end{remark}

\subsection{$X$-fold product and finitely-supported maps}\label{sec:ex:x-fold_product_and_fin_supp_maps}
The {\em opposite category} $\mathbf{C}^{\mathsf{op}}$ of $\mathbf{C}$ has the same objects and morphisms as $\mathbf{C}$ but with opposite composition. In other words one has $\mathbf{C}^{\mathsf{op}}(D,C)=\mathbf{C}(C,D)$. $f^{\mathsf{op}}$ denotes the $\mathbf{C}$-morphism $f$ considered as a $\mathbf{C}^{\mathsf{op}}$-morphism.
%\begin{Facts}\label{appendix:cattheory:facts}\index{Opposite (category, functor, natural transformation)}
%\begin{enumerate}
%\item $(\mathbf{C}\times\mathbf{D})^{\mathsf{op}}=\mathbf{C}^{\mathsf{op}}\times\mathbf{D}^{\mathsf{op}}$.
Let $\mathbf{C}\xrightarrow{F}\mathsf{D}$ be a functor. Then, $\mathbf{C}^{\mathsf{op}}\xrightarrow{F^{\mathsf{op}}}\mathbf{D}^{\mathsf{op}}$ with for a $\mathbf{C}$-morphism $f$, $F^{\mathsf{op}}(f^{\mathsf{op}})=F(f)^{\mathsf{op}}$, is called the {\em opposite} of $F$.  %Actually the opposite construction provides a functor of the meta-category of categories. In particular, $id_{\mathbf{C}}^{\mathsf{op}}=id_{\mathbf{C}^{\mathsf{op}}}$, and $(G\circ F)^{\mathsf{op}}=G^{\mathsf{op}}\circ F^{\mathsf{op}}$ for composable functors $G,F$.
 Each natural transformation\footnote{The notation $\alpha\colon F\Rightarrow G\colon \mathbf{C}\to\mathbf{D}$ means that $\alpha$ is a {\em natural transformation} between two functors $\mathbf{C}\xrightarrow{F,G}\mathbf{D}$. Given a $\mathbf{C}$-object $C$, $\alpha_C\in \mathbf{D}(F(C),G(C))$ denotes the {\em component at $C$} of $\alpha$. Thus $\alpha=(\alpha_C)_C$. For each functor $\mathbf{C}\xrightarrow{F}\mathbf{D}$, there is an {\em identity} at $F$, $id_F\colon F\Rightarrow F\colon \mathbf{C}\to\mathbf{D}$ with $(id_{F})_C:=id_{F(C)}$, also denoted simply $id$.} $\alpha\colon F\Rightarrow G\colon \mathbf{C}\to\mathbf{D}$ has an {\em opposite natural transformation} $\alpha^{\mathsf{op}}\colon G^{\mathsf{op}}\Rightarrow F^{\mathsf{op}}\colon \mathbf{C}^{\mathsf{op}}\to\mathbf{D}^{\mathsf{op}}$ with $(\alpha^{\mathsf{op}})_C=\alpha_C^{\mathsf{op}}$ for each $\mathbf{C}$-object $C$. %$\alpha^{\mathsf{op}}$ is the {\em opposite of $\alpha$}. In particular, $\mathbf{C}\simeq \mathbf{D}$ if, and only if, $\mathbf{C}^{\mathsf{op}}\simeq\mathbf{D}^{\mathsf{op}}$\index{$\alpha^{\mathsf{op}}$}.
%\end{enumerate}
%\end{Facts}

%\paragraph{(Topological) module structures.}

Let $\ring{R}$ be a ring, and let $X$ be a set. The $\ring{R}$-module $R^X$ of all maps from $X$ to $R$, equivalently defined as the $X$-fold power of $R$ in the category\footnote{$\mathbf{Mod}_{\ring{R}}$\index{Mod@$\mathbf{Mod}_{\ring{R}}$} is the category of unital left-$\ring{R}$-modules with $\ring{R}$-linear maps. When $\ring{R}$ is a field  $\mathbb{k}$ one uses $\mathbf{Vect}_{\mathbb{k}}$\index{Vect@$\mathbf{Vect}_{\mathbb{k}}$} instead.} $\mathbf{Mod}_{\ring{R}}$, is the object component of a functor\footnote{$\mathbf{Set}$\index{Set@$\mathbf{Set}$} is the category of sets with set-theoretic maps.} $\mathbf{Set}^{\mathsf{op}}\xrightarrow{P_{\ring{R}}}\mathbf{Mod}_{\ring{R}}$\index{P@$P_{\ring{R}}$} whose  action on maps is as follows: given $X\xrightarrow{f}Y$ and $g\in R^{Y}$, $P_{\ring{R}}(f)(g)=g\circ f$. $R^{X}$ is merely not just a $\ring{R}$-module but, under point-wise multiplication\index{M@$M_X$} $R^{X}\times R^{X}\xrightarrow{M_X}R^{X}$, a commutative $\ring{R}$-algebra, the usual {\em function algebra}\index{Function algebra} on $X$, denoted $A_{\ring{R}}(X)$, with unit $1_{A_{\ring{R}}(X)}:=\sum_{x\in X}\delta_{x}^{\ring{R}}$, where $\delta_x^{\ring{R}}$\index{$\delta_x^{\ring{R}}$ (or $\delta_x$)}, or simply $\delta_x$, is the member of $R^X$ with $\delta_x^{\ring{R}}(x)=1_{\ring{R}}$, $x\in R$, and for $y\in X$, $y\not=x$, $\delta_x^{\ring{R}}(y)=0$. This actually provides a functor\footnote{${}_1\mathbf{Alg}_{\ring{R}}$\index{Alg@${}_1\mathbf{Alg}_{\ring{R}}$} the category of (associative) unital $\ring{R}$-algebras with unit-preserving algebra maps. (The multiplication $m_{\ring{A}}$ of an algebra $\mathsf{A}$ thus is a  $\ring{R}$-bilinear map $A\times A\xrightarrow{m_{\mathsf{A}}}A$.) ${}_{1,c}\mathbf{Alg}_{\ring{R}}$\index{Alg4@${}_{1,c}\mathbf{Alg}_{\ring{R}}$} is the category of unital and  commutative algebras.} $\mathbf{Set}^{\mathsf{op}}\xrightarrow{A_{\ring{R}}}{}_{1,c}\mathbf{Alg}_{\ring{R}}$\index{A@$A_{\ring{R}}$}.

%Let $\ring{R}$ be a ring and let $X$ be a set. Let $x\in X$. Then, $\delta_x^{\ring{R}}$\index{$\delta_x^{\ring{R}}$ (or $\delta_x$)}, or simply $\delta_x$, is the member of $R^X$ given for $y\in X$, by $\delta_x^{\ring{R}}(y)=\left\{\begin{array}{ll}
%1_{\ring{R}} & \mbox{if}\ x=y\\
%0 & \mbox{if}\ x\not=y
%\end{array}\right . .$

Let $f\in R^X$. The {\em support}\index{Support} of $f$ is the set $supp(f):=\{\, x\in X\colon f(x)\not=0\,\}$\index{supp@$supp(f)$}. Let $R^{(X)}$\index{R5@$R^{(X)}$} be the sub-$\ring{R}$-module of $R^X$ consisting of all {\em finitely-supported maps} (or maps with {\em finite support})\index{Finitely-supported maps}, i.e., the maps $f$ such that $supp(f)$ is finite. 

$R^{(X)}$ is actually the free $\ring{R}$-module over $X$, and a basis is given by $\{\, \delta_x^{\ring{R}}\colon x\in X\,\}$. %\left\{\begin{array}{ll}
%1_{\ring{R}} & \mbox{if}\ x=y\\
%0 & \mbox{if}\ x\not=y
%\end{array}\right . .$
%\begin{remark}\label{rem:unit_for_free_modules}
Observe that the map $X\xrightarrow{\delta^{\ring{R}}_X}R^X$\index{$\delta_X^{\ring{R}}$}, $x\mapsto \delta_x$, is one-to-one if, and only if, $\ring{R}$ is  non-trivial.
%\end{remark}
%\begin{notation}\label{ex:freemod}
%Let  $\ring{R}$ be a ring and let $X$ be a set. Let $x\in X$, and define $p_x:=R^{(X)}\hookrightarrow R^X\xrightarrow{\pi_x}R$\index{p@$p_x$}, where $R^{(X)}\hookrightarrow R^X$ is the canonical inclusion. 
%\end{notation}

\begin{remark}\label{rem:categorytheoretic_recasting_of_free_modules}
Of course one has the {\em free module functor} $\mathbf{Set}\xrightarrow{F_{\ring{R}}}\mathbf{Mod}_{\ring{R}}$\index{F@$F_{\ring{R}}$} which is a left adjoint of  the usual forgetful functor $\mathbf{Mod}_{\ring{R}}\xrightarrow{|\cdot|}\mathbf{Set}$\index{$\mid\cdot\mid$}; $F_{\ring{R}}(X):=R^{(X)}$, and for $X\xrightarrow{f}Y$, $p\in R^{(X)}$, $F_{\ring{R}}(f)(p):=\sum_{y\in Y}\left (\sum_{x\in f^{-1}(\{\, y\,\})}p(x)\right)\delta_y^{\ring{R}}$, i.e.,  $F_{\ring{R}}(f)(\delta^{\ring{R}}_x)=\delta_{f(x)}^{\ring{R}}$, $x\in X$. The map $X\xrightarrow{\delta_X^{\ring{R}}}|R^{(X)}|$ is the component at $X$ of the unit of the adjunction\footnote{\label{footnote:dashv}In this presentation, by $F\dashv G\colon\mathbf{C}\to \mathbf{D}$\index{$F\dashv G$} is meant an adjunction with left adjoint $\mathbf{C}\xrightarrow{F}\mathbf{D}$ and right adjoint $\mathbf{D}\xrightarrow{G}\mathbf{C}$.} $F_{\ring{R}}\dashv |\cdot|\colon \mathbf{Set}\to\mathbf{Mod}_{\ring{R}}$.
%
%\item The construction of $(R,\tau)^X$ is also (contravariently) functorial in $X$. Let $\mathbf{Set}^{\mathsf{op}}\xrightarrow{P_{(\ring{R},\tau)}}\mathbf{TopMod}_{(\ring{R},\tau)}$\index{P2@$P_{(\ring{R},\tau)}$} be given by $P_{(\ring{R},\tau)}(X):=(R,\tau)^X$, and $P_{(\ring{R},\tau)}(f)(g):=g\circ f$, $X\xrightarrow{f}Y$, $g\in R^{Y}$.  In other words, for each $x\in X$, $\pi_x\circ P_{(\ring{R},\tau)}(f)=\pi_{f(x)}$ (which also ensures continuity of $P_{(\ring{R},\tau)}(f)$). 
%
%Of course, the following diagram commutes (where the unnamed arrow is the obvious forgetful functor).
%\begin{equation}\label{diag:prod_and_top_prod}
%\xymatrix@R=0.7pc{
%\mathbf{Set}^{\mathsf{op}} \ar[r]^-{P_{(\ring{R},\tau)}}\ar[rd]_{P_{\ring{R}}}& \mathbf{TopMod}_{(\ring{R},\tau)}\ar[d]\\
%&\mathbf{Mod}_{\ring{R}}
%}
%\end{equation}
%\end{enumerate}
\end{remark}

Let  $(\ring{R},\tau)$ be a topological ring and let $X$ be a set. Since for a map $X\xrightarrow{f}Y$, $\pi_x\circ P_{(\ring{R},\tau)}(f)=\pi_{f(x)}$, $x\in X$, $(R,\tau)^Y\xrightarrow{P_{\ring{R}}(f)}(R,\tau)^X$\index{R4@$(R,\tau)^X$} is continuous, and thus one has a {\em topological power functor}\index{Topological power functor} $\mathbf{Set}^{\mathsf{op}}\xrightarrow{P_{(\ring{R},\tau)}}\mathbf{TopMod}_{(\ring{R},\tau)}$\index{P2@$P_{(\ring{R},\tau)}$}. 

\begin{lemma}\label{lem:algebra_structure_on_fun_space_and_on_finsupp_maps}
$(R,\tau)^X\times(R,\tau)^X\xrightarrow{M_X}(R,\tau)^X$ is continuous. %In consequence $A_{(\ring{R},\tau)}(X):=((R,\tau)^{X},M_X,1_{\ring{R}^{X}})$ is an object of ${}_{1,c}\mathbf{TopAlg}_{(\ring{R},\tau)}$.
\end{lemma}

\begin{proof}
$M_X$ is separately continuous because for each $x\in X$, $\pi_x\circ M_{X}=m_\ring{R}\circ (\pi_x \times \pi_x)$. Let $A\subseteq X$ be finite, and for each $x\in A$, let $U_x$ be an open neighborhood zero in $(R,\tau)$.  Continuity of $m_{\ring{R}}$ at zero implies the existence of neighborhoods $V_x,W_x$ of zero such that $m_{\ring{R}}(V_x,W_x)\subseteq U_x$. $M_X(\bigcap_{x\in A}\pi_x^{-1}(V_x),\bigcap_{x\in A}\pi_{x}^{-1}(W_x))\subseteq \bigcap_{x\in A}\pi_{x}^{-1}(U_x)$ ensures continuity at zero of $M_X$, and thus continuity by~\cite[Theorem~2.14, p.~16]{Warner}.
\end{proof}

$A_{(\ring{R},\tau)}\colon X\mapsto ((R,\tau)^X,M_X,1_{A_{\ring{R}}(X)})$\index{A2@$A_{(\ring{R},\tau)}$} is a functor too and the diagram below commutes, with the forgetful functors unnamed.
\begin{equation}
\xymatrix@R=0.7pc{
{}_{1,c}\mathbf{TopAlg}_{(\ring{R},\tau)}\ar[rr]\ar[dd] && {}_{1,c}\mathbf{Alg}_{\ring{R}}\ar[dd]\\
&\ar[lu]_-{A_{(\ring{R},\tau)}}\ar[ld]^-{P_{(\ring{R},\tau)}}\mathbf{Set}^{\mathsf{op}}\ar[ru]^-{A_{\ring{R}}}\ar[rd]_-{P_{\ring{R}}}&\\
\mathbf{TopMod}_{(\ring{R},\tau)} \ar[rr]&& \mathbf{Mod}_{\ring{R}}
}
\end{equation}

\begin{notation}\label{notation:top_ring_of_functions}
The underlying topological ring of $A_{(\ring{R},\tau)}(X)$ is denoted $(\ring{R},\tau)^X$\index{R8@$(\ring{R},\tau)^X$} (of course, it is the $X$-fold product of $(\ring{R},\tau)$ in the category of rings).
\end{notation}

\section{Recollection of results about algebraic and topological duals}\label{subsec:recollection}%Rigid rings, definition and examples}

%The topological ``rigidity'' result, recalled at the beginning of the Introduction, was already upgraded to an equivalence of categories in~\cite{Poinsot_Rigidity} using the algebraic and topological dual functors. Since the quoted reference  dealt with topological bimodules over a topological division ring, a distinction between the left and right (topological or algebraic) duals was made there, a distinction which becomes irrelevant when dealing with topological modules over a topological commutative ring as is the case of the present contribution. (Note that a left module on a commutative ring is of course a bimodule with both the left and right module structures identical.) 
%
%In what follows when a result may be easily adapted from a corresponding result from~\cite{Poinsot_Rigidity} one omits the proof and only provides the precise reference in the statement. 

\subsection{Algebraic dual functor}\label{subsec:algebraic_dual_functor}

Let $\ring{R}$ be a ring. Let $M$ be a $\ring{R}$-module. Let $M^*:=\mathbf{Mod}_{\ring{R}}(M,R)$ be the {\em algebraic} (or {\em linear}) {\em dual} of $M$\index{Algebraic (or linear) dual}. This is readily a $\ring{R}$-module on its own. 

When $(\ring{R},\tau)$ is a topological ring, then $M^*$ may be topologized with the initial topology (\cite{BouTop,Warner}) $w_{(\ring{R},\tau)}^*$\index{w@$w_{(\ring{R},\tau)}^*$}, called the {\em weak-$*$ topology}\index{Weak-$*$ topology}, induced by the family $(M^*\xrightarrow{\Lambda_M(v)}R)_{v\in M}$\index{$\Lambda_M$} of {\em evaluations} at some points, where $(\Lambda_M(v))(\ell):=\ell(v)$. This provides a structure of topological $(\ring{R},\tau)$-module on $M^*$, which is even Hausdorff\footnote{The initial topology on $X$ induced by a family $F$ of maps all with domain $X$, is Hausdorff when $F$ separates the points of $X$, i.e., when for each $x\not=y$ in $X$, there is a map $f\in F$ such that $f(x)\not=f(y)$.} (since if $\ell(v)=0$ for all $v\in M$, then $\ell=0$). 

Moreover given a linear map $M\xrightarrow{f}N$, $N^*\xrightarrow{f^*}M^*$, $\ell\mapsto f^*(\ell):=\ell\circ f$, is continuous for the above topologies. Consequently,  this provides a functor $\mathbf{Mod}_{\ring{R}}^{\mathsf{op}}\xrightarrow{{Alg}_{(\ring{R},\tau)}}\mathbf{TopMod}_{(\ring{R},\tau)}$\index{Alg@$Alg_{(\ring{R},\tau)}$} called the {\em algebraic dual functor}\index{Algebraic dual functor}. %from $\mathbf{Mod}^{\mathsf{op}}_R$ to $\mathbf{TopMod}_{(R,\tau)}$. 

\begin{remark}
$M^*$ or $f^*$ stand for $Alg_{(\ring{R},\tau)}(M)$ or $Alg_{(\ring{R},\tau)}(f)$.
\end{remark}

Up to isomorphism, one recovers the module of all $\ring{R}$-valued maps on a set $X$, with its product topology, as the algebraic dual of the module of all finitely-supported maps duly topologized as above.
\begin{lemma}\label{lem:rho_X_is_iso}
For each set $X$, $ (R,\tau)^X\simeq Alg_{(\ring{R},\tau)}(R^{(X)})$  (in $\mathbf{TopMod}_{(\ring{R},\tau)}$) under the map\index{$\rho_X$} $$\rho_X\colon (R,\tau)^X\to((R^{(X)})^*,w^*_{(\ring{R},\tau)})$$ given by $(\rho_X(f)(p)):=\sum_{x\in X}p(x)f(x)$, $f\in R^X$, $p\in R^{(X)}$. 
\end{lemma}
\begin{proof}
Let $\ell\in (R^{(X)})^*$. Let us define $X\xrightarrow{\hat{\ell}}R$ by $\hat{\ell}(x):=\ell(\delta_x)$, $x\in X$.  That the two constructions are linear and inverse one from the other is clear. 

It remains to make sure that there are also continuous. Let $\ell\in (R^{(X)})^*$, and let $x\in X$. Then, $\pi_x(\hat{\ell})=\hat{\ell}(x)=\ell(\delta_x)=(\Lambda_{R^{(X)}}({\delta_x}))(\ell)$, which ensures  continuity of $((R^{(X)})^*,w_{(\ring{R},\tau)}^*)\xrightarrow{\rho_X^{-1}}(R,\tau)^X$. Let $f\in R^{X}$, and $p\in R^{(X)}$. As $(\Lambda_{R^{(X)}}(p))(\rho_X(f))=(\rho_X(f))(p)=\sum_{x\in X}p(x)f(x)=\sum_{x\in X}\pi_x(p)f(x)=\sum_{x\in X}\pi_x(p)\pi_x(f)$, $\Lambda_{R^{(X)}}(p)\circ \rho_X$ is a finite linear combination of projections, whence is continuous for the product topology, so is $\rho_X$. 
\end{proof}

%\begin{definition}
Let $M$ be a free $\ring{R}$-module. Let $B$ be a basis of $M$. This defines a family of $\ring{R}$-linear maps, the {\em coefficient maps}\index{Coefficient maps} $(M\xrightarrow{b^*}R)_{b\in B}$\index{b@$b^*$} such that each $v\in M$ is uniquely represented as a finite linear combination $v=\sum_{b\in B}b^*(v)b$. One denotes $\mathbf{FreeMod}_{\ring{R}}$\index{FreeMod@$\mathbf{FreeMod}_{\ring{R}}$} the full subcategory of $\mathbf{Mod}_{\ring{R}}$ spanned\footnote{Each (even large) subset $X$ of the objects of some category $\mathbf{C}$ determines uniquely a full subcategory of $\mathbf{C}$ called the {\em full subcategory of $\mathbf{C}$ spanned by $X$}, namely the subcategory $\mathbf{C}_X$ whose set of objects is $X$ and $\mathbf{C}_X(C,D)=\mathbf{C}(C,D)$, $C,D\in X$.} by the free modules. When $\mathbb{k}$ is a field, $\mathbf{FreeMod}_{\mathbb{k}}$  is just $\mathbf{Vect}_{\mathbb{k}}$ itself.%\end{definition}

%\begin{remark}

%\end{remark}

\begin{example}\label{ex:p_x}
For each set $X$, $p_x=\delta_x^*$, where $p_x:=R^{(X)}\hookrightarrow R^X\xrightarrow{\pi_x}R$\index{p@$p_x$},   $x\in X$, with $R^{(X)}\hookrightarrow R^X$ the canonical inclusion.
\end{example}

\begin{remark}\label{rem:coefficient_maps}
$b^*(d)=\delta_b(d)$, $b,d\in B$. So $B\xrightarrow{(-)^*}B^*:=\{\, b^*\colon b\in B\,\}$\index{B@$B^*$} is a bijection. 
\end{remark}

Given a free $\ring{R}$-module, any choice of a basis $B$ provides the initial topology $\Pi^{\tau}_B$ on $M^*$ induced by $(\Lambda_M(b))_{b\in B}$. (Of course, $\Pi_{B}^{\tau}\subseteq w^*_{(\ring{R},\tau)}$.)%In view of Remark~\ref{rem:coefficient_maps}, it is Hausdorff.)

\begin{lemma}\label{lem:dual_of_free_is_top_free_mod}
Let $M$ be a free $\ring{R}$-module. The topology $\Pi^{\tau}_B$ is independent of the choice of the basis $B$ of $M$ since it is equal to $w^*_{(\ring{R},\tau)}$. Moreover, for each basis $B$ of $M$, $(M^*,w^*_{(\ring{R},\tau)})\simeq (R,\tau)^B$ (in $\mathbf{TopMod}_{(\ring{R},\tau)}$).
\end{lemma}
\begin{proof}
For $\ell\in M^*$, $(\Lambda_M(v))(\ell)=\sum_{b\in B}b^*(v)\ell(b)=\sum_{b\in B}b^*(v) (\Lambda_M({b}))(\ell)$, $v\in M$,  thus $\Lambda_M(v)$ is a finite linear combination of some $\Lambda_M({b})$'s, whence is continuous for $\Pi^{\tau}_B$, and so $w^*_{(\ring{R},\tau)}\subseteq \Pi^{\tau}_B$. The last assertion is clear.
\end{proof}

\subsection{Topological dual functor}

Let $(\ring{R},\tau)$ be a topological ring, and let $(M,\sigma)$ be a topological $(\ring{R},\tau)$-module. Let $(M,\sigma)':=\mathbf{TopMod}_{(\ring{R},\tau)}((M,\sigma),(R,\tau))$\index{$(M,\sigma)'$} be the {\em topological dual}\index{Topological dual} of $(M,\sigma)$, which is a $\ring{R}$-submodule of $M^*$. Let $(M,\sigma)\xrightarrow{f}(N,\gamma)$ be a continuous $(\ring{R},\tau)$-linear map between topological modules. Let $(N,\gamma)'\xrightarrow{f'}(M,\sigma)'$ be the $\ring{R}$-linear map given by $f'(\ell):=\ell\circ f$. All of this evidently forms a functor $\mathbf{TopMod}^{\mathsf{op}}_{(\ring{R},\tau)}\xrightarrow{Top_{(\ring{R},\tau)}}\mathbf{Mod}_{\ring{R}}$\index{Top@$Top_{(\ring{R},\tau)}$}.

Let $(\ring{R},\tau)$ be a topological ring, and let $X$ be a set. Let $R^{(X)}\xrightarrow{\lambda_X}(R^{X})^*$\index{$\lambda_X$} be given by $(\lambda_X(p))(f):=\sum_{x\in X}p(x)f(x)$, $p\in R^{(X)}$, $f\in R^{X}$. 

Let $\rho_X$ be the map from Lemma~\ref{lem:rho_X_is_iso}. Then, for each $p\in R^{(X)}$, $\lambda_X(p)=\Lambda_{R^{(X)}}(p)\circ \rho_X$, which ensures continuity of $\lambda_X(p)$, i.e., $\lambda_X(p)\in ((R,\tau)^X)'$. Next lemma follows from  the equality $p(x)=(\lambda_X(p))(\delta_x)$, $p\in R^{(X)}$, $x\in X$. 
\begin{lemma}\label{lem:lambda_is_one_to_one}
$R^{(X)}\xrightarrow{\lambda_X}((R,\tau)^{X})'$ is one-to-one. 
\end{lemma}

\section{Rigid rings: definitions and (counter-)examples}

The notion  of {\em rigidity}, recalled at the beginning of the Introduction, was originally but only implicitly introduced in~\cite[Theorem~5, p.~156]{Poinsot_Rigidity} as the main result therein and the possibility that its conclusion could remain valid for more general topological rings than topological division rings was not noticed. Since a large part of this presentation is given for arbitrary rigid rings (Definition~\ref{def:rigidity} below), one here provides a stock of basic examples. 

%\subsection{Duality pairing}\label{subsec:duality_pairing_1}
%
%Let $\ring{R}$ be a ring, and let $X$ be a set. Let us define a $\ring{R}$-bilinear map, called the {\em duality pairing}\index{Duality pairing}, $R^{(X)}\times R^{X}\xrightarrow{\langle\cdot,\cdot\rangle_X}R$\index{$\langle\cdot,\cdot\rangle_X$}  by $$\langle p,f\rangle_X:=\sum_{x\in X}p(x)f(x)$$ (which is a sum with finitely many non-zero terms since $p$ is finitely-supported).
%
%This $\ring{R}$-bilinear map is {\em non-degenerate}~\cite[p.~155]{Poinsot_Rigidity}. 
%\begin{enumerate}
%\item The one-to-one linear map $R^{(X)}\xrightarrow{\lambda_X}((R,\tau)^{X})'$ (Lemma~\ref{lem:lambda_is_one_to_one}) is given by $(\lambda_X(p))(f)=\langle p,f\rangle_X$, which implies that $\langle\cdot,\cdot\rangle_X$ is separately continuous in its second variable, when its first variable is fixed.
%\item The isomorphism $\rho_X\colon (R,\tau)^X\simeq ((R^{(X)})^*,w^*_{(\ring{R},\tau)})$ (Lemma~\ref{lem:rho_X_is_iso}, p.~\pageref{lem:rho_X_is_iso}) is given by $(\rho_X(f))(p)=\langle p,f\rangle_X$.
%\end{enumerate}

%\subsection{Rigid rings: definition}

As~\cite[Lemma~13, p.~158]{Poinsot_Rigidity}, one has the following fundamental lemma.
\begin{lemma}\label{lem:summability_of_maps}
Let $(\ring{R},\tau)$ be a topological ring, and let $X$ be a set.   For each $f\in R^X$, $(f(x)\delta_x)_{x\in X}$ is summable in $(R,\tau)^X$ with sum $f$.
\end{lemma}
%\begin{proof}
%According to~\cite[Theorem~10.10, p.~76]{Warner} it suffices to check that for each $z\in X$, $(\pi_{z}(f(x)\delta_x))_{x\in X}$ is summable with sum $\pi_{z}(f)=f(z)$ in $(R,\tau)$. But since $\pi_{z}(f(x)\delta_x)=f(x)\delta_{x}(z)$ this is obvious. \end{proof}

\begin{definition}\label{def:rigidity}
Let $(\ring{R},\tau)$ be a topological  ring. It is said to be {\em rigid}\index{Rigid (ring or ring topology or field)} when for each set $X$, $R^{(X)}\xrightarrow{\lambda_X}((R,\tau)^{X})'$ is an isomorphism in $\mathbf{Mod}_R$, i.e., $\lambda_X$ is onto. In this situation, one sometimes also called {\em rigid} a ring topology $\tau$ such that $(\ring{R},\tau)$ is rigid. %By a {\em rigid field} $(\mathbb{k},\tau)$ is meant a field $\mathbb{k}$ with a ring topology $\tau$, which is rigid as a topological ring.
\end{definition}

\begin{lemma}\label{lem:image_of_lambda_X}
Let $\ell\in ((R,\tau)^{X})'$. $\ell\in im(\lambda_X)$ if, and only if, $\hat{\ell}\colon X\to R$\index{$\hat{\ell}$} given by $\hat{\ell}(x):=\ell(\delta_x)$, belongs to $R^{(X)}$. Moreover, $im(\lambda_X)\xrightarrow{\hat{(-)}}R^{(X)}$, $\ell\mapsto \hat{\ell}=\sum_{x\in X}\ell(\delta_x)\delta_x$, is the inverse of $\lambda_X$.
\end{lemma}

\subsection{Basic stock of examples}
%Several examples of rigid rings are now presented. 

%\subsubsection{A field with a ring topology}

%By a {\em field with a ring topology} is meant a topological ring $(\mathbb{k},\tau)$ with $\mathbb{k}$ a field. Whence in particular the multiplication is continuous, but the inverse map is not assumed continuous. 
%
Of course, the trivial ring is rigid (under the (in)discrete topology!). 

The first assertion of the following result is a slight generalization of the main theorem in~\cite{Poinsot_Rigidity}, which is precisely the second assertion below, since the proof of~\cite[Theorem~5, p.~156]{Poinsot_Rigidity} does not use continuity of the inversion.% in the base field. 
\begin{lemma}\label{lem:rigidity_of_field_with_ring_topology}
Let $(\mathbb{k},\tau)$ be a field  with a ring topology (see Definition~\ref{def:top_ring_top_field}).  Then, $(\mathbb{k},\tau)$ is rigid. In particular, any topological field\footnote{A {\em topological field} is a field with a ring topology $(\mathbb{k},\tau)$ such that the inversion $\alpha\mapsto\alpha^{-1}$ is continuous from $\mathbb{k}\setminus\{\, 0\,\}$ to itself with the subspace topology.} is rigid.
\end{lemma}

%\noindent\textbf{B. Discrete rings:} There are also rings that are not fields while rigid.% still they are rigid. The most obvious is given below.
\begin{lemma}\label{lem:rigidity_of_discrete_rings}
For each ring $\ring{R}$, the discretely topologized ring $(\ring{R},\mathsf{d})$ is rigid.
\end{lemma}
\begin{proof}
Let $\ell\in ((R,\mathsf{d})^X)'$. As a consequence of Lemma~\ref{lem:summability_of_maps},  $(\ell(\delta_x))_x $ is summable in $(R,\mathsf{d})$, with sum $\ell(1_{A_{\ring{R}}(X)})$. Since $\{\, 0\,\}$ is an open neighborhood of zero in $(R,\mathsf{d})$, $\ell(\delta_x)=0$ for all but finitely many $x\in X$ (\cite[Theorem~10.5, p.~73]{Warner}). The conclusion follows by Lemma~\ref{lem:image_of_lambda_X}. %or in other terms, $\lambda_X(p_{\ell})=\ell$, with $p_{\ell}(x):=\ell(\delta_x)$, $x\in X$ (since $(\lambda_X(p_{\ell}))(f)=\sum_{x\in X}p_{\ell}(x)f(x)=\sum_{x\in X}\ell(\delta_x)f(x)=\ell(\sum_{x\in X}f(x)\delta_x)=\ell(f)$). Whence $\lambda_X$ is onto, and therefore is a  $R$-linear isomorphism from $R^{(X)}$ onto $((R,\mathsf{d})^X)'$. 
\end{proof}

%\noindent\textbf{C. Normed algebras:} 
Every normed, complex or real,  commutative and unital algebra (e.g., Banach or  $C^*$-algebra) is rigid. 
\begin{lemma}\label{lem:rigidity_of_normed_algebra}
Let $\mathbb{k}=\mathbb{R},\mathbb{C}$. 
Let $(\mathsf{A},\|\cdot\|)$ be a commutative normed $\mathbb{k}$-algebra\footnote{In a normed algebra $(\ring{A},\|\cdot\|)$, unital or not, commutative or not,  the norm is assumed {\em sub-multiplicative}, i.e., $\|xy\|\leq \|x\|\|y\|$, which ensures that the multiplication of $\ring{A}$ is jointly continuous with respect to the topology induced by the norm.} with a unit. Then, as a topological ring under the topology induced by the norm, it is rigid.
\end{lemma}
\begin{proof}
Let $\tau_{\|-\|}$ be the topology on $A$ induced by the norm of $\ring{A}$, where $A$ is the underlying $\mathbb{k}$-vector space of $\ring{A}$. Let $X$ be a set. Let $\ell \in ((A,\tau_{\|-\|})^X)'$. Let $f\in A^X$ be given by $f(x)=\frac{1}{\|\ell(\delta_x)\|}1_{\ring{A}}$ if $x\in supp(\hat{\ell})$ and $f(x)=0$ for $x\not\in supp(\hat{\ell})$.  Since by Lemma~\ref{lem:summability_of_maps}, $(f(x)\delta_x)_{x\in X}$ is summable with sum $f$,  $(f(x)\ell(\delta_x))_{x\in X}$ is summable in $(A,\tau_{\|-\|})$ with sum $\ell(f)$. So according to~\cite[Theorem~10.5, p.~73]{Warner},  for  $1>\epsilon>0$, there exists a finite set $F_{\epsilon}\subseteq X$ such that $\|f(x)\ell(\delta_x)\|<\epsilon$ for all $x\in X\setminus F_{\epsilon}$.  But  $1=\|f(x)\ell(\delta_x)\|$ for all $x\in supp(\hat{\ell})$ so that $supp(\hat{\ell})$ is finite, and $\lambda_X$ is onto by Lemma~\ref{lem:image_of_lambda_X}. % then there are infinitely many $x$ such that $\|f(x)\ell(\delta_x)\|\geq \epsilon$, which contradicts our above assumption. Therefore, $\lambda_X(p_{\ell})=\ell$ with $p_{\ell}(x):=\ell(\delta_x)$. 
\end{proof}

%\end{itemize}
\subsection{A supplementary example: von Neumann regular rings}

%\subsubsection{Basic results and examples}
%\subsubsection{Rigidity of a  von Neumann regular ring with a ring topology}

A ring\footnote{Assumed commutative and unital  as in Section~\ref{conv:convention1}.}  is said to be  {\em von Neumann regular}\index{Von Neumann regular ring} if for each $x\in R$, there exists $y\in R$ such that $x=xyx$~\cite[Theorem~4.23, p.~65]{Lam}. %It is worth noticing that commutativity forces uniqueness of a ``weak inverse''. 

%\begin{lemma}\label{lem:unique}
Let us assume that $\ring{R}$ is a (commutative) von Neumann regular ring. For each $x\in R$, there is a unique $x^{\dagger}\in R$\index{$x^{\dagger}$}, called the {\em weak inverse} of $x$, such that $x=xx^{\dagger}x$ and $x^{\dagger}=x^{\dagger}xx^{\dagger}$.\footnote{Given $y\in R$ with $x=xyx$, then $z:=yxy$ meets the requirements to be a ``weak inverse'' of $x$, and if $y,z$ are two candidates, then one has  $z=z^2x=z^2x^2y=(x^2z)zy=xzy=(x^2y)zy=(x^2z)y^2=xy^2=y$.} 
%\end{lemma}
%\begin{proof}
%For each $x$, let $z$ such that $x=xzx$, then $x^{\dagger}:=zxz$  meets the requirement of the statement of the lemma. Uniqueness is easily checked. 
%\end{proof}

\begin{example}\label{ex:prod_fields}
A field is a von Neumann regular with $x^{\dagger}:=x^{-1}$, $x\not=0$, and $0^{\dagger}=0$. More generally, let $(\mathbb{k}_i)_{i\in I}$ be a  family of fields. Let $\ring{R}$ be a ring, and let $\jmath\colon \ring{R}\hookrightarrow\prod_{i\in I}\mathbb{k}_i$ be a one-to-one ring map. Assume that for each $x\in R$, $\jmath(x)^{\dagger}\in im(\jmath)$, where for $(x_i)_{i\in I}\in \prod_{i\in I}\mathbb{k}_i$,  $(x_i)_{i\in I}^{\dagger}:=(x_i^{\dagger})_{i\in I}$. Then, $\ring{R}$ is  von Neumann regular.
\end{example}

\begin{remark}\label{rem:xxstar_is_idempotent_and_nonzero_for_x_nonzero}
Let $\ring{R}$ be a von Neumann regular ring. For each $x\in R$, $x\not=0$ if, and only if, $xx^{\dagger}\not=0$. Moreover, $xx^{\dagger}$ belongs to the set $E(\ring{R})$\index{E@$E(\ring{R})$} of all idempotents ($e^2=e$) of $\ring{R}$.% since $xx^{\dagger}xx^{\dagger}=xx^{\dagger}$. 
\end{remark}

%One has a sufficient condition for rigidity.
\begin{proposition}\label{prop:the_result_about_com_von_Neumann_regular_rings}
Let $(\ring{R},\tau)$ be a topological ring such that $\ring{R}$ is von Neumann regular. If $0\not\in \overline{E(\ring{R})\setminus\{\, 0\,\}}$, then $(\ring{R},\tau)$ is rigid. In particular, if $E(\ring{R})$ is finite, then $(\ring{R},\tau)$ is rigid.
\end{proposition}

\begin{proof}
That the second assertion follows from the first is immediate. Let $X$ be a  set. Let us assume that $0\not\in \overline{E(\ring{R})\setminus\{\, 0\,\}}$. Let $V\in \mathfrak{V}_{(R,\tau)}(0)$\footnote{Given a topological space $(E,\tau)$ and $x\in E$, $\mathfrak{V}_{(E,\tau)}(x)$\index{V@$\mathfrak{V}_{(E,\tau)}(x)$} is the set of all neighborhoods of $x$.}  such that $V\cap (E(R)\setminus\{\, 0\,\})=\emptyset$.  Let $\ell\in ((R,\tau)^X)'$. Let $f\in R^X$ be given by $f(x):=\ell(\delta_x)^{\dagger}$ for each $x\in X$.  Since $(f(x)\ell(\delta_x))_{x\in X}$ is summable in $(R,\tau)$ with sum $\ell(f)$, by Cauchy's condition~\cite[Definition~10.3, p.~72]{Warner}, there exists a finite set $A_{f,V}\subseteq X$ such that for all $x\not\in A_{f,V}$, $f(x)\ell(\delta_x)\in V$. But for $x\in X$, $f(x)\ell(\delta_x)=\ell(\delta_x)^{\dagger}\ell(\delta_x)\in E(\ring{R})$. Whence, in view of Remark~\ref{rem:xxstar_is_idempotent_and_nonzero_for_x_nonzero},  for all but finitely many $x$'s, $f(x)\ell(\delta_x)=0$, i.e., $\ell(\delta_x)=0$.
%
%In particular, for each $f\in R^X$, $(f(x)\ell(\delta_x))_{x\in X}$ is summable in $(R,\tau)$ with sum $\ell(f)$. By Cauchy's condition~\cite[Definition~10.3, p.~72]{Warner}, there exists a finite set $A_{f,V}\subseteq X$ such that for all $x\not\in A_{f,V}$, $f(x)\ell(\delta_x)\in V$. Now, let us consider the case where $f$ is given by $f(x):=\ell(\delta_x)^{\dagger}$ for each $x\in X$. Then, for each $x\in X$, $f(x)\ell(\delta_x)=\ell(\delta_x)^{\dagger}\ell(\delta_x)\in E(\ring{R})$. Whence, in view of Remark~\ref{rem:xxstar_is_idempotent_and_nonzero_for_x_nonzero},  for all but finitely many $x$'s, $f(x)\ell(\delta_x)=0$, i.e., $\ell(\delta_x)=0$.
\end{proof}

\begin{remark}
Lemma~\ref{lem:rigidity_of_field_with_ring_topology} is a consequence of Proposition~\ref{prop:the_result_about_com_von_Neumann_regular_rings} since for a field $\mathbb{k}$, $E(\mathbb{k})=\{\, 0,1_{\mathbb{k}}\,\}$.
\end{remark}

%\noindent \textbf{A. An application: Boolean rings}
%
%A {\em Boolean ring}\index{Boolean ring} is a  ring\footnote{According to our Conventions~\ref{conv:convention1} (p.~\pageref{conv:convention1}) one excludes the trivial ring, even if it  is usually considered as a perfectly correct Boolean ring.} $\ring{B}$ in which $x^2=x$ for all $x\in B$. In other words, $E(\ring{B})=B$.  It is apparent that  any Boolean ring is von Neumann regular (with $x^{\dagger}=x$, for each $x$). 
%\begin{corollary}
%Let $(\ring{B},\tau)$ be a topological ring such that $\ring{B}$ is a Boolean ring. $0\not\in \overline{E(\ring{B})\setminus\{\, 0\,\}}$ if, and only if, $\tau=\mathsf{d}$. In other terms, a topological Boolean ring is rigid if, and only if, its topology is discrete.
%\end{corollary}
%\begin{proof}
%According to Lemma~\ref{lem:rigidity_of_discrete_rings} (p.~\pageref{lem:rigidity_of_discrete_rings}), $(\ring{B},\mathsf{d})$ is rigid.   
%
%Conversely, by contraposition, let us assume that $\tau\not=\mathsf{d}$. Thus $\{\, 0\,\}$ is not open in $\tau$, so that for each $V\in \mathfrak{V}_{(\ring{B},\mathsf{d})}(0)$, there exists $x\in B\setminus\{\, 0\,\}$ with $x\in V$. Therefore, $0\in \overline{{B}\setminus\{\, 0\,\}}=\overline{E(\ring{B})\setminus\{\, 0\,\}}$.
%\end{proof}
%
%\noindent\textbf{B. An application: the box topology}

Now, let $(E_i,\tau_i)_{i\in I}$ be a family of topological spaces. On $\prod_{i\in I}E_i$ is defined the {\em box topology}~\cite[p.~107]{Kelley}  a basis of open sets of which is given by the ``box'' $\prod_{i\in I}V_i$\index{Box topology}, where each $V_i\in\tau_i$, $i\in I$.  The product $\prod_{i\in I}E_i$ together with the box topology is denoted by $\bigsqcap_{i\in I}(E_i,\tau_i)$\index{$\bigsqcap_{i\in I}(E_i,\tau_i)$}. (This topology is Hausdorff as soon as all the $(E_i,\tau_i)$'s are.)

It is not difficult to see that given a family $(\ring{R}_i,\tau_i)_{i\in I}$ of topological rings, then $\bigsqcap_{i\in I}(\ring{R}_i,\tau_i)$ still is a topological ring (under component-wise operations). 

\begin{proposition}
Let $(\mathbb{k}_i)_{i\in I}$ be a family of fields, and for each $i\in I$, let $\tau_i$ be a ring topology on $\mathbb{k}_i$. Let $\ring{R}$ be a ring with a one-to-one ring map $\jmath\colon \ring{R}\hookrightarrow\prod_{i\in I}\mathbb{k}_i$. Let us assume that for each $x\in R$, $\jmath(x)^{\dagger}\in im(\jmath)$ ($(x_i)_i^{\dagger}$ as in Example~\ref{ex:prod_fields}). Let $\ring{R}$ be topologized with the subspace topology $\tau_\jmath$ inherited from $\bigsqcap_{i\in I}(\mathbb{k}_i,\tau_i)$. Then, $(\ring{R},\tau_\jmath)$ is rigid.
\end{proposition}
\begin{proof}
Naturally $(x_i)_{i\in I}\in E(\prod_{i\in I}\mathbb{k}_i)$ if, and only if, $x_i\in \{\, 0,1_{\mathbb{k}_i}\,\}$ for each $i\in I$. Now, for each $i\in I$, let $U_i$ be an open neighborhood of zero in $(\mathbb{k}_i,\tau_i)$ such that $1_{\mathbb{k}_i}\not\in U_i$. Then, $\prod_i U_i$ is an open neighborhood of zero in $\bigsqcap_{i\in I}(\mathbb{k}_i,\tau_i)$ whose only idempotent member is $0$. Therefore, $0\not\in \overline{E(\prod_{i\in I}\mathbb{k}_i)\setminus\{\, 0\,\}}$.

Under the assumptions of the statement, an application of Example~\ref{ex:prod_fields} states that $\ring{R}$ is a (commutative) von Neumann regular ring. It is also of course a topological ring under $\tau_\jmath$ (since $\jmath$ is a one-to-one ring map). It is also clear that $E(R)\simeq E(\jmath(R))\subseteq E(\prod_i\mathbb{k}_i)$. Furthermore, $\jmath(\overline{E(R)\setminus\{\, 0\,\}})= \overline{E(\jmath(R))\setminus\{\, 0\,\}}\cap \jmath(R)\subseteq\overline{E(\prod_i \mathbb{k}_i)\setminus\{\, 0\,\}}$, and thus $0\not\in \overline{E(R)\setminus\{\, 0\,\}}$ according to the above discussion. Therefore, by Proposition~\ref{prop:the_result_about_com_von_Neumann_regular_rings}, $(\ring{R},\tau_\jmath)$ is rigid.
%
%\in \mathfrak{V}_{\prod_{i\in I}^{\beta}(\mathbb{k}_i,\tau_i)}(0)$ and $E(R)\cap \prod_i U_i=\{\, 0\,\}$. 
\end{proof}

\subsection{A counter-example}\label{sec:counterexample}

%Let $(R,\tau)$ be a Hausdorff commutative ring with a unit. Let $(R,\tau)^{X}$ be the $X$-fold product of $(R,\tau)$ in the category ${}_{1,c}\mathbf{TopRing}$ of Hausdorff topoligical commutative ring with a unit. (The underlying ring of $(R,\tau)^{X}$ thus is $R^{X}$ under the component-wise operations, which is unital with unit the constant map $X\xrightarrow{1_{R^{X}}}1_R$, and with the product topology, i.e., the weakest topology that makes continuous all the projections $R^{X}\xrightarrow{\pi_x}(R,\tau)$, $x\in X$.)

Let $(\ring{R},\tau)$ be a topological ring, and let us consider the topological $(\ring{R},\tau)^{X}$-module $((R,\tau)^{X})^{X}$ for a given set $X$. To avoid confusion one denotes by $(R^X)^X\xrightarrow{\Pi_x} R^X$ the canonical projections.%, while one keeps $\pi_x$ for those of $R^X$. % (one here uses a capital greek letter for the projections, rather than the lowercase, to avoid confusion.) 

Let us define a linear map $(R^{X})^{X}\xrightarrow{\ell}(R,\tau)^{X}$ by setting $\ell(f)\colon x\mapsto (f(x))(x)$, $f\in (R^X)^X$. %$(\ell((f_x)_x))(y)=f_y(y)$, $(f_x)_x\in (R^X)^{X}$, $y\in X$. $\ell$ is readily $R^{X}$-linear. (Indeed, $(\ell(f(f_x)_x))(y)=(\ell((ff_x)_x))(y)=
%(ff_y)(y)=f(y)f_y(y)=(f(\ell((f_x)_x)))(y)$, $(\ell((f_x)_x+(g_x)_x))(y)=(\ell((f_x+g_x)_x))(y)=
%(f_y+g_y)(y)=f_y(y)+g_y(y)=(\ell((f_x)_x)+\ell((g_x)_x))(y)$, for all $f\in R^{X}$, $(f_x)_x,(g_x)_x\in (R^{X})^{X}$, $y\in X$.)
$\ell$ is continuous, and thus belongs to  $ (((R,\tau)^{X})^{X})'$, since for each $x\in X$, $\pi_x\circ \ell=\pi_x\circ \Pi_x$. %To see this, by definition of the topology on $(R,\tau)^{X}$ it suffices to check that for each $x\in X$, $(R^{X})^{X}\xrightarrow{\ell}R^{X}\xrightarrow{\pi_x}(R,\tau)$ is continuous. But given $f\in (R^{X})^{X}$, $\pi_x(\ell(f))=(f(x))(x)=\pi_x(\Pi_x(f))$. In other words, $\pi_x\circ \ell=\pi_x\circ \Pi_x$ which certainly is continuous. It follows that $\ell\in (((R,\tau)^{X})^{X})'$.
Now, for each $x\in X$, $(\ell(\delta_x^{\ring{R}^X}))(x)=\delta_x^{\ring{R}^X}(x)=1_{\ring{R}^X}$, so that $supp(\hat{\ell})=X$.  Consequently one obtains
\begin{proposition}
Let $(\ring{R},\tau)$ be topological ring, and let $X$ be a set. If $X$ is infinite, then $(\ring{R},\tau)^X$ is not rigid.
\end{proposition}
However the above negative result may balanced by the following.
\begin{proposition}
Let $(\ring{R},\tau)$ be a rigid ring. If $I$ is finite, then $(\ring{R},\tau)^I$ is  rigid too.
\end{proposition}
\begin{proof}
Let $(\ring{R},\tau)$ be a topological ring. For a set $I$, one recalls  that  $(\ring{R},\tau)^I$ is the underlying ring $(\ring{R},\tau)^I$ (Notation~\ref{notation:top_ring_of_functions}) of $A_{(\ring{R},\tau)}(I)$. Any topological $(\ring{R},\tau)^I$-module is also a topological $(\ring{R},\tau)$-module under restriction of scalars\footnote{Let $(\ring{R},\tau)$ and $(\ring{S},\sigma)$ be  topological rings, and let $(\ring{R},\tau)\xrightarrow{f}(\ring{S},\sigma)$ be a continuous ring map. It may be used to transform a topological $(\ring{S},\sigma)$-module into a topological $(\ring{R},\tau)$-module by restriction of scalars along $f$. In details, let $(M,\gamma)$ be a topological $(\ring{S},\sigma)$-module. There is a scalar action of $\ring{R}$ on  $M$ given by $\alpha\cdot v:=f(\alpha)v$, $\alpha\in R$, $v\in M$ (where by juxtaposition is denoted the scalar action $\ring{S}\times M\to M$). Furthermore this action is again continuous (by composition of continuous maps). Let $f^*(M,\gamma)$ be the topological $(\ring{R},\tau)$-module just obtained. At present let $(M,\gamma)\xrightarrow{g}(N,\pi)$ be a continuous $(\ring{S},\sigma)$-linear map.  $g$ is also $\ring{R}$-linear because of $g(\alpha\cdot v)=g(f(\alpha)v)=f(\alpha)g(v)=\alpha\cdot g(v)$, and thus provides a continuous $(\ring{R},\tau)$-linear map $f^*(M,\gamma)\xrightarrow{g}f^*(N,\pi)$. All this results in a functor $\mathbf{TopMod}_{(\ring{S},\sigma)}\xrightarrow{f^*}\mathbf{TopMod}_{(\ring{R},\tau)}$ of {\em restriction of scalars along $f$}.}   along the unit map $(\ring{R},\tau)\xrightarrow{\eta_I}(\ring{R},\tau)^I$, $\eta_I(1_{\ring{R}})=1_{\ring{R}^I}$, which of course is a ring map, and is continuous (because $\eta_I(\alpha)=m_{\ring{R}^I}(\eta_I(\alpha),1_{\ring{R}^I})$, $\alpha\in R$.) 

Let $X$ be a set, and let $\ell\in ((({R},\tau)^I)^X)'$, i.e., $(({R},\tau)^I)^X\xrightarrow{\ell}(R,\tau)^I$ is continuous and $(\ring{R},\tau)^I$-linear, and by restriction of scalar along $\eta_I$ it is also a continuous  $(\ring{R},\tau)$-linear. Therefore for each $i\in I$, $(({R},\tau)^I)^X\xrightarrow{\ell}(R,\tau)^I\xrightarrow{\pi_i}(R,\tau)$ belongs to the topological dual space of $((R,\tau)^I)^X$ seen as a $(\ring{R},\tau)$-module. 

Let us assume that $(\ring{R},\tau)$ is rigid. Then, by Lemma~\ref{lem:image_of_lambda_X}, $supp(\widehat{\pi_i\circ \ell})$ is finite for each $i\in I$. One also has $supp(\hat{\ell})=\bigcup_{i\in I}supp(\widehat{\pi_i\circ \ell})$, with $X\xrightarrow{\hat{\ell}}R^I$, $\hat{\ell}(x):=\ell(\delta_x^{\ring{R}^I})$, $x\in X$. Whence if $I$ is finite, then $supp(\hat{\ell})$ is finite too. 
\end{proof}

\section{Rigidity as an equivalence of categories}\label{sec:recasting}

The main result of this section is Theorem~\ref{thm:equiv_of_cats} which provides a translation of the rigidity condition on a topological ring into a dual equivalence between the category of free modules and that of topologically-free modules (see below), provided by the topological dual functor with equivalence inverse the (opposite of the) algebraic dual functor, with both functors conveniently co-restricted. The purpose of this section thus is to prove this result.
%
%the purpose of this section is to prove  consists to translate the rigidity condition (Definition~\ref{def:rigidity}) on a topological ring into a dual equivalence between some categories of (topological) modules. The equivalence of categories will be given by the algebraic dual functor with equivalence inverse the (opposite of the) topological dual functor (see Section~\ref{subsec:recollection}), with both functors conveniently co-restricted. %Therefore, a first step in this direction is to introduce the categories to prove equivalent, and actually just one, namely the category of topologically-free modules  (see below), since the other is the category of free modules. 
%
%\textcolor{red}{Mieux r\'ediger: Rigidity turns out to be an equivalence of categories. The categories involved are that of free modules and that of topologically-free modules now introduced.}

\paragraph{Topologically-free modules.}
%
%\subsubsection{Definition and characterization}
%\begin{definition}
Let $(\ring{R},\tau)$ be a topological ring. Let $(M,\sigma)$ be a topological $(R,\tau)$-module. It is said to be a {\em topologically-free $(\ring{R},\tau)$-module}\index{Topologically-free module} if $(M,\sigma)\simeq (R,\tau)^{X}$, in $\mathbf{TopMod}_{(\ring{R},\tau)}$, for some set $X$. Such topological modules span the full subcategory $\mathbf{TopFreeMod}_{(\ring{R},\tau)}$\index{TopFreeMod@$\mathbf{TopFreeMod}_{(\ring{R},\tau)}$} of $\mathbf{TopMod}_{(\ring{R},\tau)}$. For a field $(\mathbb{k},\tau)$  with a ring topology, one defines correspondingly the category $\mathbf{TopFreeVect}_{(\mathbb{k},\tau)}\hookrightarrow\mathbf{TopVect}_{(\mathbb{k},\tau)}$\index{TopFreeVect@$\mathbf{TopFreeVect}_{(\mathbb{k},\tau)}$} of {\em topologically-free $(\mathbb{k},\tau)$-vector spaces}\index{Topologically-free vector space}. 
%
%\end{definition}

\begin{remark}\label{rem:corestriction_of_the_power_functor}
The topological power functor $\mathbf{Set}^{\mathsf{op}}\xrightarrow{P_{(\ring{R},\tau)}}\mathbf{TopMod}_{(\ring{R},\tau)}$  factors as indicated below (the co-restriction obtained is also called $P_{(\ring{R},\tau)}$). 
\begin{equation}
\xymatrix@R=1pc{
\mathbf{Set}^{\mathsf{op}}\ar[r]^-{P_{(\ring{R},\tau)}}\ar[rd]_-{P_{(\ring{R},\tau)}}&\mathbf{TopMod}_{(\ring{R},\tau)}\\
&\mathbf{TopFreeMod}_{(\ring{R},\tau)}\ar@{^{(}->}[u]
}
\end{equation}
%The co-restriction obtained is still denoted $\mathbf{Set}^{\mathsf{op}}\xrightarrow{P_{(\ring{R},\tau)}}\mathbf{TopFreeMod}_{(\ring{R},\tau)}$.
\end{remark}

Topologically-free modules are characterized by the fact of possessing  ``topological bases'' (see Corollary~\ref{cor:top_basis} below) which makes easier a number of calculations and proofs, once such a basis is chosen. 
\begin{definition}\label{def:top_basis}
Let $(M,\sigma)$ be a topological $(\ring{R},\tau)$-module. Let $B\subseteq M$. It is said to be a {\em topological basis}\index{Topological basis} of $(M,\sigma)$ if the following hold.
\begin{enumerate}
\item For each $v\in M$, there exists a unique family $(b'(v))_{b\in B}$, with $b'(v)\in R$ for each $b\in B$, such that $(b'(v)b)_b$  is summable in $(M,\sigma)$ with sum $v$. $b'(v)$ is referred to as the {\em coefficient} of $v$ at $b\in B$.
\item For each family $(\alpha_b)_{b\in B}$ of elements of $R$, there is a  member $v$ of $M$ such that $b'(v)=\alpha_b$, $b\in B$. (By the above point such $v$ is unique.)
\item $\sigma$ is equal to the initial topology induced by the {\em (topological) coefficient maps}\index{Topological coefficient maps} $(M\xrightarrow{b'}(R,\tau))_{b\in B}$\index{b2@$b'$}. (According to the two above points, each $b'$ is $\ring{R}$-linear.)
\end{enumerate}
\end{definition}

\begin{remark}\label{rem:top_coef_maps_as_delta}
It is an immediate consequence of the definition that for a topological basis $B$ of some topological module, $0\not\in B$ and  $b'(d)=\delta_{b}(d)$, $b,d\in B$  (since $\sum_{b\in B}\delta_b(d)b=d=\sum_{b\in B}b'(d)b$). In particular, $B\xrightarrow{(-)'}B':=\{\, b'\colon b\in B\,\}$\index{B2@$B'$} is a bijection.
\end{remark}

%The next two results are obvious.
\begin{lemma}\label{lem:change_of_top_basis}
Let $(M,\sigma)$ and $(N,\gamma)$ be isomorphic topological $(\ring{R},\tau)$-modules. Let $\Theta\colon (M,\sigma)\simeq (N,\gamma)$ be an isomorphism (in $\mathbf{TopMod}_{(\ring{R},\tau)}$). Let $B$ be a topological basis of $(M,\sigma)$. Then, $\Theta(B)=\{\, \Theta(b)\colon b\in B\,\}$ is a topological basis of $(N,\gamma)$. 
\end{lemma}
%\begin{proof}
%Let $w\in N$. Then, $\Theta^{-1}(w)$ is the sum of the summable family $(b'(\Theta^{-1}(w))b)_{b\in B}$. Therefore $w=\Theta(\Theta^{-1}(w))$ is the sum of the summable family $(b'(\Theta^{-1}(w))\Theta(b))_{b\in B}$. Uniqueness of the family is obvious. In particular, $\theta(b)'(w)=b'(\theta^{-1}(w))$. Now, let $(\alpha_b)_b$ be a family of scalars. Let $v\in M$ be given by $b'(v)=\alpha_b$, $b\in B$. Then, $(\Theta(b))'(\Theta(v))=b'(\Theta^{-1}(\Theta(v)))=b'(v)=\alpha_b$, $b\in B$. Finally by construction, $b'\circ\Theta^{-1}=\Theta(b)'$ for each $b\in B$. As $\Theta$ is an homeomorphism it is clear that $\sigma$ is the initial topology associated with the maps $(N\xrightarrow{\Theta(b)'}(R,\tau))_b$. 
%\end{proof}

\begin{corollary}\label{cor:top_basis}
Let $(M,\sigma)$ be a (Hausdorff) topological $(\ring{R},\tau)$-module. It admits a topological basis if, and only if, it is topologically-free.
\end{corollary}
%\begin{proof}
%Let us assume that $\Theta\colon (R,\tau)^X\simeq (M,\sigma)$ for some set $X$. Let $B:=\{\, \theta(\delta_x)\colon x\in X\,\}$. By Lemma~\ref{lem:change_of_top_basis}, since $(M,\sigma)$ is topologically-free, $B$ is a topological basis of $(M,\sigma)$. 
%
%The converse assertion is almost obvious.
%\end{proof}

\begin{example}\label{ex:top_basis_of_R_to_X}
Let $(\ring{R},\tau)$ be a topological ring. 
For each set $X$, $\{\, \delta_x\colon x\in X\,\}$ is a topological basis of $(R,\tau)^X$. Moreover $\pi_x=\delta_x'$, $x\in X$.
\end{example}

%\subsubsection{Properties of topological bases}\label{subsubsec:prop_of_top_bases}

Let us now take the time to establish a certain number of quite useful properties of topological bases.
\begin{lemma}\label{lem:top_basis_is_linearly_independent_and_its_span_is_dense}
Let $(M,\sigma)$ be a topologically-free $(\ring{R},\tau)$-module with topological basis $B$. Then, $B$ is $\ring{R}$-linearly independent and the linear span $\langle B\rangle$\index{$\langle B\rangle$} of $B$ is dense in $(M,\sigma)$.
\end{lemma}
\begin{proof}
Concerning the assertion of independence, it suffices to note that $0$ may be written as $\sum_{b\in B}0b$, and conclude by the uniqueness of the decomposition in a topological basis.  
%Let $b_i$, $i=1,\cdots,n$ be pairwise distinct elements of $B$. Let $\alpha_i\in R$, $i=1,\cdots,n$. Let us assume that $\sum_{i=1}^{n}\alpha_i b_i=0$. Since $0=\sum_{b\in B}0 b$ is the sum of the summable family $(z_b b)_b$, with $z_b=0$, $b\in B$, and $0$ is also the sum of the summable family $(a_b b)_b$ with $a_{b_i}=\alpha_i$, $i=1,\cdots,n$, and $a_{b}=0$, $b\not=b_i$, $i=1,\cdots,n$, by uniqueness it follows that $\alpha_i=0$, $i=1,\cdots,n$. 
Let $u\in M$ and let $V:=\{\, v\in M\colon b'(v)\in U_b,\ b\in A\,\}\in \mathfrak{V}_{(M,\sigma)}(0)$, where $A$ is a finite subset of $B$ and $U_b\in \mathfrak{V}_{(\ring{R},\tau)}(0)$, $b\in A$. Let $\alpha_b\in U_b$, $b\in A$, and $v:=\sum_{b\in A}\alpha_b b-\sum_{b\in B\setminus A}b'(u)b\in V$. So $u+v\in \langle B\rangle$. Thus, $u+V$ meets $\langle B\rangle$ and $\langle B\rangle$ is dense in $(M,\sigma)$. %One has to check that $u+V$ meets $\langle B\rangle$. Let $u_A:=\sum_{b\in A}(b'(u)+b'(v))b\in \langle B\rangle$ for some $v=\sum_{b\in B}b'(v)b\in V$ (whence $b'(v)\in U_b$, $b\in A$). Then, $u_A-u=\sum_{b\in A}b'(v) b-\sum_{b\in B\setminus A}b'(u)b$ so that for each $b\in A$, $b'(u_A-u)_b=b'(v)\in U_b$, whence $u_A-u\in V$. 
\end{proof}

%As a consequence of Lemma~\ref{lem:top_basis_is_linearly_independent_and_its_span_is_dense} one immediately gets the following.
\begin{corollary}\label{cor:equality_of_conti_lin_map_on_top_basis}
Let $(M,\sigma)$  be a topologically-free $(\ring{R},\tau)$-module, and let $(N,\gamma)$ be a topological $(\ring{R},\tau)$-module.  Let $(M,\sigma)\xrightarrow{f,g}(N,\gamma)$ be two continuous $(\ring{R},\tau)$-linear maps. $f=g$ if, and only if, for any topological basis $B$ of $(M,\sigma)$, $f(b)=g(b)$ for each $b\in B$.
\end{corollary}

Topologically-free modules allow for the definition of changes of topological bases (see Proposition~\ref{prop:free_top_mod} for a more general construction).
\begin{lemma}\label{lem:top_iso_from_top_basis}
Let $(M,\sigma)$ and $(N,\gamma)$ be two topologically-free $(\ring{R},\tau)$-modules, and let $B,D$ be respective topological bases. Let $f\colon B\to D$ be a bijection. Then, there is a unique isomorphism $g$ in $\mathbf{TopMod}_{(\ring{R},\tau)}$ such that $g(b)=f(b)$, $b\in B$. %then the relation $\phi^{\top}(v):=\sum_{d\in D}v_{\phi(d)}d$, $v\in M$, defines an isomorphism in $\mathbf{TopMod}_{(R,\tau)}$ from $(N,\gamma)$ to $(M,\sigma)$. Moreover, there is a unique isomorphism $\theta\colon (M,\sigma)\to (N,\gamma)$ such that $\theta(b)=\phi^{-1}(b)$, namely $\phi^{\top}$.
\end{lemma}

\begin{proof}
The question of uniqueness is settled by Corollary~\ref{cor:equality_of_conti_lin_map_on_top_basis}. If such an isomorphism $g$ exists, then $g(v)=g\left (\sum_{b\in B}b'(v)b\right)=\sum_{b\in B}b'(v)g(b)=\sum_{b\in B}b'(v)f(b)=
\sum_{d\in D}(f^{-1}(d))'(v)d$, $v\in M$. One observes that $g$ as defined by the right hand-side of the last equality, is $\ring{R}$-linear, and it is also continuous since for each $d\in D$,  $d'\circ g=(f^{-1}(d))'$. %whence one only needs to establish continuity. Let $d\in D$. Then, $d'(g(v))=(f^{-1}(d))'(v)$, $v\in M$, so that $d'\circ g=(f^{-1}(d))'$ which ensures continuity. Finally, it is obvious from its definition that $g(b)=f(b)$, $b\in B$.
\end{proof}

\begin{lemma}\label{lem:dual_top_basis}
Let $M$ be a free module with basis $B$. Then, $(M^*,w^*_{(\ring{R},\tau)})$ is a topologically-free module with topological basis $B^*:=\{\, b^*\colon b\in B\,\}$ (see Remark~\ref{rem:coefficient_maps}). 
\end{lemma}
\begin{proof}
According to Lemma~\ref{lem:dual_of_free_is_top_free_mod}, $(M^*,w^*_{(\ring{R},\tau)})$ is a topologically-free module.  Let $M\xrightarrow{\theta_B}R^{(B)}$ be the isomorphism given by $\theta_B(b)=\delta_b$, $b\in B$. Thus, $\theta_B^*\colon (R^{(B)})^*\simeq M^*$, and $\theta_B^*\circ\rho_B\colon R^B\simeq (R^{(B)})^*\simeq M^*$ is given by $\theta_B^*(\rho_B(\delta^{\ring{R}}_{b}))=\rho_B(\delta^{\ring{R}}_{b})\circ\theta_B=p_b\circ \theta_B=b^*$ for $b\in B$ (see Example~\ref{ex:p_x} for the definition of $p_b$). Now, $\{\, \delta_b\colon b\in B\,\}$ being a topological basis of $R^B$, by Lemma~\ref{lem:change_of_top_basis}, this shows that $B^*$ is a topological basis of $(M^*,w^*_{(\ring{R},\tau)})$. 
\end{proof}

\begin{example}\label{ex:top_basis_of_alg_dual_of_fin_supp_maps}
$\{\, p_x\colon x\in X\,\}$ is a topological basis of $(R^{(X)})^*$ (Example~\ref{ex:p_x}). %But, $\delta_x^*(p)=p(x)$, $x\in X$, $p\in R^{(X)}$. Therefore, $\delta_x^{*}=p_x$, $x\in X$.
\end{example}

\begin{remark}
If $B$ is a basis of a free module $M$, then  $B\simeq B^*$ under $b\mapsto b^*$, because for each $b,d\in B$, $b^*(d)=\delta_{b}(d)$. 
\end{remark}

\begin{corollary}\label{cor:free_to_top_free_mod}
Let $(\ring{R},\tau)$ be a topological ring. The algebraic dual functors $\mathbf{Mod}_{\ring{R}}^{\mathsf{op}}\xrightarrow{Alg_{(\ring{R},\tau)}}\mathbf{TopMod}_{(\ring{R},\tau)}$ factors as illustrated in the diagram below\footnote{When $\mathbb{k}$ is a field with a ring topology $\tau$, then one has the corresponding factorization of $\mathbf{Vect}_{\mathbb{k}}^{\mathsf{op}}\xrightarrow{Alg_{(\mathbb{k},\tau)}}\mathbf{TopVect}_{(\mathbb{k},\tau)}$. \begin{equation}
\xymatrix@R=1pc{
\mathbf{Vect}_{\mathbb{k}}^{\mathsf{op}}\ar[r]^{Alg_{(\mathbb{k},\tau)}} \ar[rd]& \mathbf{TopVect}_{(\mathbb{k},\tau)}\\
 &\mathbf{TopFreeVect}_{(\mathbb{k},\tau)}\ar@{^{(}->}[u]
}
\end{equation}}. Moreover the resulting co-restriction of $Alg_{(\ring{R},\tau)}$  (the bottom arrow of the diagram) is essentially surjective\footnote{A functor $\mathbf{C}\xrightarrow{F}\mathsf{D}$ is {\em essentially surjective} when each object $D$ in $\mathbf{D}$ is isomorphic to an object of the form $FC$, for some object $C$ in $\mathbf{C}$. Whence an equivalence of categories is a fully faithful and essentially surjective functor (see~\cite{MacLane}).}.  
\begin{equation}
\xymatrix@R=1pc{
\mathbf{Mod}_{\ring{R}}^{\mathsf{op}}\ar[r]^{Alg_{(\ring{R},\tau)}} & \mathbf{TopMod}_{(\ring{R},\tau)}\\
\mathbf{FreeMod}_{\ring{R}}^{\mathsf{op}} \ar@{^{(}->}[u]\ar[r] &\mathbf{TopFreeMod}_{(\ring{R},\tau)}\ar@{^{(}->}[u]
}
\end{equation}
\end{corollary}
\begin{proof}
The first assertion is merely Lemma~\ref{lem:dual_top_basis}. Regarding the second assertion, let $(M,\sigma)$ be a topologically-free module. So, for some set $X$, $(M,\sigma)\simeq (R,\tau)^X$. By Lemma~\ref{lem:rho_X_is_iso}, $(R,\tau)^X\simeq Alg_{(\ring{R},\tau)}(R^{(X)})$.
\end{proof}

%\begin{remark}\label{rem:free_to_top_free_mod}
%When $\mathbb{k}$ is a field with a ring topology $\tau$, then one has the corresponding factorization of $\mathbf{Vect}_{\mathbb{k}}^{\mathsf{op}}\xrightarrow{Alg_{(\mathbb{k},\tau)}}\mathbf{TopVect}_{(\mathbb{k},\tau)}$. \begin{equation}
%\xymatrix@R=1pc{
%\mathbf{Vect}_{\mathbb{k}}^{\mathsf{op}}\ar[r]^{Alg_{(\mathbb{k},\tau)}} \ar[rd]& \mathbf{TopVect}_{(\mathbb{k},\tau)}\\
% &\mathbf{TopFreeVect}_{(\mathbb{k},\tau)}\ar@{^{(}->}[u]
%}
%\end{equation}
%\end{remark}

%\begin{corollary}\label{cor:dual_top_basis}
%Let $M$ be a free module with basis $B$. Let $\ell\in M^*$. Then, for each $b\in B$, $\ell(b)=(b^*)'(\ell)$.
%\end{corollary}
%\begin{proof}
%Since by Lemma~\ref{lem:dual_top_basis}, $B^*$ is a topological basis of $(M^*,w^*_{(\ring{R},\tau)})$, $\ell=\sum_{b\in B}(b^*)'(\ell)b^*$ (sum of a summable family). According to the definition of the topology $w^*_{(\ring{R},\tau)}$ this implies that for each $v\in M$, $\ell(v)$ is the sum of the summable family $((b^*)'(\ell)(\Lambda_M(v))(b^*))_{b\in B}$ in $(R,\tau)$. Whence $\ell(v)=\sum_{b\in B}(b^*)'(\ell)b^*(v)$ (sum of a summable family) and in particular, $\ell(b)=\sum_{d\in B}(d^*)'(\ell)d^*(b)=(b^*)'(\ell)$, $b\in B$.
%\end{proof}

\begin{lemma}\label{lem:dual_lin_basis}
Let $(\ring{R},\tau)$ be a rigid ring. Let $(M,\sigma)$ be a topologically-free $(\ring{R},\tau)$-module with topological basis $B$. Then, $(M,\sigma)'$ is free with basis $B':=\{\, b'\colon b\in B\,\}$.
\end{lemma}
\begin{proof}
Let $\Theta_B\colon (M,\sigma)\simeq (R,\tau)^B$ be given by $\Theta_B(b)=\delta_b$. Therefore $\Theta_B'\colon ((R,\tau)^B)'\simeq (M,\sigma)'$, and thus one has an isomorphism $\Theta_B'\circ\lambda_B\colon R^{(B)}\simeq (M,\sigma)'$. Since a module isomorphic to a free module is free, $(M,\sigma)'$ is free. The previous isomorphism acts as:  
$\Theta_B'(\lambda_B(\delta_b))=\pi_b\circ\Theta_B=b'$ for $b\in B$. It follows from Lemma~\ref{lem:change_of_top_basis} that $B'$
 is a basis of $(M,\sigma)'$. %Now let $b\in B$ and let $v=\sum_{d\in B}d'(v) d\in M$ (sum of a summable family). By continuity and linearity, $\pi_{b}(\Theta_B(v))=\pi_{b}(\sum_{d\in B}d'(v)\delta_b)=b'(v)$.
 \end{proof}
 
\begin{example}
Let $(\ring{R},\tau)$ be a rigid ring. Let $(M,\sigma)=(R,\tau)^X$. By Example~\ref{ex:top_basis_of_R_to_X},  $\{\, \delta_x'\colon x\in X\,\}=\{\, \pi_x\colon x\in X\,\}$ is a linear basis of $((R,\tau)^X)'$.
\end{example}

\begin{corollary}\label{cor:top_dual_of_a_top_free_module}
Let $(\ring{R},\tau)$ be a rigid ring. The functor 
$\mathbf{TopMod}_{(\ring{R},\tau)}^{\mathsf{op}}\xrightarrow{Top_{(\ring{R},\tau)}}\mathbf{Mod}_{\ring{R}}$ factors\footnote{Correspondingly for a field $(\mathbb{k},\tau)$ with a ring topology,
\begin{equation}
\xymatrix@R=1pc{
\mathbf{TopVect}_{(\mathbb{k},\tau)}^{\mathsf{op}}\ar[r]^-{Top_{(\mathbb{k},\tau)}}&\mathbf{Vect}_{\mathbb{k}}\\
\ar@{^{(}->}[u]\mathbf{TopFreeVect}_{(\mathbb{k},\tau)}^{\mathsf{op}} \ar[ru]&
}
\end{equation}} as indicated by the diagram below.
\begin{equation}
\xymatrix@R=1pc{
\mathbf{TopMod}_{(\ring{R},\tau)}^{\mathsf{op}}\ar[r]^-{Top_{(\ring{R},\tau)}}&\mathbf{Mod}_{\ring{R}}\\
\ar@{^{(}->}[u]\mathbf{TopFreeMod}_{(\ring{R},\tau)}^{\mathsf{op}} \ar[r]& \mathbf{FreeMod}_{\ring{R}}\ar@{^{(}->}[u]
}
\end{equation}
\end{corollary}

\paragraph{The topological dual of the algebraic dual of a free module.}

Let $(\ring{R},\tau)$ be a topological ring. 
Let $M$ be a $\ring{R}$-module, and let us consider as in Section~\ref{subsec:algebraic_dual_functor}, the $\ring{R}$-linear map
$$M\xrightarrow{\Lambda_M}(M^*,w^*_{(\ring{R},\tau)})'$$
$(\Lambda_M(v))(\ell)=\ell(v)$, $v\in M$, $\ell\in M^*$. 

%Our objective in this section is to prove that under the assumption that  $(\ring{R},\tau)$ is rigid, $\Lambda_M$ is an isomorphism, when $M$ is free, and furthermore that it is natural in $M$. This will provide the first half of the dual equivalence of categories mentioned in the Introduction between free and topologically-free modules.

\begin{lemma}
Let  $M$ be a projective $\ring{R}$-module. Then,  $\Lambda_M$ is one-to-one. This holds in particular when $M$ is a free $\ring{R}$-module.
\end{lemma}
\begin{proof}
Let us consider a dual basis for $M$, i.e., sets $B\subseteq M$ and $\{\, \ell_{e}\colon e\in B\,\}\subseteq M^*$,  such that for all $v\in M$, $\ell_e(v)=0$ for all but finitely many $\ell_e\in B^*$ and $v=\sum_{e\in B}\ell_e(v)e$ (\cite[p.~23]{Lam}). Let $v\in \ker \Lambda_M$, i.e., $(\Lambda_M(v))(\ell)=\ell(v)=0$ for each $\ell\in M^*$. Then, in particular, $\Lambda_M(v)(\ell_e)=\ell_e(v)=0$ for all $e\in B$, and thus $v=0$. 
\end{proof}

Let $(\ring{R},\tau)$ (resp. $(\mathbb{k},\tau)$) be a rigid ring (resp. field). Let us still denote by $\mathbf{FreeMod}_{\ring{R}}^{\mathsf{op}}\xrightarrow{Alg_{(\ring{R},\tau)}}\mathbf{TopFreeMod}_{(\ring{R},\tau)}$ (resp. $\mathbf{Vect}_{\mathbb{k}}^{\mathsf{op}}\xrightarrow{Alg_{(\mathbb{k},\tau)}}\mathbf{TopFreeVect}_{(\mathbb{k},\tau)}$) and by $\mathbf{TopFreeMod}_{(\ring{R},\tau)}^{\mathsf{op}}\xrightarrow{Top_{(\ring{R},\tau)}}\mathbf{FreeMod}_{\ring{R}}$ (resp. $\mathbf{TopFreeVect}_{(\mathbb{k},\tau)}^{\mathsf{op}}\xrightarrow{Top_{(\mathbb{k},\tau)}}\mathbf{Vect}_{\mathbb{k}}$) the functors provided by Corollaries~\ref{cor:free_to_top_free_mod} and~\ref{cor:top_dual_of_a_top_free_module}.  
\begin{proposition}\label{cor:first_nat_iso}
Let us assume that $(\ring{R},\tau)$ is rigid. 
$\Lambda:=(\Lambda_M)_M\colon id\Rightarrow Top_{(\ring{R},\tau)}\circ Alg_{(\ring{R},\tau)}^{\mathsf{op}}\colon \mathbf{FreeMod}_{\ring{R}}\to\mathbf{FreeMod}_{\ring{R}}$ is a natural isomorphism\footnote{A {\em natural isomorphism}  $\alpha\colon F\Rightarrow G\colon \mathbf{C}\to\mathbf{D}$ is a natural transformation the components of which are isomorphisms in $\mathbf{D}$. Two functors $F,G\colon \mathbf{C}\to\mathbf{D}$ are said to be {\em naturally isomorphic}, which is denoted $F\simeq G$, if there is a natural isomorphism $\alpha\colon F\Rightarrow G\colon \mathbf{C}\to\mathbf{D}$.}.
\end{proposition}
\begin{proof}
Naturality is clear. Let $(\ring{R},\tau)$ be a topological ring. Let $M$ be a free $\ring{R}$-module. For each free basis $X$ of $M$,  the following diagram commutes in $\mathbf{Mod}_{\ring{R}}$, where $M\xrightarrow{\theta_X}R^{(X)}$\index{$\theta_X$} is the canonical isomorphism given by $\theta_X(x)=\delta_x^{\ring{R}}$, $x\in X$. Consequently, when $(\ring{R},\tau)$ is rigid, then for each free $\ring{R}$-module $M$, $M\xrightarrow{\Lambda_M}(M^*,w^*_{(\ring{R},\tau)})'$ is an isomorphism. 
\begin{equation}\label{diag:lienentreleslambdas}
\xymatrix@R=0.5pc{
M\ar[r]^-{\Lambda_M} & (M^*,w_{(\ring{R},\tau)}^*)' \eq[rd]^{(\theta_X^*)'}& \\
&&((R^{(X)})^*,w^*_{(\ring{R},\tau)})'\\
\eq[uu]^{\theta_X}R^{(X)} \ar[r]_-{\lambda_X}& ((R,\tau)^X)'\eq[ru]_{\rho_X'}&
}
\end{equation}
\end{proof}
%Of course, one also has immediately the following.
\begin{corollary}\label{cor:first_nat_iso-for-fields}
Let us assume that $(\mathbb{k},\tau)$ is a field with a ring topology. Then, 
$\Lambda=(\Lambda_M)_M\colon id\Rightarrow Top_{(\mathbb{k},\tau)}\circ Alg_{(\mathbb{k},\tau)}^{\mathsf{op}}\colon \mathbf{Vect}_{\mathbb{k}}\to\mathbf{Vect}_{\mathbb{k}}$ is a natural isomorphism.
\end{corollary}

\paragraph{The algebraic dual of the topological dual of a topologically-free module.}

Let $(M,\sigma)$ be a topological $(\ring{R},\tau)$-module. Let us consider the $\ring{R}$-linear map $M\xrightarrow{\Gamma_{(M,\sigma)}}((M,\sigma)')^*$\index{$\Gamma_{(M,\sigma)}$} by setting $(\Gamma_{(M,\sigma)}(v))(\ell):=\ell(v)$.

%In order to obtain the other half of the dual equivalence between free and topologically-free modules it remains to prove that when $(\ring{R},\tau)$ is rigid and $(M,\sigma)$ topologically-free, $\Gamma_{(M,\sigma)}$ is an isomorphism, and is natural in $(M,\sigma)$.

\begin{proposition}\label{lem:second_nat_iso}
Let us assume that $(\ring{R},\tau)$ is a rigid ring. Then, $\Gamma\colon id\Rightarrow Alg_{(\ring{R},\tau)}\circ Top^{\mathsf{op}}_{(\ring{R},\tau)}\colon \mathbf{TopFreeMod}_{(\ring{R},\tau)}\to \mathbf{TopFreeMod}_{(\ring{R},\tau)}$ is a natural isomorphism, with $\Gamma:=(\Gamma_{(M,\sigma)})_{(M,\sigma)}$.
\end{proposition}
\begin{proof}
Naturality is clear. Let $\Theta\colon (M,\sigma)\simeq (R,\tau)^X$ be an isomorphism (in $\mathbf{TopMod}_{(\ring{R},\tau)}$). Since $(\ring{R},\tau)$ is rigid, $\lambda_X\colon R^{(X)}\simeq ((R,\tau)^X)'$ is an isomorphism. Therefore $R^{(X)}\xrightarrow{\lambda_X}((R,\tau)^X)'\xrightarrow{\Theta'}(M,\sigma)'$ is an isomorphism too in $\mathbf{Mod}_R$. In particular, $(M,\sigma)'$ is a free with basis $\{\, \Theta'(\lambda_X(\delta_x^{\ring{R}}))\colon x\in X\,\}$. By Lemma~\ref{lem:dual_of_free_is_top_free_mod}, the weak-$*$ topology on $((M,\sigma)')^*$ is the same as the initial topology given by the maps $((M,\sigma)')^*\xrightarrow{\Lambda_{(M,\sigma)'}(\pi_x\circ \Theta)}(R,\tau)$, $x\in X$, because $\Theta'(\lambda_X(\delta_x^{\ring{R}}))=\Theta'(\pi_x)=\pi_x\circ\Theta$. Therefore, $\Gamma_{(M,\sigma)}$ is continuous if, and only if, for each $x\in X$, $\Lambda_{(M,\sigma)'}(\pi_x\circ \Theta)\circ \Gamma_{(M,\sigma)}=\pi_x\circ\Theta$ is continuous.  Continuity of $\Gamma_{(M,\sigma)}$ thus is proved. 

That $\Gamma_{(M,\sigma)}$ is an isomorphism in $\mathbf{TopMod}_{(R,\tau)}$   follows from the commutativity of the diagram (in $\mathbf{TopMod}_{(R,\tau)}$) below (which may be checked by hand). 
\begin{equation}
\xymatrix@R=0.1pc{
&\ar[ld]_{\Theta}(M,\sigma) \ar[r]^{\Gamma_{(M,\sigma)}}& (((M,\sigma)')^*,w_{(\ring{R},\tau)}^*)\\
(R,\tau)^X\ar[rd]_{\rho_X} &&\\
&((R^{(X)})^*,w_{(\ring{R},\tau)}^*) \ar[r]_{(\lambda_X^{-1})^*}& ((((R,\tau)^X)')^*,w_{(\ring{R},\tau)}^*)\ar[uu]_{((\Theta')^*)^{-1}}
}
\end{equation}
%
%Let $(M,\sigma),(N,\gamma)$ be topologically-free $(\ring{R},\tau)$-modules. Let $(M,\sigma)\xrightarrow{f}(N,\gamma)$ be a continuous $(\ring{R},\tau)$-linear map. Let $v\in M$ and $\ell\in (N,\gamma)'$. One has 
%\begin{equation}
%\begin{array}{lll}
%((f')^*((\Gamma_{(M,\sigma)}(v)))(\ell)&=&(\Gamma_{(M,\sigma)}(v))(f'(\ell))\\
%&=&(\Gamma_{(M,\sigma)}(v))(\ell\circ f)\\
%&=&\ell(f(v))\\
%&=&(\Gamma_{(N,\gamma)}(f(v)))(\ell).
%\end{array}
%\end{equation}
%This is equivalent to the naturality of  $\Gamma^{(\ring{R},\tau)}$.
%\begin{equation}
%\xymatrix@R=0.8pc{
%(M,\sigma)\ar[d]_{f}\ar[r]^-{\Gamma_{(M,\sigma)}}&(((M,\sigma)')^*,w_{(\ring{R},\tau)}^*)\ar[d]^{(f')^*}\\
%(N,\gamma)\ar[r]_-{\Gamma_{(N,\gamma)}}&(((N,\gamma)')^*,w_{(\ring{R},\tau)}^*)
%}
%\end{equation}
\end{proof}

\begin{corollary}\label{cor:second_nat_iso}
Let us assume that $(\mathbb{k},\tau)$ is a field with a ring topology. $\Gamma\colon id\Rightarrow Alg_{(\mathbb{k},\tau)}\circ Top^{\mathsf{op}}_{(\mathbb{k},\tau)}\colon \mathbf{TopFreeVect}_{(\mathbb{k},\tau)}\to \mathbf{TopFreeVect}_{(\mathbb{k},\tau)}$ with $\Gamma:=(\Gamma_{(M,\sigma)})_{(M,\sigma)}$, is  a natural isomorphism.
\end{corollary}

%\subsection{For a rigid ring $(\ring{R},\tau)$, $\mathbf{FreeMod}_{\ring{R}}\simeq \mathbf{TopFreeMod}_{(\ring{R},\tau)}^{\mathsf{op}}$}

%The results of the two previous sections may be combined to provide the expected equivalence of categories. 
\paragraph{The equivalence and some of its immediate consequences.}

Collecting Proposition~\ref{cor:first_nat_iso} and Lemma~\ref{lem:second_nat_iso}, one immediately gets the following.
\begin{theorem}\label{thm:equiv_of_cats}
Let us assume that $(\ring{R},\tau)$ is rigid. $\mathbf{TopFreeMod}_{(\ring{R},\tau)}^{\mathsf{op}}\xrightarrow{Top_{(\ring{R},\tau)}}\mathbf{FreeMod}_{\ring{R}}$ is an equivalence of categories\footnote{\label{footnote:reflect_iso}$\mathbf{C}\simeq \mathbf{D}$ means that the categories $\mathbf{C}$ and $\mathbf{D}$ are {\em equivalent}, i.e., that there is an equivalence of categories $\mathbf{C}\xrightarrow{F}\mathbf{D}$. In this situation an {\em equivalent inverse} of $F$ is a functor $\mathbf{D}\xrightarrow{G}\mathbf{C}$ such that there are two natural isomorphisms 
$\eta\colon id_{\mathbf{C}}\Rightarrow G\circ F\colon \mathbf{C}\to\mathbf{C}$ and $\epsilon\colon F\circ G\Rightarrow id_{\mathbf{D}}\colon \mathbf{D}\to\mathbf{D}$.}, with equivalence inverse the functor $\mathbf{FreeMod}_{\ring{R}}\xrightarrow{Alg_{(\ring{R},\tau)}^{\mathsf{op}}}\mathbf{TopFreeMod}_{(\ring{R},\tau)}^{\mathsf{op}}$.
\end{theorem}

\begin{corollary}\label{cor:equiv_of_cats_0}
$\mathbf{TopFreeVect}_{(\mathbb{k},\tau)}^{\mathsf{op}}\xrightarrow{Top_{(\mathbb{k},\tau)}}\mathbf{Vect}_{\mathbb{k}}$ is an equivalence of categories, and $\mathbf{Vect}_{\mathbb{k}}\xrightarrow{Alg_{(\mathbb{k},\tau)}^{\mathsf{op}}}\mathbf{TopFreeVect}_{(\mathbb{k},\tau)}^{\mathsf{op}}$ is its equivalence inverse, whenever $(\mathbb{k},\tau)$ is a field with a ring topology.
\end{corollary}

\paragraph{Finite-dimensional vector spaces.}%\label{subsubsec:findimtvs}

Let $\mathbb{k}$ be a field. Let $(M,\sigma)$ be a topologically-free $(\mathbb{k},\mathsf{d})$-vector space with $M$ finite-dimensional. Then, $\sigma$ is the discrete topology on $M$. It follows that $(M,\sigma)'=M^*$, and the equivalence established in Corollary~\ref{cor:equiv_of_cats_0} coincides with the classical dual equivalence $\mathbf{FinDimVect}_{\mathbb{k}}\simeq \mathbf{FinDimVect}_{\mathbb{k}}^{\mathsf{op}}$ under the algebraic dual functor, where $\mathbf{FinDimVect}_{\mathbb{k}}$ is the category of finite-dimensional $\mathbb{k}$-vector spaces. 

\paragraph{Linearly compact vector spaces.}\label{subsubsec:linearly_compact_tvs}

Let $\mathbb{k}$ be a field. A topological $(\mathbb{k},\mathsf{d})$-vector space $(M,\sigma)$ is said to be a {\em linearly compact} $\mathbb{k}$-vector space\index{Linearly compact $\mathbb{k}$-vector space} when $(M,\sigma)\simeq (\mathbb{k},\mathsf{d})^X$ for some set $X$ (see~\cite[Proposition~24.4, p.~105]{Bergman}). The full subcategory $\mathbf{LCpVect}_{\mathbb{k}}$\index{LCpVect@$\mathbf{LCpVect}_{\mathbb{k}}$} of $\mathbf{TopVect}_{(\mathbb{k},\mathsf{d})}$ spanned by these spaces is equal to $\mathbf{TopFreeVect}_{(\mathbb{k},\mathsf{d})}$. 

\begin{corollary} (of Theorem~\ref{thm:equiv_of_cats})
Let $\ring{R}$ be a ring. For each rigid topologies $\tau,\sigma$ on $\ring{R}$, the categories $\mathbf{TopFreeMod}_{(\ring{R},\tau)}$ and $\mathbf{TopFreeMod}_{(\ring{R},\sigma)}$ are equivalent. Moreover, for each field $(\mathbb{k},\tau)$ with a ring topology, $\mathbf{TopFreeVect}_{(\mathbb{k},\tau)}$ is equivalent to $\mathbf{LCpVect}_{\mathbb{k}}$.
\end{corollary}
%\begin{proof}
%This follows from Theorem~\ref{thm:equiv_of_cats} because of transitivity of equivalence of categories (or more precisely, by composition of adjunctions, and in particular of adjoint equivalences).
%\end{proof}

In particular, one recovers the result of J.~Dieudonn\'e~\cite{Dieudo} that $\mathbf{Vect}_{\mathbb{k}}^{\mathsf{op}}\simeq \mathbf{LCpVect}_{\mathbb{k}}$. 

\paragraph{The universal property of $(R,\tau)^X$.}\label{sec:univ_property_of_Rtau_X}

For a ring $\ring{R}$, the forgetful functor $\mathbf{Mod}_{\ring{R}}\xrightarrow{|\cdot|}\mathbf{Set}$ (see Remark~\ref{rem:categorytheoretic_recasting_of_free_modules}) may be restricted as indicated in the following commutative diagram, and the restriction still is denoted $\mathbf{FreeMod}_{\ring{R}}\xrightarrow{|\cdot|}\mathbf{Set}$.
\begin{equation}
\xymatrix@R=0.8pc{
\mathbf{Mod}_{\ring{R}} \ar[r]^{|\cdot|}& \mathbf{Set}\\
\mathbf{FreeMod}_{\ring{R}}\ar@{^{(}->}[u]\ar[ru]&
}
\end{equation}

Likewise $\mathbf{Set}\xrightarrow{F_{\ring{R}}}\mathbf{Mod}_{\ring{R}}$ (see again Remark~\ref{rem:categorytheoretic_recasting_of_free_modules}) may be co-restricted as indicated by the commutative diagram below, and the co-restriction is given the same name $\mathbf{Set}\xrightarrow{F_{\ring{R}}}\mathbf{FreeMod}_{\ring{R}}$. 
\begin{equation}
\xymatrix@R=0.8pc{
\mathbf{Set} \ar[r]^{F_{\ring{R}}}\ar[rd]& \mathbf{Mod}_{\ring{R}}\\
&\mathbf{FreeMod}_{\ring{R}}\ar@{^{(}->}[u]
}
\end{equation}
(Of course, when $\ring{R}$ is a field $\mathbb{k}$, there is no need to consider the corresponding co-restrictions.)

The adjunction $F_{\ring{R}}\dashv |\cdot|\colon \mathbf{Set}\to\mathbf{Mod}_{\ring{R}}$ gives rise to a new one $F_{\ring{R}}\dashv |\cdot|\colon \mathbf{Set}\to \mathbf{FreeMod}_{\ring{R}}$~\cite[p.~147]{MacLane}, and by composition, for each rigid ring $(\ring{R},\tau)$, there is also the adjunction 
$Alg^{\mathsf{op}}_{(\ring{R},\tau)}\circ F_{\ring{R}}\dashv |\cdot|\circ Top_{(\ring{R},\tau)}\colon \mathbf{Set}\to \mathbf{TopFreeMod}_{(\ring{R},\tau)}^{\mathsf{op}}$. Since $(R,\tau)^{X}\simeq ((R^{(X)})^*,w^*_{(\ring{R},\tau)})$ (Lemma~\ref{lem:rho_X_is_iso}), this may be translated into a {\em universal property} of $(R,\tau)^{X}$, as explained below, which somehow legitimates the terminology {\em topologically-free}.

\begin{proposition}\label{prop:free_top_mod}
Let us assume that $(\ring{R},\tau)$ is rigid. 
Let $X$ be a set. For each topologically-free module $(M,\sigma)$ and any map $X\xrightarrow{f}|(M,\sigma)'|$, there is a unique continuous $(\ring{R},\tau)$-linear map $(M,\sigma)\xrightarrow{f^{\sharp}}(R,\tau)^{X}$ such that $|(f^\sharp)'|\circ |\lambda_X|\circ\delta^{\ring{R}}_X=f$ (recall that $\delta_X^{\ring{R}}(x)=\delta_x^{\ring{R}}$, $x\in X$).
\end{proposition}

\begin{proof}
There is a unique $\ring{R}$-linear map $R^{(X)}\xrightarrow{\tilde{f}}(M,\sigma)'$ such that $|\tilde{f}|\circ\delta^{\ring{R}}_X=f$. Let us define the continuous linear map $(M,\sigma)\xrightarrow{f^{\sharp}}(R,\tau)^X:=(M,\sigma)\xrightarrow{\Gamma_{(M,\sigma)}}(((M,\sigma)')^*,w^*_{(\ring{R},\tau)})\xrightarrow{(\tilde{f})^*}((R^{(X)})^*,w^*_{(\ring{R},\tau)})\xrightarrow{\rho_X^{-1}}(R,\tau)^X$. One has 
\begin{equation}
\begin{array}{lll}
|(f^{\sharp})'|\circ |\lambda_X|\circ\delta^{\ring{R}}_X&=&|\Gamma_{(M,\sigma)}'|\circ |((\tilde{f})^*)'|\circ|(\rho_X^{-1})'|\circ |\lambda_X|\circ \delta_X^{\ring{R}}\\
&=&|\Gamma_{(M,\sigma)}'|\circ |((\tilde{f})^*)'|\circ|\Lambda_{R^{(X)}}|\circ\delta_X^{\ring{R}}\\
&&\mbox{(because $(\rho^{-1}_X)'\circ \lambda_X=\Lambda_{R^{(X)}}$)}\\% Eq.~(\ref{eq:rho_X_prime}), p.~\pageref{eq:rho_X_prime})}\\
&=&|\Gamma_{(M,\sigma)}'|\circ |\Lambda_{(M,\sigma)'}|\circ |\tilde{f}|\circ\delta_X^{\ring{R}}\\
&&\mbox{(by naturality of $\Lambda$)}\\
&=&|\tilde{f}|\circ\delta_X^{\ring{R}}\\
&&\mbox{(triangular identities for an adjunction~\cite[p.~85]{MacLane})}\\%by Corollary~\ref{cor:equiv_of_cats-2})}\\
&=&f.
\end{array}
\end{equation}
It remains to check uniqueness of $f^{\sharp}$. Let $(M,\sigma)\xrightarrow{g}(R,\tau)^X$ be a continuous linear map such that $|g'|\circ |\lambda_X|\circ \delta^{\ring{R}}_X=f$. Then, $g'\circ\lambda_X=\tilde{f}$. Thus, $\lambda_X^*\circ (g')^*=\tilde{f}^*=\rho_X\circ f^{\sharp}\circ \Gamma^{-1}_{(M,\sigma)}$. So $\rho_X\circ\Gamma^{-1}_{(R,\tau)^X}\circ (g')^*=\rho_X\circ f^{\sharp}\circ\Gamma^{-1}_{(M,\sigma)}$ because $\Gamma_{(R,\tau)^X}=(\lambda_X^{-1})^*\circ \rho_X$ (by direct inspection), and thus $\Gamma^{-1}_{(R,\tau)^X}\circ (g')^*=f^{\sharp}\circ \Gamma^{-1}_{(M,\sigma)}$. Then, by naturality of $\Gamma^{-1}$, $g\circ \Gamma_{(M,\sigma)}^{-1}=f^{\sharp}\circ\Gamma^{-1}_{(M,\sigma)}$.
\end{proof}

\begin{corollary}\label{cor:P_is_the_free_top_free_fun}
Let $(\ring{R},\tau)$ be a rigid ring. 
$\mathbf{Set}\xrightarrow{P_{(\ring{R},\tau)}^{\mathsf{op}}}\mathbf{TopFreeMod}^{\mathsf{op}}_{(R,\tau)}$ is a left adjoint of $\mathbf{TopFreeMod}^{\mathsf{op}}_{(R,\tau)}\xrightarrow{Top_{(\ring{R},\tau)}}\mathbf{FreeMod}_R\xrightarrow{|-|}\mathbf{Set}$, and thus is naturally equivalent to $\mathbf{Set}\xrightarrow{Alg_{(\ring{R},\tau)}^{\mathsf{op}}\circ F_{\ring{R}}}\mathbf{TopFreeMod}_{(R,\tau)}^{\mathsf{op}}$.
\end{corollary}

\begin{proof}
A quick calculation shows that $P_{(\ring{R},\tau)}(f)=(|\lambda_Y|\circ \delta_Y^{\ring{R}}\circ f)^{\sharp}$ for a set-theoretic map $X\xrightarrow{f}Y$. The relation $f\mapsto (\lambda_Y\circ\delta_Y^{\ring{R}}\circ f)^{\sharp}$ provides a functor from $\mathbf{Set}^{\mathsf{op}}$ to $\mathbf{TopFreeMod}_{(\ring{R},\tau)}$ whose opposite is, by construction, a left adjoint of $|-|\circ Top_{(\ring{R},\tau)}$ (this is basically the content of Proposition~\ref{prop:free_top_mod}).
\end{proof}

\section{Tensor product of topologically-free modules}\label{sec:topprodtopfreemodules}

In this section most of the ingredients so far introduced and developed fit together so as to lift the equivalence  $\mathbf{FreeMod}\simeq\mathbf{TopFreeMod}^{\mathsf{op}}_{(\ring{R},\tau)}$  to a duality between some (suitably generalized) topological algebras and coalgebras. 
Let us briefly describe how this goal is achieved:
\begin{enumerate}
\item A topological tensor product $\ostar_{(\ring{R},\tau)}$ for topologically-free modules (over a rigid ring $(\ring{R},\tau)$) is provided by transport of the algebraic tensor product $\otimes_{\ring{R}}$ of $\ring{R}$-modules along the dual equivalence from Section~\ref{sec:recasting}. %$\mathbf{FreeMod}\simeq\mathbf{TopFreeMod}^{\mathsf{op}}_{(\ring{R},\tau)}$.
\item A topological basis (Definition~\ref{def:top_basis}) for $(M,\sigma)\ostar_{(\ring{R},\tau)}(N,\gamma)$ is described in terms of topological bases of $(M,\sigma),(N,\gamma)$. %The properties of topological bases from Section~\ref{subsubsec:prop_of_top_bases}, 
\item The aforementioned equivalence is proved to be compatible with $\otimes_{\ring{R}}$ and $\ostar_{(\ring{R},\tau)}$ (i.e., it is a monoidal equivalence).% thanks to the above accurate description of a topological basis for the topological tensor product. 
\item Accordingly, for category-theoretic reasons (see Appendix~\ref{appendix:moncat}), one obtains a dual equivalence between some topological algebras and coalgebras, still under the algebraic and topological dual functors.
%\item Finally, one discusses how it interacts with the standard duality between algebras and coalgebras, known as {\em finite duality}. 
\end{enumerate}

%{sec:monoidality_of_ostar}

\subsection{Algebraic tensor product of modules}\label{sec:algebraic_tens_prod}

The results recalled below are well-known (cf.~\cite{BouAlg} for instance) but they give us the opportunity to introduce some notations used hereafter. 

Let $\ring{R}$ be a ring. For each $\ring{R}$-modules $M,N$, one denotes by $M\otimes_{\ring{R}}N$\index{$\otimes_{\ring{R}}$} their (algebraic) tensor product, and by $M\times N\xrightarrow{\otimes}M\otimes_{\ring{R}}N$ the {\em universal}\footnote{In other words, each $\ring{R}$-bilinear map $M\times N\xrightarrow{f}P$ uniquely factors as 
\begin{equation}
\xymatrix@R=0.7pc{
M\times N\ar[r]^{f}\ar[d]_{\otimes} &P\\
M\otimes_{\ring{R}}P\ar[ru]_{\tilde{f}}&
}
\end{equation}
where $\tilde{f}$ is $\ring{R}$-linear.} $\ring{R}$-bilinear map. As usually the image of $(x,y)\in M\times N$ by $\otimes$ is denoted $x\otimes y$. 

Furthermore, with the following isomorphisms $(M\otimes_{\ring{R}}N)\otimes_{\ring{R}}P\xrightarrow{\alpha_{M,N,P}}M\otimes_{\ring{R}}(N\otimes_{\ring{R}}P)$, $\alpha_{M,N,P}((x\otimes y)\otimes z)=x\otimes (y\otimes z)$, $M\otimes_{\ring{R}}R\xrightarrow{\rho_{M}}M$, $\rho_M(x\otimes 1_{\ring{R}})=x$, $R\otimes_{\ring{R}}N\xrightarrow{\lambda_N}N$, $\lambda_N(1_{\ring{R}}\otimes y)=y$, $M\otimes_{\ring{R}}N\xrightarrow{\sigma_{M,N}}N\otimes_{\ring{R}}M$, $\sigma_{M,N}(x\otimes y)=y\otimes x$, $\mathbb{Mod}_{\ring{R}}=(\mathbf{Mod}_{\ring{R}},\otimes_{\ring{R}},R)$ is a symmetric monoidal category (\cite[p.~184]{MacLane}). Given two sets $X,Y$, $R^{(X)}\otimes_{\ring{R}}R^{(Y)}\simeq R^{(X\times Y)}$ under the unique $\ring{R}$-linear map $R^{(X)}\otimes_{\ring{R}} R^{(Y)}\xrightarrow{\Phi_{X,Y}}R^{(X\times Y)}$\index{$\Phi_{X,Y}$} which maps $f\otimes g$ to $\sum_{(x,y)\in X\times Y}f(x)g(y)\delta_{(x,y)}^{\ring{R}}$, with inverse $R^{(X\times Y)}\xrightarrow{\Psi_{X,Y}}R^{(X)}\otimes_{\ring{R}} R^{(Y)}$\index{$\Psi_{X,Y}$} the unique $\ring{R}$-linear whose values on the basis elements are given by $\Psi_{X,Y}(\delta_{(x,y)}^{\ring{R}})=\delta_x^{\ring{R}}\otimes_{\ring{R}}\delta_y^{\ring{R}}$. 

By functoriality, it is clear that given free modules $M,N$, $M\otimes_{\ring{R}} N$ is free too.
%\begin{proof}
%Let $X$ be a basis of $M$ and $Y$ be a basis of $N$. 
%$M\otimes_{\ring{R}} N\simeq R^{(X)}\otimes_{\ring{R}} R^{(Y)}\simeq R^{(X\times Y)}$ where the first isomorphism is due to functoriality of $\otimes_{\ring{R}}$, and where the second isomorphism is given as follows: let $R^{(X)}\times R^{(Y)}\xrightarrow{\phi_{X,Y}}R^{(X\times Y)}$, $\phi_{X,Y}(f,g)=\sum_{(x,y)\in X\times Y}f(x)g(y)\delta^{\ring{R}}_{(x,y)}$. It is $\ring{R}$-bilinear so there is a unique $\ring{R}$-linear map $R^{(X)}\otimes_{\ring{R}} R^{(Y)}\xrightarrow{\Phi_{X,Y}}R^{(X\times Y)}$, $\Phi_{X,Y}(f\otimes g)=\phi_{X,Y}(f,g)$\index{$\Phi_{X,Y}$}. Conversely, let $R^{(X\times Y)}\xrightarrow{\Psi_{X,Y}}R^{(X)}\otimes_{\ring{R}} R^{(Y)}$ be the unique $\ring{R}$-linear map given on the basis elements by $\Psi_{X,Y}(\delta^{\ring{R}}_{(x,y)})=\delta^{\ring{R}}_x\otimes\delta^{\ring{R}}_y$\index{$\Psi_{X,Y}$}. It is then easy to check that $\Phi_{X,Y}$ and $\Psi_{X,Y}$ are inverse one from the other.
%\end{proof}
As a consequence,  $\mathbb{FreeMod}_{\ring{R}}=(\mathbf{FreeMod}_{\ring{R}},\otimes_{\ring{R}},R)$\index{FreeMod@$\mathbb{FreeMod}_{\ring{R}}$} is a (symmetric) monoidal subcategory\footnote{By a {\em (symmetric) monoidal subcategory}\index{Monoidal subcategory} of a (symmetric) monoidal category $\mathbb{C}=(\mathbf{C},-\otimes-,I)$ we mean a subcategory $\mathbf{C}'$ of $\mathbf{C}$, closed under tensor products, containing $I$, and the coherence constraints of $\mathbb{C}$ between $\mathbf{C}'$-objects. (The last condition is automatically fulfilled when $\mathbf{C}'$ is a full subcategory.) The embedding $E_{\mathbf{C}'}$ of $\mathbf{C}'$ into $\mathbf{C}$ then is a strict monoidal functor $\mathbb{E}_{\mathbb{C}'}$ (see e.g.,~\cite[Def.~2, p.~4876]{PoinsotPorstunit}.} of $\mathbb{Mod}_{\ring{R}}$.%  (see again Appendix~\ref{appendix:moncat}).

\subsection{Topological tensor product of topologically-free modules}\label{subsec:top_tens_prod_of_top_free_mod}

We now wish to take advantage of the equivalence of categories $\mathbf{FreeMod}_{\ring{R}}\simeq \mathbf{TopFreeMod}_{(\ring{R},\tau)}^{\mathsf{op}}$ (Theorem~\ref{thm:equiv_of_cats}) for a rigid ring $(\ring{R},\tau)$, to introduce a topological tensor product of topologically-free modules. {\textbf{From here to the end of Section~\ref{subsec:top_tens_prod_of_top_free_mod}, $(\ring{R},\tau)$ denotes a rigid ring.}}

\paragraph{The bifunctor $\ostar_{(\ring{R},\tau)}$.}
Let $(M,\sigma),(N,\gamma)$ be two topologically-free $(\ring{R},\tau)$-modules. One defines their {\em topological tensor product over $(\ring{R},\tau)$}\index{Topological tensor product} as 
\begin{equation}
(M,\sigma)\ostar_{(\ring{R},\tau)}(N,\gamma):=Alg_{(\ring{R},\tau)}((M,\sigma)'\otimes_{\ring{R}} (N,\gamma)').
\end{equation}
One immediately observes that $(M,\sigma)\ostar_{(\ring{R},\tau)}(N,\gamma)$\index{$\ostar_{(\ring{R},\tau)}$} still is a topologically-free $(\ring{R},\tau)$-module as
%\begin{enumerate}
%\item 
$(M,\sigma)'$ and $(N,\gamma)'$ are free $\ring{R}$-modules (Lemma~\ref{lem:dual_lin_basis}),
%\item 
$(M,\sigma)'\otimes_{\ring{R}}(N,\gamma)'$ also is (Section~\ref{sec:algebraic_tens_prod}),
%\item 
and the algebraic dual of a free module is topologically-free (Lemma~\ref{lem:dual_top_basis}).
%\end{enumerate}
 
Actually, this definition is just the object component of a bifunctor\footnote{By a  {\em bifunctor} is meant a functor with domain a product of two categories.}  $$\mathbf{TopFreeMod}_{(\ring{R},\tau)}\times \mathbf{TopFreeMod}_{(\ring{R},\tau)}\xrightarrow{-\ostar_{(\ring{R},\tau)}-}\mathbf{TopFreeMod}_{(\ring{R},\tau)}$$ namely\footnote{For every categories $\mathbf{C},\mathbf{D}$, $(\mathbf{C}\times\mathbf{D})^{\mathsf{op}}=\mathbf{C}^{\mathsf{op}}\times\mathbf{D}^{\mathsf{op}}$.} $\mathbf{TopFreeMod}_{(\ring{R},\tau)}\times\mathbf{TopFreeMod}_{(\ring{R},\tau)}\xrightarrow{Top^{\mathsf{op}}_{(\ring{R},\tau)}\times Top^{\mathsf{op}}_{(\ring{R},\tau)}}\mathbf{Mod}_{\ring{R}}^{\mathsf{op}}\times\mathbf{Mod}_{\ring{R}}^{\mathsf{op}}\xrightarrow{\otimes_{\ring{R}}^{\mathsf{op}}}\mathbf{Mod}_{\ring{R}}^{\mathsf{op}}\xrightarrow{Alg_{(\ring{R},\tau)}}\mathbf{TopMod}_{(\ring{R},\tau)}$. 

\begin{remark}\label{rem:explicit_form_of_a_member_of_the_tensor_product}
Let $X,Y$ be sets. One has 
\begin{equation}
\begin{array}{lll}
(R,\tau)^{X}\ostar_{(\ring{R},\tau)}(R,\tau)^{Y}&=&((((R,\tau)^X)'\otimes_{\ring{R}}((R,\tau)^Y)')^*,w_{(\ring{R},\tau)}^{*})\\
&\simeq&
((R^{(X)}\otimes_{\ring{R}} R^{(Y)})^*,w^*_{(\ring{R},\tau)})\\
&&\mbox{(under $(\lambda_X\otimes_{\ring{R}}\lambda_Y)^*$)}\\
&\simeq& ((R^{(X\times Y)})^*,w^*_{(\ring{R},\tau)})\\
&&\mbox{(under $\Psi_{X,Y}^*$; see Section~\ref{sec:algebraic_tens_prod})}\\
&\simeq& (R,\tau)^{X\times Y}.\\
&&\mbox{(under $\rho^{-1}_{X\times Y}$)}
\end{array}
\end{equation}
\end{remark}

Given $f_i\in \mathbf{TopFreeMod}_{(\ring{R},\tau)}((M_i,\sigma_i),(N_i,\gamma_i))$, $i=1,2$, then $f_1\ostar_{(\ring{R},\tau)}f_2:=(M_1,\sigma_1)\ostar_{(\ring{R},\tau)}(M_2,\sigma_2)\xrightarrow{(f_1'\otimes_{\ring{R}} f_2')^*}(N_1,\gamma_1)\ostar_{(\ring{R},\tau)}(N_2,\gamma_2)$. 

In details, let $L\in ((M_1,\sigma_1)'\otimes_{\ring{R}}(M_2,\sigma_2)')^*$, $\ell_1\in (N_1,\gamma_1)'$ and  $\ell_2\in (N_2,\gamma_2)'$. Then, 
\begin{equation}\label{eq:explicit_form_of_tens_prod_of_lin_maps}
\begin{array}{lll}
((f_1\ostar_{(\ring{R},\tau)}f_2)(L))(\ell_1\otimes\ell_2)&=&(((f'\otimes_{\ring{R}}g')^*)(\ell))(\ell_1\otimes\ell_2)\\
&=&
L((f_1'\otimes_{\ring{R}} f_2')(\ell_1\otimes \ell_2))\\
%&=&\ell((f_1'(\ell_1))\otimes(f_2'(\ell_2)))\\
&=&
L((\ell_1\circ f_1)\otimes(\ell_2\circ f_2)).
\end{array}
\end{equation}

\paragraph{A topological basis of $(M,\sigma)\ostar_{(\ring{R},\tau)}(N,\gamma)$.}%\label{subsubsec:topbasis4ostar}
Our next goal will be to explicitly describe a topological basis (Definition~\ref{def:top_basis}) of $(M,\sigma)\ostar_{(\ring{R},\tau)}(N,\gamma)$ in terms of topological bases of $(M,\sigma)$ and $(N,\gamma)$. %Once this goal achieved, we will take advantage of the use of topological bases  in order to handle more appropriately the topological tensor products. 

\begin{definition}\label{def:can_inj_of_Mstar_otimes_Nstar_into_M_otimes_N_star}
Given a ring $\ring{S}$, for every $\ring{S}$-modules $M,N$, one has a natural $\ring{R}$-linear map $M^*\otimes_{\ring{S}} N^*\xrightarrow{\Theta_{M,N}}(M\otimes_{\ring{S}} M)^*$\index{$\Theta_{M,N}$} given by $(\Theta_{M,N}(\ell_1\otimes\ell_2))(u\otimes v)=\ell_1(u)\ell_2(v)$, $\ell_1\in M^*$, $\ell_2\in N^*$, $u\in M$ and $v\in N$. 
\end{definition}

Let $(M,\sigma),(N,\gamma)$ be two topologically-free $(\ring{R},\tau)$-modules. Let $u\in M$ and $v\in N$. Let us define\index{$\ostar$} 
\begin{equation}
u\ostar v:=\Theta_{(M,\sigma)',(N,\gamma)'}(\Gamma_{(M,\sigma)}(u)\otimes_{\ring{R}} \Gamma_{(N,\gamma)}(v))\in (M,\sigma)\otimes_{(\ring{R},\tau)}(N,\gamma).
\end{equation}
In details, given $\ell_1\in (M,\sigma)'$ and $\ell_2\in (N,\gamma)'$, $(u\ostar v)(\ell_1\otimes \ell_2)=\ell_1(u)\ell_2(v)$.

\begin{lemma}
Let $(M,\sigma)$ and $(N,\gamma)$ be both topologically-free $(\ring{R},\tau)$-modules, with respective topological bases $B,D$. The map $B\times D\xrightarrow{-\ostar-} (M,\sigma)\ostar_{(\ring{R},\tau)}(N,\gamma)$ given by $(b,d)\mapsto b\ostar d$, is one-to-one.
\end{lemma}
%\begin{proof}
%This almost directly follows from Remark~\ref{rem:top_coef_maps_as_delta} (p.~\pageref{rem:top_coef_maps_as_delta}).
%Let $b_1,b_2\in B$ and $d_1,d_2\in D$, and let us assume that $b_1\ostar d_1=b_2\ostar d_2$. So $1_{\ring{R}}=b_1'(b_1)d_1'(d_1)=(b_1\ostar d_1)(b_1'\otimes d_1')=(b_2\ostar d_2)(b_1'\otimes d_1')=b_1'(b_2)d_1'(d_2)=\delta_{b_1}(b_2)\delta_{d_1}(d_2)$  by Remark~\ref{rem:top_coef_maps_as_delta} (p.~\pageref{rem:top_coef_maps_as_delta}). Thus $b_1=b_2$ and $d_1=d_2$.
%\end{proof}

\begin{lemma}\label{lem:continuity_of_point_tensors}
Let $(M,\sigma)$ and $(N,\gamma)$ be both topologically-free $(\ring{R},\tau)$-modules. The map $M\times N\xrightarrow{-\ostar-}(M,\sigma)\ostar_{(\ring{R},\tau)}(N,\gamma)$ is $\ring{R}$-bilinear and separately continuous in both variable. Moreover, if $\tau=\mathsf{d}$, then $\ostar$ is even jointly continuous.
\end{lemma}
\begin{proof}
$\ring{R}$-bilinearity is clear. 
%Let $u,v\in M$, $w\in N$ and $\alpha\in R$. Let $\ell_1\in (M,\sigma)'$ and $\ell_2\in (N,\gamma)'$. Then, $((\alpha u+ v)\ostar w)(\ell_1\otimes \ell_2)=\ell_1(\alpha u+v)\ell_2(w)=\alpha\ell_1(u)\ell_2(w)+\ell_1(v)\ell_2(w)=(\alpha(u\ostar w)+(v\ostar w))(\ell_1\otimes\ell_2)$. The map thus is $\ring{R}$-bilinear.
Since $(M,\sigma)'\otimes_{\ring{R}}(N,\gamma)'$ is free on $\{\, x\otimes y\colon x\in X,\ y\in Y\,\}$, where $X$ (resp. $Y$) is a basis of $(M,\sigma)'$ (resp. $(N,\gamma)'$), by Lemma~\ref{lem:dual_of_free_is_top_free_mod}, the topology $w^*_{(\ring{R},\tau)}$ on $(M,\sigma)\ostar_{(\ring{R},\tau)}(N,\gamma)$ is the initial topology induced by $((M,\sigma)'\otimes_R ((N,\mu)')^*\xrightarrow{\Lambda_{(M,\sigma)'\otimes_{\ring{R}} (N,\gamma)'}(x\otimes y)}(R,\tau)$,  $x\in X$, $y\in Y$. 

Let $x\in X$, $y\in Y$, $u\in M$ and $v\in N$. Then, $\Lambda_{(M,\sigma)'\otimes_{\ring{R}} (N,\gamma)'}(x\otimes y)(u\ostar v)=(u\ostar v)(x\otimes y)=x(u)y(v)=m_{\ring{R}}(x(u),x(v))$, and this automatically guarantees separate continuity in each variable of $\ostar$. 

Let us assume that $\tau=\mathsf{d}$. According to the above general case, to see that $\ostar$ is continuous, by~\cite[Theorem~2.14, p.~17]{Warner}, it suffices to prove continuity at zero of $\ostar$. Let $A\subseteq X\times Y$ be a finite set, and for each $(x,y)\in A$, let $U_{(x,y)}$ be an open neighborhood of zero in $(R,\mathsf{d})$. Let $A_1:=\{\, x\in X\colon \exists y\in Y,\ (x,y)\in A\,\}$ and $A_2:=\{\, y\in Y\colon \exists x\in X,\ (x,y)\in A\,\}$. $A_1,A_2$ are both finite and $A\subseteq A_1\times A_2$. Let $u\in M$ such that $x(u)=0$ for all $x\in A_1$, and $v\in N$ such that $y(v)=0$ for all $y\in A_2$. Then, $(u\ostar v)(x\otimes y)=0\in U_{(x,y)}$ for all $(x,y)\in A_1\times A_2$.%, and in particular for all $(x,y)\in A$.
%
%By reasoning in the same way as in the proof of Lemma~\ref{lem:algebra_structure_on_fun_space_and_on_finsupp_maps} one concludes that $\ostar$ is continuous.
\end{proof}

\begin{remark}\label{rem:ostar_on_sum}
Let $X,Y$ be sets, and let $f\in R^X$, $g\in R^Y$. Since $\ostar$ is separately continuous by Lemma~\ref{lem:continuity_of_point_tensors},\footnote{The second equality in Eq.~(\ref{eq:explicit_computation_of_fostarg}) follows from the proof of~\cite[Theorem~10.15, p.~78]{Warner} which, by inspection, shows that the cited result still is valid more generally after the replacement of a jointly continuous bilinear map by a separately continuous bilinear map.}
\begin{equation}\label{eq:explicit_computation_of_fostarg}
\begin{array}{lll}
f\ostar g&=&\left (\sum_{x\in X}f(x)\delta_x^{\ring{R}}\right)\ostar \left (\sum_{y\in Y}g(y)\delta_y^{\ring{R}}\right)\\
&=&\sum_{(x,y)\in X\times Y}f(x)g(y)\delta_x^{\ring{R}}\ostar\delta_y^{\ring{R}}.\\
&&\mbox{(as a sum of a summable family)}
\end{array}
\end{equation}
For the same reason as above,  if $B$ is a topological basis of $(M,\sigma)$ and $D$ is a topological basis of $(N,\gamma)$, then $u\ostar v=\sum_{(b,d)\in B\times D}b'(u)d'(v)b\ostar d$, $u\in M$, $v\in N$. In particular, one observes  that $(b\ostar d)'(u\ostar v)=b'(u)d'(v)$, $b\in B$, $u\in M$, $d\in D$, and $v\in N$.
\end{remark}

\begin{proposition}\label{prop:top_basis_of_a_tensor_prod}
Let $(M,\sigma)$ and $(N,\gamma)$ be topologically-free $(\ring{R},\tau)$-modules, with respective topological bases $B,D$. Then, $(b\ostar d)_{(b,d)\in B\times D}$ is a topological basis of $(M,\sigma)\ostar_{(\ring{R},\tau)}(N,\gamma)$.
\end{proposition}
\begin{proof}
%Let us first prove that $\{\, (\delta^{\ring{R}}_x\ostar\delta^{\ring{R}}_y)_{(x,y)\in X\times Y}\colon (x,y)\in X\times Y\,\}$ is a topological basis of $(R,\tau)^X\ostar_{(\ring{R},\tau)}(R,\tau)^Y$. 
By virtue of Lemma~\ref{lem:change_of_top_basis} and Remark~\ref{rem:explicit_form_of_a_member_of_the_tensor_product}, $\{\, (\delta^{\ring{R}}_x\ostar\delta^{\ring{R}}_y)_{(x,y)\in X\times Y}\colon (x,y)\in X\times Y\,\}$  is a topological basis of  $(R,\tau)^X\ostar_{(\ring{R},\tau)}(R,\tau)^Y$  since one has $((\lambda_X^{-1}\otimes_{\ring{R}}\lambda_Y^{-1})^*(\Phi^*_{X,Y}(\rho_{X\times Y}(\delta_{(x,y)}^{\ring{R}}))))=\delta_x^{\ring{R}}\ostar\delta_y^{\ring{R}}$. Let us consider the isomorphism $\Theta_B\colon (M,\sigma)\simeq(R,\tau)^B$, $\Theta_B(b)=\delta_b$, $b\in B$ (resp. $\Theta_D\colon (N,\gamma)\simeq(R,\tau)^D$). By functoriality, $\Theta_{B}\ostar_{(\ring{R},\tau)}\Theta_D\colon (M,\sigma)\ostar_{(\ring{R},\tau)}(N,\gamma)\simeq (R,\tau)^{B}\ostar_{(\ring{R},\tau)}(R,\tau)^D$, and since  $(\theta_B\ostar_{(\ring{R},\tau)}\theta_{D})(b\ostar d)=\delta_b\ostar \delta_d$, $(b,d)\in B\times D$, $\{\, b\ostar d\colon (b,d)\in B\times D\,\}$ is a topological basis of $(M,\sigma)\ostar_{(\ring{R},\tau)}(N,\gamma)$.
\end{proof}

%Using Lemma~\ref{lem:top_basis_is_linearly_independent_and_its_span_is_dense} one obtains
\begin{corollary}\label{cor:top_basis_of_a_tensor_prod}
Let $(M,\sigma)$ and $(N,\gamma)$ be topologically-free $(\ring{R},\tau)$-modules. Then, $\{\, u\ostar v\colon u\in M,\ v\in N\,\}$ spans a dense subspace in $(M,\sigma)\ostar_{(\ring{R},\tau)}(N,\gamma)$.
\end{corollary}
\begin{proof}
It is clear in view of Lemma~\ref{lem:top_basis_is_linearly_independent_and_its_span_is_dense}.
\end{proof}

%\subsubsection{Application}

Let $(M_i,\sigma_i)$ and $(N_i,\gamma_i)$, $i=1,2$, be topologically-free $(\ring{R},\tau)$-modules. Let $(M_i,\sigma_i)\xrightarrow{f_i}(N_i,\gamma_i)$, $i=1,2$, be continuous $(\ring{R},\tau)$-linear maps. Let $(u,v)\in M_1\times M_2$. By Eq.~(\ref{eq:explicit_form_of_tens_prod_of_lin_maps}) it is clear that \begin{equation}\label{eq:otimesRtau_tensor}
(f_1\ostar_{(\ring{R},\tau)} f_2)(u\ostar v)=f_1(u)\ostar f_2(v).
\end{equation}
If $B$ and $D$ are topological bases of $(M_1,\sigma_1)$ and $(M_2,\sigma_2)$ respectively, $(f_1\ostar_{(\ring{R},\tau)} f_2)(u\ostar v)=\sum_{(b,d)\in B\times D}b'(u)d'(v)f_1(b)\ostar f_2(d)$ (in view of Remark~\ref{rem:ostar_on_sum}).
%
%Let us assume that $B$ is a topological basis of $(M_1,\sigma_1)$ and $D$ is a topological basis of $(M_2,\sigma_2)$. Given $t\in (M_1,\sigma_1)\ostar_{(\ring{R},\tau)}(M_2,\sigma_2)$, by  Proposition~\ref{prop:top_basis_of_a_tensor_prod}, $t=\sum_{(b,d)\in B\times D}(b\ostar d)'(t)b\ostar d$ (sum of a summable family) uniquely, and thus using Eq.~(\ref{eq:otimesRtau_tensor}), 
%
% \begin{equation}
% \begin{array}{lll}
% (f_1\ostar_{(\ring{R},\tau)}f_2)(t)&=&\displaystyle (f_1\ostar_{(\ring{R},\tau)}f_2)\left (\sum_{(b,d)\in B\times D}(b\ostar d)'(t)b\ostar d\right)\\
% &=&\displaystyle\sum_{(b,d)\in B\times D}(b\ostar d)'(t) f_1(b)\ostar f_2(b).\end{array}\end{equation}
%
%In particular, if $t=u\ostar v$, $u\in M_1$, $v\in M_2$, then in view of Remark~\ref{rem:ostar_on_sum}, 
%\begin{equation}
%f_1(u)\ostar f_2(v)=(f_1\ostar_{(\ring{R},\tau)}f_2)(u\ostar v)=\sum_{(b,d)\in B\times D}b'(u)d'(v)f_1(b)\ostar f_2(d).
%\end{equation}

\subsection{Monoidality of $\ostar_{(\ring{R},\tau)}$ and its direct consequences}\label{sec:monoidality_of_ostar}

%\subsubsection{Monoidality}

Since most of the proofs from this section mainly consist in rather tedious, but simple, inspections of commutativity of some diagrams, essentially by working with given topological or linear bases\footnote{E.g., associativity of $\ostar_{(\ring{R},\tau)}$ is given by the isomorphism $(b\ostar d)\ostar e\mapsto b\ostar (d\ostar e)$ on basis elements (Lemma~\ref{lem:top_iso_from_top_basis}).}, and because they did not provide much understanding, they are not included in the presentation. The reader is kindly invited to consult Appendix~\ref{appendix:moncat} where are summarized some notions and notations about monoidal category theory.

%Most of the proofs of the results of this section are postponed to Appendix~\ref{appendix:proofs} (p.~\pageref{appendix:proofs}) because they mainly consist in rather tedious, but simple, inspections of commutativity of some diagrams. Moreover one freely uses notations and freely applies results about monoidal categories joined together in Appendix~\ref{appendix:moncat} (p.~\pageref{appendix:moncat}). The reader familiar with standard notions from this topic (monoidal categories and functors, and monoids in monoidal categories) may skip the specific appendices, unless a thorough understanding  is wished. 

\begin{proposition}\label{prop:ostar_is_a_coherent_monoidal_functor}
Let $(\ring{R},\tau)$ be a rigid ring. $$\mathbb{TopFreeMod}_{(\ring{R},\tau)}:=(\mathbf{TopFreeMod}_{(\ring{R},\tau)},\ostar_{(\ring{R},\tau)},(R,\tau))$$\index{TopFreeMod2@$\mathbb{TopFreeMod}_{(\ring{R},\tau)}$} is a symmetric monoidal category\footnote{See~\cite{MacLane} for the definition of a symmetric monoidal category.}.
\end{proposition}
%\begin{proof}
%See Appendix~\ref{appendix:proofs:coherence} (p.~\pageref{appendix:proofs:coherence}).
%\end{proof}

\begin{corollary}
For each field $(\mathbb{k},\tau)$ with a ring topology, $$\mathbb{TopFreeVect}_{(\mathbb{k},\tau)}:=(\mathbf{TopFreeVect}_{(\mathbb{k},\tau)},\ostar_{(\mathbb{k},\tau)},(\mathbb{k},\tau))$$\index{TopFreeVect2@$\mathbb{TopFreeVect}_{(\mathbb{k},\tau)}$} is a symmetric monoidal category.
\end{corollary}

\begin{example}\label{ex:hadamard_monoid}
Let $(\ring{R},\tau)$ be a rigid ring. Let $X$ be a set. Let us define a commutative monoid $M_{(\ring{R},\tau)}(X):=((R,\tau)^{X},\mu_X,\eta_X)$ in $\mathbb{TopFreeMod}_{(\ring{R},\tau)}$ by $\mu_X(f\ostar g)=\sum_{x\in X}f(x)g(x)\delta_x^{\ring{R}}$, $f,g\in R^X$ (under $(R,\tau)^X\ostar_{(\ring{R},\tau)}(R,\tau)^X\simeq (R,\tau)^{X\times X}$ from Remark~\ref{rem:explicit_form_of_a_member_of_the_tensor_product}) and $\eta_X(1_{\ring{R}})=\sum_{x\in X}\delta_x^{\ring{R}}$\index{$\eta_X$}. This provides  a functor\index{M2@$M_{(\ring{R},\tau)}$} $\mathbf{Set}^{\mathsf{op}}\xrightarrow{M_{(\ring{R},\tau)}}{}_{c}\mathbf{Mon}(\mathbb{TopFreeMod}_{(\ring{R},\tau)})$. 
\end{example}
%
%\begin{remark}\label{rem:hadamard_monoid_is_functorial}
%The construction of $M_{(\ring{R},\tau)}(X)$ for a rigid ring $(\ring{R},\tau)$, provided in Example~\ref{ex:hadamard_monoid}, is the object component of a functor\index{M2@$M_{(\ring{R},\tau)}$} $\mathbf{Set}^{\mathsf{op}}\xrightarrow{M_{(\ring{R},\tau)}}{}_{c}\mathbf{Mon}(\mathbb{TopFreeMod}_{(\ring{R},\tau)})$. 
%\end{remark}

Let $(\ring{R},\tau)$ be a rigid ring.  For each free $\ring{R}$-modules $M,N$, let us define\index{$\Phi_{M,N}$} $\Phi_{M,N}:=(M^*,w^*_{(\ring{R},\tau)})\ostar_{(\ring{R},\tau)}(N^*,w^*_{(\ring{R},\tau)})=Alg_{(\ring{R},\tau)}((M^*,w^*_{(\ring{R},\tau)})'\otimes_{\ring{R}} (N^*,w^*_{(\ring{R},\tau)})')\xrightarrow{(\Lambda_M\otimes_{\ring{R}} \Lambda_N)^*}((M\otimes_{\ring{R}} N)^*,w^{*}_{(\ring{R},\tau)})$. According to Proposition~\ref{cor:first_nat_iso}, $\Phi_{M,N}$ is an isomorphism in $\mathbf{TopFreeMod}_{(\ring{R},\tau)}$. Naturality in $M,N$ is clear, so this provides a natural isomorphism 
\begin{equation}
\begin{array}{l}
\Phi\colon Alg_{(\ring{R},\tau)}(-)\ostar_{(\ring{R},\tau)}Alg_{(\ring{R},\tau)}(-)
\Rightarrow Alg_{(\ring{R},\tau)}(-\otimes_{\ring{R}}-)\colon\\
 \mathbf{FreeMod}_{\ring{R}}^{\mathsf{op}}\times\mathbf{FreeMod}_{\ring{R}}^{\mathsf{op}}\to \mathbf{TopFreeMod}_{(\ring{R},\tau)}.
\end{array}
\end{equation}
Furthermore, let us consider the isomorphism $(R,\tau)\xrightarrow{\phi}(R^*,w_{(\ring{R},\tau)}^*)$\index{$\phi$} given by $\phi(1_{\ring{R}}):=id_{R}$, with inverse $\phi^{-1}(\ell)=\ell(1_{\ring{R}})$.

Let $(M,\sigma),(N,\gamma)$ be topologically-free $(\ring{R},\tau)$-modules. One defines the map\index{$\Psi_{(M,\sigma),(N,\gamma)}$} $\Psi_{(M,\sigma),(N,\gamma)}:=(M,\sigma)'\otimes_{\ring{R}}(N,\gamma)'
\xrightarrow{\Lambda_{(M,\sigma)'\otimes_{\ring{R}}(N,\gamma)'}}(Alg_{(\ring{R},\tau)}((M,\sigma)'\otimes_{\ring{R}}(N,\gamma)'))'=((M,\sigma)\ostar_{(\ring{R},\tau)}(N,\gamma))'$. This gives rise to a natural isomorphism 
\begin{equation}
\begin{array}{l}
\Psi\colon Top_{(\ring{R},\tau)}(-)\otimes_{\ring{R}} Top_{(\ring{R},\tau)}(-)\Rightarrow Top_{(\ring{R},\tau)}(-\ostar_{(\ring{R},\tau)}-) \colon \\
\mathbf{TopFreeMod}_{(\ring{R},\tau)}^{\mathsf{op}}\times \mathbf{TopFreeMod}_{(\ring{R},\tau)}^{\mathsf{op}}\to\mathbf{FreeMod}_{\ring{R}}.\end{array}\end{equation}
Let also $R\xrightarrow{\psi}(R,\tau)'$\index{$\psi$} be given by $\psi(1_{\ring{R}})=id_{R}$ and $\psi^{-1}(\ell)=\ell(1_{\ring{R}})$.

\begin{theorem}\label{thm:monoidal_equivalence}
Let $(\ring{R},\tau)$ be a rigid ring. 
\begin{enumerate}
\item $\mathbb{Alg}_{(\ring{R},\tau)}=(Alg_{(\ring{R},\tau)},\Phi,\phi)\colon \mathbb{FreeMod}_{\ring{R}}^{\mathbb{op}}\to\mathbb{TopFreeMod}_{(\ring{R},\tau)}$\index{Alg2@$\mathbb{Alg}_{(\ring{R},\tau)}$} is a strong symmetric monoidal functor.
\item $\mathbb{Top}_{(\ring{R},\tau)}=(Top_{(\ring{R},\tau)},\Psi,\psi)$\index{Top2@$\mathbb{Top}_{(\ring{R},\tau)}$} is a strong symmetric monoidal functor from $\mathbb{TopFreeMod}^{\mathbb{op}}_{(\ring{R},\tau)}$ to $\mathbb{FreeMod}_{\ring{R}}$, so is $\mathbb{Top}_{(\ring{R},\tau)}^{\mathbb{d}}$ (Remark~\ref{rem:induced_functor_symmetric_case}, Appendix~\ref{appendix:moncat:mon_fun_and_ind_fun})  from $\mathbb{TopFreeMod}_{(\ring{R},\tau)}$ to $\mathbb{FreeMod}^{\mathbb{op}}_{\ring{R}}$. 
\item $\Lambda^{\mathsf{op}}\colon \mathbb{Top}_{(\ring{R},\tau)}^{\mathbb{d}}\circ \mathbb{Alg}_{(\ring{R},\tau)}\Rightarrow \mathbb{id}\colon \mathbb{FreeMod}_{\ring{R}}^{\mathbb{op}}\to\mathbb{FreeMod}_{\ring{R}}^{\mathbb{op}}$ is a monoidal isomorphism. 
\item $\Gamma\colon \mathbb{id}\Rightarrow \mathbb{Alg}_{(\ring{R},\tau)}\circ \mathbb{Top}^{\mathbb{d}}_{(\ring{R},\tau)}\colon \mathbb{TopFreeMod}_{(\ring{R},\tau)}\to\mathbb{TopFreeMod}_{(\ring{R},\tau)}$ is a monoidal isomorphism.
\end{enumerate}
In particular, $\mathbb{FreeMod}_{\ring{R}}^{\mathbb{op}}$ and $\mathbb{TopFreeMod}_{(\ring{R},\tau)}$ are monoidally equivalent (Appendix~\ref{appendix:moncat:mon_fun_and_ind_fun}).
\end{theorem}
%\begin{proof}
%See Appendix~\ref{appendix:proof:thm:monoidal_equivalence}, p.~\pageref{appendix:proof:thm:monoidal_equivalence}.
%\end{proof}

\begin{corollary}
For each field $(\mathbb{k},\tau)$ with a ring topology, the monoidal categories $\mathbb{Vect}^{\mathbb{op}}_{\mathbb{k}}$ and $\mathbb{TopFreeVect}_{(\mathbb{k},\tau)}$ are monoidally equivalent.
\end{corollary}

\begin{corollary}\label{cor:induced_equivalence_between_top_mon_and_coalgebras}
For each rigid ring $(\ring{R},\tau)$, the induced natural transformations (see Remark~\ref{rem:induced_nat_transfo} in Appendix~\ref{appendix:moncat:mon_fun_and_ind_fun} and Example~\ref{ex:cats_of_mon} in Appendix~\ref{appendix:mon_and_comon})
\begin{itemize}
\item $\widetilde{\Lambda^{\mathsf{op}}}\colon \widetilde{(\mathbb{Top}^{\mathbb{d}}_{(\ring{R},\tau)})}\circ \widetilde{\mathbb{Alg}_{(\ring{R},\tau)}}\Rightarrow id_{{}_{\epsilon}\mathbf{Coalg}_{\ring{R}}^{\mathsf{op}}}\colon {}_{\epsilon}\mathbf{Coalg}_{\ring{R}}^{\mathsf{op}}\to {}_{\epsilon}\mathbf{Coalg}_{\ring{R}}^{\mathsf{op}}$\index{$\widetilde{\Lambda^{\mathsf{op}}}$},
\item $\widetilde{\Lambda^{\mathsf{op}}}\colon \widetilde{(\mathbb{Top}^{\mathbb{d}}_{(\ring{R},\tau)})}\circ \widetilde{\mathbb{Alg}_{(\ring{R},\tau)}}\Rightarrow id_{{}_{\epsilon,coc}\mathbf{Coalg}_{\ring{R}}^{\mathsf{op}}}\colon {}_{\epsilon,coc}\mathbf{Coalg}_{\ring{R}}^{\mathsf{op}}\to {}_{\epsilon,coc}\mathbf{Coalg}_{\ring{R}}^{\mathsf{op}}$\index{Alg3@$\widetilde{\mathbb{Alg}_{(\ring{R},\tau)}}$},
\item $\tilde{\Gamma}\colon id_{\mathbf{Mon}(\mathbb{TopFreeMod}_{(\ring{R},\tau)})}\Rightarrow \widetilde{\mathbb{Alg}_{(\ring{R},\tau)}}\circ \widetilde{(\mathbb{Top}^{\mathbb{d}}_{(\ring{R},\tau)})}\colon\\ \mathbf{Mon}(\mathbb{TopFreeMod}_{(\ring{R},\tau)})\to \mathbf{Mon}(\mathbb{TopFreeMod}_{(\ring{R},\tau)})$\index{$\tilde{\Gamma}$},
\item $\tilde{\Gamma}\colon id_{{}_c\mathbf{Mon}(\mathbb{TopFreeMod}_{(\ring{R},\tau)})}\Rightarrow \widetilde{\mathbb{Alg}_{(\ring{R},\tau)}}\circ \widetilde{(\mathbb{Top}^{\mathbb{d}}_{(\ring{R},\tau)})}\colon\\ {}_c\mathbf{Mon}(\mathbb{TopFreeMod}_{(\ring{R},\tau)})\to {}_c\mathbf{Mon}(\mathbb{TopFreeMod}_{(\ring{R},\tau)})$\index{Top3@$\widetilde{(\mathbb{Top}^{\mathbb{d}}_{(\ring{R},\tau)})^{\mathbb{op}}}$}.
\end{itemize}
are all natural isomorphisms. 

Thus  ${}_{\epsilon}\mathbf{Coalg}_{\ring{R}}^{\mathsf{op}}$ (resp., ${}_{\epsilon,coc}\mathbf{Coalg}_{\ring{R}}^{\mathsf{op}}$) and $\mathbf{Mon}(\mathbb{TopFreeMod}_{(\ring{R},\tau)})$ (resp.,  ${}_c\mathbf{Mon}(\mathbb{TopFreeMod}_{(\ring{R},\tau)})$) are equivalent for each rigid topology $\tau$ on $\ring{R}$.

In particular,   $\mathbf{Mon}(\mathbb{TopFreeMod}_{(\ring{R},\tau)})\simeq \mathbf{Mon}(\mathbb{TopFreeMod}_{(\ring{R},\sigma)})$ (resp., ${}_c\mathbf{Mon}(\mathbb{TopFreeMod}_{(\ring{R},\tau)})\simeq {}_c\mathbf{Mon}(\mathbb{TopFreeMod}_{(\ring{R},\sigma)})$), for each rigid topologies $\tau,\sigma$ on $\ring{R}$.
\end{corollary}
\begin{proof}
Follows from Theorem~\ref{thm:monoidal_equivalence} together with Remarks~\ref{rem:induced_nat_transfo} and~\ref{rem:inv_of_mon_iso_is_mon_iso} in Appendix~\ref{appendix:moncat:mon_fun_and_ind_fun}.
\end{proof}

\begin{example}\label{ex:explicit_calculus_of_tilde_Gamma}
Let us make explicit the domain and codomain of the natural isomorphism $\tilde{\Gamma}$ from Corollary~\ref{cor:induced_equivalence_between_top_mon_and_coalgebras}. %one has an induced natural isomorphism $\tilde{\Gamma}\colon id_{\mathbf{Mon}(\mathbb{TopFreeMod}_{(\ring{R},\tau)})}\Rightarrow\widetilde{\mathbb{Alg}_{(\ring{R},\tau)}}\circ \widetilde{(\mathbb{Top}^{\mathbb{d}}_{(\ring{R},\tau)})^{\mathbb{op}}}\colon \mathbf{Mon}(\mathbb{TopFreeMod}_{(\ring{R},\tau)})\to \mathbf{Mon}(\mathbb{TopFreeMod}_{(\ring{R},\tau)})$ and we now intend to make explicit its domain and codomain. 

Let $((M,\sigma),m,e)$ be an object of $\mathbf{Mon}(\mathbb{TopFreeMod}_{(\ring{R},\tau)})$. Let its {\em topological dual coalgebra}\index{Topological dual coalgebra} be
\begin{equation}
\begin{array}{lll}
((M,\sigma)',\delta,\epsilon)&:=&\widetilde{(\mathbb{Top}^{\mathbb{d}}_{(\ring{R},\tau)})}((M,\sigma),m,e)\\
&=&((M,\sigma)',\Lambda^{-1}_{(M,\sigma)'\otimes_{\ring{R}}(M,\sigma)'}\circ m',\psi^{-1}\circ e').
\end{array}
\end{equation} In details, for $\ell\in (M,\sigma)'$, $\epsilon(\ell)=\psi^{-1}(e'(\ell))=\psi^{-1}(\ell\circ e)=\ell(e(1_{\ring{R}}))$,  and $\delta(\ell)=((\Lambda_{(M,\sigma)'\otimes_{\ring{R}}(M,\sigma)'}^{-1}\circ m')(\ell))=\sum_{i=1}^{n}\ell_i\otimes r_i$ for some $\ell_i,r_i\in (M,\sigma)'$. Thus, $\ell(m(u\ostar v))=\sum_{i=1}^{n}\ell_i(u)r_i(v)$, $u,v\in M$.

Now, $(((M,\sigma)')^*,M,E):=\widetilde{\mathbb{Alg}_{(\ring{R},\tau)}}((M,\sigma)',\delta,\epsilon)$ is given by $M:=\delta^*\circ (\Lambda_{(M,\sigma)'}\otimes_{\ring{R}}\Lambda_{(M,\sigma)'})^*$ and $E:=\epsilon^*\circ \phi$. Thus $E(1_{\ring{R}})=\epsilon^*(\phi(1_{\ring{R}}))=\epsilon^*(id_R)=\epsilon$, and given $L_1,L_2\in ((M,\sigma)')^*$, and $\ell\in (M,\sigma)'$, 
\begin{equation}
\begin{array}{lll}
(M(L_1\ostar L_2))(\ell)&=&(\delta^*((\Lambda_{(M,\sigma)'}\otimes_{\ring{R}}\Lambda_{(M,\sigma)'})^*(L_1\ostar L_2)))(\ell)\\
&=&((\Lambda_{(M,\sigma)'}\otimes_{\ring{R}}\Lambda_{(M,\sigma)'})^*(L_1\ostar L_2))(\delta(\ell))\\
&=&(L_1\ostar L_2)((\Lambda_{(M,\sigma)'}\otimes_{\ring{R}}\Lambda_{(M,\sigma)'})(\delta(\ell)))\\
&=&(L_1\ostar L_2)(\sum_{i=1}^n\Lambda_{(M,\sigma)'}(\ell_i)\otimes\Lambda_{(M,\sigma)'}(r_i))\\
&=&\sum_{i=1}^n\Lambda_{(M,\sigma)'}(\ell_i)(L_1)\Lambda_{(M,\sigma)'}(r_i)(L_2)\\
&=&\sum_{i=1}^n L_1(\ell_i)L_2(r_i).
\end{array}
\end{equation}
%Finally, $\tilde{\Gamma}_{((M,\sigma),m,e)}=\Gamma_{(M,\sigma)}\colon ((M,\sigma),m,e)\simeq (((M,\sigma)')^*,M,E)$ is a monoid isomorphism. 
\end{example}

\begin{corollary}
The equivalence from Corollary~\ref{cor:induced_equivalence_between_top_mon_and_coalgebras} restricts to an equivalence between the category ${}_{\epsilon}\mathbf{FinDimCoalg}_{\mathbb{k}}$ (resp. ${}_{\epsilon,coc}\mathbf{FinDimCoalg}_{\mathbb{k}}$) of finite-dimensional (resp. cocommutative) coalgebras and  $\mathbf{Mon}(\mathbb{FinDimVect}_{\mathbb{k}})$ (resp. ${}_c\mathbf{Mon}(\mathbb{FinDimVect}_{\mathbb{k}})$), where $\mathbb{FinDimVect}_{\mathbb{k}}=(\mathbf{FinDimVect}_{\mathbb{k}},\otimes_{\mathbb{k}},\mathbb{k})$. 
\end{corollary}

\section{Relationship with finite duality}\label{subsub:relation_with_finite_duality}

Over a field, there is a standard and well-known notion of duality between algebras and coalgebras, known as the finite duality~\cite{Abe,Dasca} and we have the intention to understand the relations if any, between the equivalence of categories from Corollary~\ref{cor:induced_equivalence_between_top_mon_and_coalgebras} and this finite duality. 

%Over a field, there is a nice duality between counital coalgebras and unital algebras that is not briefly reviewed. See~\cite{Dasca} for more details. 

%First of all, recall from Example~\ref{ex:cats_of_mon} (p.~\pageref{ex:cats_of_mon}, Appendix~\ref{appendix:mon_and_comon}) that one has a functorial isomorphism $O\colon \mathbf{Mon}(\mathbb{Mod}_{\ring{R}})\simeq {}_1\mathbf{Alg}_{\mathsf{R}}$ (concrete over $\mathbf{Mod}_{\ring{R}}$). In particular, by an {\em ideal}\index{Ideal of a monoid} of a monoid $\mathsf{A}$ in $\mathbf{Mon}(\mathbb{Mod}_{\ring{R}})$ is meant an ideal of the corresponding unital algebra $O(\mathsf{A})$.

Let $(-)^*\colon \mathbf{Mod}^{\mathsf{op}}_{\ring{R}}\to\mathbf{Mod}_{\ring{R}}$ be the usual algebraic dual functor. Then, $\mathbb{D}_{*}:=((-)^*,\Theta,\theta)$\index{D@$\mathbb{D}_{*}$} is a lax symmetric monoidal functor from $\mathbb{Mod}^{\mathbb{op}}_{\ring{R}}$ to $\mathbb{Mod}_{\ring{R}}$ (where $\Theta$ is as in Definition~\ref{def:can_inj_of_Mstar_otimes_Nstar_into_M_otimes_N_star}, and $\theta\colon {R}\to {R}^*$\index{$\theta$} is the isomorphism $\theta(1_{\mathsf{R}})=id_{\mathsf{R}}$ (and $\theta^{-1}(\ell)=\ell(1_{\mathsf{R}})$)). %According to Appendix~\ref{appendix:moncat:mon_fun_and_ind_fun} (p.~\pageref{appendix:moncat:mon_fun_and_ind_fun}), it induces a functor ${}_{\epsilon}\mathbf{Coalg}_{\ring{R}}^{\mathsf{op}}=\mathbf{Comon}(\mathbb{Mod}_{\ring{R}})^{\mathsf{op}}\xrightarrow{\widetilde{\mathbb{D}_{*}}}\mathbf{Mon}(\mathbb{Mod}_{\ring{R}})$.
%
%Secondly, let $\mathsf{C}$ be a counital $\ring{R}$-coalgebra. Then, $\widetilde{\mathbb{D}_{*}}(\mathsf{C}):=(C^*,\delta_{\mathsf{C}}^{*}\circ \Theta_{C,C},\epsilon_{\mathsf{C}}^{*}\circ \theta)$ is a monoid in $\mathbb{Mod}_{\ring{R}}$, where $\Theta$ is as in Definition~\ref{def:can_inj_of_Mstar_otimes_Nstar_into_M_otimes_N_star} (p.~\pageref{def:can_inj_of_Mstar_otimes_Nstar_into_M_otimes_N_star}), and $\theta\colon \ring{R}\to \ring{R}^*$\index{$\theta$} is the isomorphism $\theta(1_{\mathsf{R}})=id_{\mathsf{R}}$ (and $\theta^{-1}(\ell)=\ell(1_{\mathsf{R}})$). Moreover given a coalgebra map $\mathsf{C}\xrightarrow{f}\mathsf{D}$ one gets a morphism of monoids $\widetilde{\mathbb{D}_{*}}(\mathsf{D})\xrightarrow{f^*}\widetilde{\mathbb{D}_{*}}(\mathsf{C})$. 
%
%\begin{remark}
%Let $(-)^*\colon \mathbf{Mod}^{\mathsf{op}}_{\ring{R}}\to\mathbf{Mod}_{\ring{R}}$ be the usual algebraic dual functor. Then, $\mathbb{D}_{*}:=((-)^*,\Theta,\theta)$\index{D@$\mathbb{D}_{*}$} is a lax symmetric monoidal functor from $\mathbb{Mod}^{\mathbb{op}}_{\ring{R}}$ to $\mathbb{Mod}_{\ring{R}}$. According to Appendix~\ref{appendix:moncat:mon_fun_and_ind_fun} (p.~\pageref{appendix:moncat:mon_fun_and_ind_fun}), it induces a functor ${}_{\epsilon}\mathbf{Coalg}_{\ring{R}}^{\mathsf{op}}=\mathbf{Comon}(\mathbb{Mod}_{\ring{R}})^{\mathsf{op}}\xrightarrow{\widetilde{\mathbb{D}_{*}}}\mathbf{Mon}(\mathbb{Mod}_{\ring{R}})$.
%\end{remark}
When $\mathbb{k}$ is a field, there is  the {\em finite dual functor}\index{Dfin@$D_{fin}$} $D_{fin}\colon \mathbf{Mon}(\mathbb{Vect}_{\mathbb{k}})^{\mathsf{op}}\to {}_{\epsilon}\mathbf{Coalg}_{\mathbb{k}}$ (denoted by $(-)^{0}$ in~\cite{Abe,Dasca}).% defined as follows. Let $\mathsf{A}=(A,\mu_{\mathsf{A}},\eta_{\mathsf{A}})$ be a monoid in $\mathbb{Vect}_{\mathbb{k}}$. Let 
 The aforementioned finite duality is the adjunction $D_{fin}^{\mathsf{op}}\dashv \widetilde{\mathbb{D}_{*}}\colon {}_{\epsilon}\mathbf{Coalg}_{\mathbb{k}}\to \mathbf{Mon}(\mathbb{Vect}_{\mathbb{k}})$ (see e.g., \cite[Theorem~1.5.22, p.~44]{Dasca}, where $\widetilde{\mathbb{D}_{*}}$ is denoted by $(-)^*$). %{thm:finite_duality}

\subsection{The underlying algebra}\label{sec:the_underlying_algebra}

Let $(\ring{R},\tau)$ be a rigid ring. Let $(M,\sigma),(N,\gamma)$ be topologically-free $(\ring{R},\tau)$-modules. According to Lemma~\ref{lem:continuity_of_point_tensors}, $M\times N\xrightarrow{-\ostar-}(M,\sigma)\ostar_{(\ring{R},\tau)}(N,\gamma)$ is $\ring{R}$-bilinear. Denoting by $\mathbf{TopFreeMod}_{(\ring{R},\tau)}\xrightarrow{\|-\|}\mathbf{Mod}_{\ring{R}}$\index{$\mid\mid\cdot\mid\mid$} the canonical forgetful functor, this means that there is a unique $\ring{R}$-linear map 
$\|(M,\sigma)\|\otimes_{\ring{R}}\|(N,\gamma)\|\xrightarrow{\Xi_{(M,\sigma),(N,\gamma)}}\|(M,\sigma)\ostar_{(\ring{R},\tau)}(N,\gamma)\|$\index{$\Xi_{(M,\sigma),(N,\gamma)}$} such that for each $u\in M$, $v\in N$, $\Xi_{(M,\sigma),(N,\gamma)}(u\otimes v)=u\ostar v$. 

\begin{lemma}\label{lem:underlying_algebra}
$\mathbb{A}:=(\|-\|,(\Xi_{(M,\sigma),(N,\gamma)})_{(M,\sigma),(N,\gamma)},id_{R})$\index{A@$\mathbb{A}$} is a lax symmetric monoidal functor from $\mathbb{TopFreeMod}_{(\ring{R},\tau)}$ to $\mathbb{Mod}_{\ring{R}}$.
\end{lemma}
%\begin{proof}
%See Appendix~\ref{Appendix:proof_of_Lemma_underlying_algebra} (p.~\pageref{Appendix:proof_of_Lemma_underlying_algebra}).
%\end{proof}

Let $\tilde{\mathbb{A}}\colon\mathbf{Mon}(\mathbb{TopFreeMod}_{(\ring{R},\tau)})\to \mathbf{Mon}(\mathbb{Mod}_{\ring{R}})$\index{A2@$\tilde{\mathbb{A}}$} be the functor induced by $\mathbb{A}$ as introduced in Appendix~\ref{appendix:moncat:mon_fun_and_ind_fun}. Using the functorial isomorphism $O\colon \mathbf{Mon}(\mathbb{Mod}_{\ring{R}})\simeq {}_1\mathbf{Alg}_{\ring{R}}$ (Example~\ref{ex:cats_of_mon}, Appendix~\ref{appendix:mon_and_comon}), to any monoid in $\mathbb{TopFreeMod}_{(\ring{R},\tau)}$ is associated an ordinary algebra. 
\begin{definition}\label{def:underlying_algebra}
Let us define $UA:=\mathbf{Mon}(\mathbb{TopFreeMod}_{(\ring{R},\tau)})\xrightarrow{O\circ \tilde{\mathbb{A}}}{}_{1}\mathbf{Alg}_{\ring{R}}$\index{UA@$UA$}. 
Given a monoid $((M,\sigma),\mu,\eta)$ in $\mathbb{TopFreeMod}_{(\ring{R},\tau)}$, 
then $UA((M,\sigma),\mu,\eta)=O(\tilde{\mathbb{A}}((M,\sigma),\mu,\eta))$ is referred to as the {\em underlying (ordinary) algebra}\index{Underlying algebra} of the monoid $((M,\sigma),\mu,\eta)$. In details, $UA((M,\sigma),\mu,\eta)=(M,\mu_{bil},\eta(1_{\ring{R}}))$ with $\mu_{bil}\colon M\times M\to M$\index{$\mu_{bil}$}  given by $\mu_{bil}(u,v):=\mu(u\ostar v)$.
\end{definition}

%In details, let $((A,\sigma),\mu,\eta)$ be an object of $\mathbf{Mon}(\mathbb{TopFreeMod}_{(\ring{R},\tau)})$. Then, $UA((A,\sigma),\mu,\eta)=(A,\mu_{bil},\eta(1_{\ring{R}}))$ with $\mu_{bil}\colon A\times A\to A$\index{$\mu_{bil}$}  given by $\mu_{bil}(u,v):=\mu(u\ostar v)$.

\begin{remark}
Since by Lemma~\ref{lem:underlying_algebra}, $\mathbb{A}$ is symmetric, it induces a functor (see Remark~\ref{rem:induced_functor_symmetric_case},  Appendix~\ref{appendix:moncat:mon_fun_and_ind_fun}) ${}_{c}\mathbf{Mon}(\mathbb{TopFreeMod}_{(\ring{R},\tau)})\xrightarrow{\tilde{\mathbb{A}}}{}_{c}\mathbf{Mon}(\mathbb{Mod}_{\ring{R}})$. Because one has the co-restriction ${}_{c}\mathbf{Mon}(\mathbb{Mod}_{\ring{R}})\xrightarrow{O}{}_{1,c}\mathbf{Alg}_{\ring{R}}$, one may consider the {\em underlying algebra} functor $UA={}_c\mathbf{Mon}(\mathbb{TopFreeMod}_{(\ring{R},\tau)})\xrightarrow{O\circ \tilde{\mathbb{A}}}{}_{1,c}\mathbf{Alg}_{\ring{R}}$.
\end{remark}

\begin{example}\label{ex:hadamard_monoid_2} (Continuation of Example~\ref{ex:hadamard_monoid})
$UA(M_{(\ring{R},\tau)}(X))=A_{\ring{R}}X$. %(see Section~\ref{sec:ex:x-fold_product_and_fin_supp_maps}).
\end{example}

\subsection{Relations with $\widetilde{\mathbb{D}_{*}}$}\label{paragraph:relation_with_algebraic_dual}

Let $(\ring{R},\tau)$ be a rigid ring. 
Let $\mathbf{FreeMod}_{\ring{R}}\xrightarrow{E}\mathbf{Mod}_{\ring{R}}$ be the canonical embedding functor. Since $\mathbb{FreeMod}_{\ring{R}}$ is a symmetric monoidal subcategory of $\mathbb{Mod}_{\ring{R}}$ it follows that $\mathbb{E}=(E,id,id)$ is a strict monoidal functor from $\mathbb{FreeMod}_{\ring{R}}$ to $\mathbb{Mod}_{\ring{R}}$ (see Appendix~\ref{appendix:moncat:mon_fun_and_ind_fun}). 

One claims that $\mathbb{D}_{*}\circ \mathbb{E}^{\mathbb{op}}=\mathbb{A}\circ \mathbb{Alg}_{(\ring{R},\tau)}$. In particular, if $\mathbb{k}$ is a field (and $\tau$ is a ring topology on $\mathbb{k}$), then this reduces to $\mathbb{D}_{*}=\mathbb{A}\circ\mathbb{Alg}_{(\mathbb{k},\tau)}$. 

%\begin{enumerate}
%\item 
That $\|-\|\circ Alg_{(\ring{R},\tau)}=(-)^*\circ E^{\mathsf{op}}$ is due to the very definition of $Alg_{(\ring{R},\tau)}$.
%\item 
Of course, $\|\phi\|=\theta$. 
%\item 
That for each free modules $M,N$, $\|(\Lambda_M\otimes\Lambda_N)^*\|\circ\Xi_{M^*,N^*}=\Theta_{M,N}$ is easy to check. 
%Let $\ell_1\in M^*$, $\ell_2\in N^*$, $u\in M$ and $v\in N$. One has 
%\begin{equation}
%\begin{array}{ll}
%&((\Lambda_M\otimes_R \Lambda_N)^*(\Xi_{M^*,N^*}(\ell_1\otimes\ell_2)))(u\otimes v)\\
%=&
%(\ell_1\ostar \ell_2)(\Lambda_M(u)\otimes\Lambda_N(v))\\
%=&(\Lambda_M(u))(\ell_1)(\Lambda_N(v))(\ell_2)\\
%=&\ell_1(u)\ell_2(v)\\
%=&(\Theta_{M,N}(\ell_1\otimes\ell_2))(u\otimes v).
%\end{array}
%\end{equation}
%\end{enumerate}
So $((-)^*\circ E^{\mathsf{op}},\Theta,\theta)=(\|-\|,\Xi,id_R)\circ (Alg_{(\ring{R},\tau)},\Phi,\phi)
=(\|-\|\circ Alg_{(\ring{R},\tau)},(\|\Phi_{M,N}\|\circ\Xi_{M^*,N^*})_{M,N},\|\phi\|)$.

It follows that the finite dual monoid\footnote{In $\mathbb{Vect}_{\mathbb{k}}$.} $\widetilde{\mathbb{D}_{*}}(\ring{C})$ of a $\mathbb{k}$-coalgebra $\ring{C}$, is equal to $\widetilde{\mathbb{A}}(\widetilde{\mathbb{Alg}_{(\mathbb{k},\tau)}}(\ring{C}))$ whatever the ring topology $\tau$ on the field $\mathbb{k}$, and thus as ordinary algebras, $O(\widetilde{\mathbb{D}_*}(\ring{C}))=UA(\widetilde{\mathbb{Alg}_{(\mathbb{k},\tau)}}(\ring{C}))$. 

\begin{proposition}\label{prop:relation_with_D_star}
Let $(\mathbb{k},\tau)$ be a field with a ring topology. 
The functor $\mathbf{Mon}(\mathbb{FreeTopVect}_{(\mathbb{k},\tau)})\xrightarrow{O\circ \tilde{\mathbb{A}}} {}_1\mathbf{Alg}_{\mathbb{k}}$ has a left adjoint, namely $\widetilde{\mathbb{Alg}_{(\mathbb{k},\tau)}}\circ D_{fin}^{\mathsf{op}}\circ O^{-1}$.
\end{proposition}
\begin{proof}
One has\index{D2@$\widetilde{\mathbb{D}_{*}}$} $\widetilde{\mathbb{D}_{*}}=\tilde{\mathbb{A}}\circ\widetilde{\mathbb{Alg}_{(\mathbb{k},\tau)}}$, whence $\widetilde{\mathbb{D}_{*}}\circ \widetilde{(\mathbb{Top}_{(\mathbb{k},\tau)}^{\mathbb{d}})}=\tilde{\mathbb{A}}\circ\widetilde{\mathbb{Alg}_{(\mathbb{k},\tau)}}\circ \widetilde{(\mathbb{Top}_{(\mathbb{k},\tau)}^{\mathbb{d}})^{\mathbb{op}}}\simeq \tilde{\mathbb{A}}$ (natural isomorphism) by Theorem~\ref{thm:monoidal_equivalence}. Since $\widetilde{\mathbb{Alg}_{(\mathbb{k},\tau)}}\circ D_{fin}^{\mathsf{op}}$ is a left adjoint of $\widetilde{\mathbb{D}_{*}}\circ \widetilde{(\mathbb{Top}_{(\mathbb{k},\tau)}^{\mathbb{d}})}$, it follows that\footnote{Because if $F,G$ are two naturally isomorphic functors and  $L$ is a left adjoint of $F$, then it is also a left adjoint of $G$.} it is also the left adjoint of $\tilde{\mathbb{A}}$.  
\end{proof}

\subsection{The underlying topological algebra}\label{paragraph:def_of_TA}

Let $\ring{R}$ be a ring. %The underlying algebra functor $\mathbf{Mon}(\mathbb{TopFreeMod}_{(\ring{R},\mathsf{d})})\xrightarrow{UA}{}_{1}\mathbf{Alg}_{\ring{R}}$ ``lifts'' to a {\em topological algebra} functor\index{Topological algebra functor} $\mathbf{Mon}(\mathbb{TopFreeMod}_{(R,\mathsf{d})})\xrightarrow{TA}{}_{1}\mathbf{TopAlg}_{(\ring{R},\mathsf{d})}$\index{TA@$TA$}.
Let $\ring{A}=((A,\sigma),\mu,\eta)$ be an object of $\mathbf{Mon}(\mathbb{TopFreeMod}_{(\ring{R},\mathsf{d})})$. One knows that $(A,\sigma)$ is an object of $\mathbf{TopMod}_{(\ring{R},\mathsf{d})}$ and $UA(\ring{A})$ is an object of ${}_{1}\mathbf{Alg}_{\ring{R}}$. Moreover, 
$(A,\sigma)\times(A,\sigma)\xrightarrow{\mu_{bil}} (A,\sigma)$ is continuous,  since it is equal to the composition $(A,\sigma)\times(A,\sigma)\xrightarrow{-\ostar-}(A,\sigma)\ostar_{(\ring{R},\mathsf{d})}(A,\sigma)\xrightarrow{\mu}(A,\sigma)$ of continuous maps (see Lemma~\ref{lem:continuity_of_point_tensors}). Now, let $((A,\sigma),\mu,\eta)\xrightarrow{f}((B,\gamma),\nu,\zeta)$ be a morphism in $\mathbf{Mon}(\mathbb{TopFreeMod}_{(\ring{R},\mathsf{d})})$. In particular, $(A,\sigma)\xrightarrow{f}(B,\gamma)$ is linear and continuous, and the following diagram commutes. 
\begin{equation}
\xymatrix@C=0.8pc@R=0.7pc{
\ar[ddd]_{f\times f}(A,\sigma)\times (A,\sigma) \ar[rd]_{-\ostar-}\ar@/^1pc/[rrd]^{\mu_{bil}}& &\\
&(A,\sigma)\ostar_{(\ring{R},\mathsf{d})}(A,\sigma)\ar[d]_{f\ostar_{(\ring{R},\mathsf{d})}f} \ar[r]^-{\mu}& (A,\sigma)\ar[d]^{f}\\
&(B,\gamma)\ostar_{(\ring{R},\mathsf{d})}(B,\gamma)\ar[r]_-{\nu} & (B,\gamma)\\
(B,\gamma)\times (B,\gamma)\ar[ru]_{-\ostar -}\ar@/_1pc/[rru]_{\nu_{bil}}&&
}
\end{equation}
Since by assumption, one also has $f\circ\eta=\zeta$, it follows that $f(\eta(1_{\ring{R}}))=\zeta(1_{\ring{R}})$, and thus $f$ is a continuous algebra map from $((A,\sigma),\mu_{bil},\eta(1_{\ring{R}}))$ to $((B,\nu_{bil},\zeta(1_{\ring{R}}))$.) 

This provides a {\em topological algebra} functor\index{Topological algebra functor} $\mathbf{Mon}(\mathbb{TopFreeMod}_{(R,\mathsf{d})})\xrightarrow{TA}{}_{1}\mathbf{TopAlg}_{(\ring{R},\mathsf{d})}$\index{TA@$TA$} and the following diagram commutes (the unnamed arrows are either the obvious forgetful functors or the evident embedding functor), so that $TA$ is  concrete\footnote{A {\em concrete category}\index{Concrete category (or functor)} $\mathbf{C}$ over $\mathbf{D}$ is a pair $(\mathbf{C},\mathbf{C}\xrightarrow{U}\mathbf{D})$ with $U$ a faithful functor. Given concrete categories $(\mathbf{C}_i,U_i)$, $i=1,2$, over $\mathbf{D}$, by a {\em concrete functor} $(\mathbf{C}_1,U_1)\xrightarrow{F}(\mathbf{C}_2,U_2)$ is meant an ordinary functor $\mathbf{C}_1\xrightarrow{F}\mathbf{C}_2$ such that the following diagram commutes. Such a functor is necessarily faithful. 
\begin{equation}
\xymatrix@R=0.3pc{
\mathbf{C}_1 \ar[rr]^{F}\ar[rd]_{U_1}&&\mathbf{C}_2\ar[ld]^{U_2}\\
&\mathbf{D} &
}
\end{equation}} over $\mathbf{TopMod}_{(\ring{R},\mathsf{d})}$, whence faithful.
\begin{equation}
\xymatrix@R=1pc{
&\ar[d]\ar@/_1pc/[ldd]_{\tilde{\mathbb{A}}}\mathbf{Mon}(\mathbb{TopFreeMod}_{(\ring{R},\mathsf{d})}) \ar[r]^-{TA}&{}_{1}\mathbf{TopAlg}_{(\ring{R},\mathsf{d})}\ar[d]\ar@/^3pc/[dd]\\
&\ar[d]^{\|\cdot\|}\mathbf{TopFreeMod}_{(\ring{R},\mathsf{d})} \ar@{^{(}->}[r]&\ar[ld]\mathbf{TopMod}_{(\ring{R},\mathsf{d})}\\
\mathbf{Mon}(\mathbb{Mod}_{\ring{R}})\ar[r]\ar@/_1pc/[rr]_{O}&\mathbf{Mod}_{\ring{R}}&\ar[l]{}_{1}\mathbf{Alg}_{\ring{R}} &
}
\end{equation}

\begin{remark}\label{rem:TA_commutative_case}
When $\mathsf{A}$ is a commutative monoid in $\mathbb{TopFreeMod}_{(\ring{R},\mathsf{d})}$, then $TA(\mathsf{A})$ is a commutative topological algebra.
\end{remark}

\begin{example}\label{ex:hadamard_monoid_2bis} 
For each set $X$, $TA(M_{(\ring{R},\mathsf{d})}(X))=A_{(\ring{R},\mathsf{d})}(X)$.
\end{example}

\begin{proposition}\label{prop:TA_is_full}
$TA$ is a full embedding (injective on objects and faithful).
\end{proposition}
\begin{proof}
Let $\mathsf{A}=((A,\sigma),\mu,\eta)$ and $\mathsf{B}=((B,\gamma),\nu,\zeta)$ be two monoids in $\mathbb{TopFreeMod}_{(\ring{R},\mathsf{d})}$. 
Let $TA(\mathsf{A})\xrightarrow{g}TA(\mathsf{B})$ be a morphism in ${}_1\mathbf{TopAlg}_{(\ring{R},\mathsf{d})}$. In particular, $g\in {}_1\mathbf{Alg}_{\ring{R}}(UA(\mathsf{A}),UA(\mathsf{B}))\cap \mathbf{Top}((|A|,\sigma),(|B|,\gamma))$. (Recall from Remark~\ref{rem:categorytheoretic_recasting_of_free_modules}  that $\mathbf{Mod}_{\ring{R}}\xrightarrow{|\cdot|}\mathbf{Set}$ is the usual forgetful functor, and $\mathbf{Top}$ is the category of Hausdorff topological spaces.)

By assumption, for each $u,v\in A$, one has $g(\mu(u\ostar v))=g(\mu_{bil}(u,v))=\nu_{bil}(g(u),g(v))=\nu(g(u)\ostar g(v))$. Thus, $g\circ\mu=\nu\circ(g\ostar_{(\ring{R},\mathsf{d})}g)$ on $\{\, u\ostar v\colon u,v\in A\,\}$. Since this set spans a dense subset of $(A,\sigma)\otimes_{(R,\tau)}(A,\sigma)$ (according to Corollary~\ref{cor:top_basis_of_a_tensor_prod}), by linearity and continuity, $g\circ\mu=\nu\circ(g\ostar_{(\ring{R},\mathsf{d})}g)$ on the whole of $(A,\sigma)\ostar_{(\ring{R},\mathsf{d})}(A,\sigma)$. 

Moreover, $g(\eta(1_{\ring{R}}))=\zeta(1_{\ring{R}})$, then $g\circ \eta=\zeta$. Therefore, $g$ may be seen as a morphism $\mathsf{A}\xrightarrow{f}\mathsf{B}$ in $\mathbf{Mon}(\mathbb{TopFreeMod}_{(\ring{R},\mathsf{d})})$ with $TA(f)=g$, i.e., $TA$ is full.

Let $\mathsf{A}=((A,\sigma),\mu,\eta),\mathsf{B}=((B,\gamma),\nu,\zeta)$ be monoids in $\mathbb{TopFreeMod}_{(\ring{R},\mathsf{d})}$ such that $TA(\mathsf{A})=TA(\mathsf{B})$. In particular, $(A,\sigma)=(B,\gamma)$, and $\eta=\zeta$. By assumption $\mu_{bil}=\nu_{bil}$. Whence $\mu=\nu$ on $\{\, u\ostar v\colon u\in A,\ v\in B\,\}$, and by continuity they are equal on $(A,\sigma)\ostar_{(\ring{R},\mathsf{d})}(B,\gamma)$. So $\mathsf{A}=\mathsf{B}$, i.e., $TA$ is injective on objects. 
\end{proof}

%\begin{lemma}\label{lem:TA_inj_on_obj}
%$TA$ is injective on objects.
%\end{lemma}
%\begin{proof}
%Let $\mathsf{A}=((A,\sigma),\mu,\eta)$ and $\mathsf{B}=((B,\gamma),\nu,\zeta)$ be two monoids in $\mathbb{TopFreeMod}_{(\ring{R},\mathsf{d})}$ such that $TA(\mathsf{A})=TA(\mathsf{B})$. In particular, $(A,\sigma)=(B,\gamma)$, and $\eta=\zeta$. By assumption $\mu_{bil}=\nu_{bil}$. Whence $\mu=\nu$ on $\{\, u\ostar v\colon u\in A,\ v\in B\,\}$. Since this set is dense in $(A,\sigma)\ostar_{(\ring{R},\mathsf{d})}(B,\gamma)$ and $\mu,\nu$ are continuous, they are equal on the whole tensor product $(A,\sigma)\ostar_{(\ring{R},\mathsf{d})}(B,\gamma)$. So $\mathsf{A}=\mathsf{B}$. 
%\end{proof}

As a consequence of Proposition~\ref{prop:TA_is_full},  $\mathbf{Mon}(\mathbb{TopFreeMod}_{(\ring{R},\mathsf{d})})$ is isomorphic to a full subcategory of ${}_1\mathbf{TopAlg}_{(\ring{R},\mathsf{d})}$ (\cite[Proposition~4.5, p.~49]{Cats}). Accordingly a monoid in $\mathbb{TopFreeMod}_{(\ring{R},\mathsf{d})}$ is essentially a topological algebra. 

%\begin{remark}
It is clear that  ${}_{c}\mathbf{Mon}(\mathbb{TopFreeMod}_{(\ring{R},\mathsf{d})})\xrightarrow{TA}{}_{1,c}\mathbf{TopAlg}_{(\ring{R},\mathsf{d})}$ of $TA$ (see Remark~\ref {rem:TA_commutative_case}) also is a full embedding functor.

\subsection{Relations with $D_{fin}$}\label{paragraph:relation_with_D_fin}

%Finite duality between coalgebras and algebras (see Appendix~\ref{Appendix:finite_duality}, p.~\pageref{Appendix:finite_duality}) has some fundamental topological features that one now briefly discusses in order that afterwards relations with our topological duality can be brought. 
%
Let $V$ be any vector space on a field $\mathbb{k}$. Then, $V^*$ has a somewhat natural topology called the {\em $V$-topology} (\cite{Abe}) or the {\em finite topology} (\cite{Dasca}), with a fundamental system of neighborhoods of zero consisting of spaces \begin{equation}\label{eq:orthogonal_space_for_Cstar} W^{\dagger}:=\{\, \ell\in V^*\colon \forall w\in W,\ \ell(w)=0\,\}\index{$W^{\dagger}$}\end{equation} where $W$ runs over the finite-dimensional subspaces of $V$. This is manifestely the same topology as our $w^{*}_{(\mathbb{k},\mathsf{d})}$ (see Section~\ref{subsec:algebraic_dual_functor}). Accordingly  this turns $V^*$ into a linearly compact $\mathbb{k}$-vector space (p.~\pageref{subsubsec:linearly_compact_tvs}). The closed subspace of $(V^*,w^{*}_{(\mathbb{k},\mathsf{d})})$ are exactly the subspaces of the form $W^{\dagger}$, where $W$ is any subspace of $V$ (\cite[Proposition~24.4, p.~105]{Bergman}). 
\begin{lemma}\label{lem:codim_dim_1}
Let $W$ be a subspace of $V$. $\mathrm{codim}(W^{\dagger})$ is finite if, and only if, $\dim(W)$ is finite. In this case, $\mathrm{codim}(W^{\dagger})=\dim(W)$.
\end{lemma}
\begin{proof}
One observes that $V^*/W^{\dagger}\simeq W^*$ because the map $V^*\xrightarrow{incl_W^*}W^*$ is onto, where $W\xrightarrow{incl_W}V$ is the canonical inclusion, and $\ker incl_W^*=W^\dagger$. Since $V^*/W^\dagger\simeq W^*$, it follows that  $\mathrm{codim}(W^{\dagger})=\dim W^*$. %The end of the proof is obvious. %In particular, $\mathrm{codim}(W^{\dagger})$ is finite if, and only if, $\dim (W)$ is finite, and thus in this case $\mathrm{codim}(W^\dagger)=\dim(W)$.
\end{proof}
%
%\begin{remark}\label{rem:nice_connection_ideal_subcoalg_finite_dual}
%Given a $\mathbb{k}$-coalgebra $\ring{C}=(C,\delta_{\ring{C}},\epsilon_{\ring{C}})$, there is a nice connection between its subcoalgebras and the two-sided ideals of $\widetilde{\mathbb{D}_{*}}(\ring{C})$ since $D$ is a subcoalgebra of $\ring{C}$ if, and only if, $D^{\dagger}$ is a two-sided ideal of $\widetilde{\mathbb{D}_{*}}(\ring{C})$ (\cite[Theorem~2.3.1, p.~78]{Abe}).
%\end{remark}

\begin{theorem}\label{thm:findual}
Let $\mathbb{k}$ be a field. For each monoid $\mathsf{A}$ in $\mathbb{TopFreeVect}_{(\mathbb{k},\mathsf{d})}$, $\widetilde{(\mathbb{Top}_{(\mathbb{k},\mathsf{d})}^{\mathbb{d}})}(\mathsf{A})$ is a subcoalgebra of $D^{\mathsf{op}}_{fin}(\tilde{\mathbb{A}}(\mathsf{A}))$.  
Furthermore,  the assertions below are equivalent.
\begin{enumerate}
\item\label{lem:findual:pt:assertion1} In $TA(\mathsf{A})$ every finite-codimensional ideal is closed. 
\item\label{lem:findual:pt:assertion2bis} $\tilde{\mathbb{A}}(A)$ is reflexive\footnote{A monoid $\ring{A}$ in $\mathbb{Vect}_{\mathbb{k}}$ is {\em reflexive}\index{Reflexive algebra} when $\mathsf{A}\simeq \widetilde{\mathbb{D}_{*}}(D^{\mathsf{op}}_{fin}(\mathsf{A}))$ under the linear map $u\mapsto (\ell\mapsto \ell(u))$, which is the unit of the adjunction $D^{\mathsf{op}}_{fin}\dashv \widetilde{\mathbb{D}_*}$.}.
\item\label{lem:findual:pt:assertion2} The coalgebra $\widetilde{(\mathbb{Top}_{(\mathbb{k},\mathsf{d})}^{\mathbb{d}})}(\mathsf{A})$ is coreflexive\footnote{\label{fn:coreflx}A coalgebra $\ring{C}$ is {\em coreflexive}\index{Coreflexive coalgebra}, when $\ring{C}\simeq D^{\mathsf{op}}_{fin}(\widetilde{\mathbb{D}_{*}}(\mathsf{C}))$ under the natural inclusion $u\mapsto (\ell\mapsto \ell(u))$, which is the counit of $D^{\mathsf{op}}_{fin}\dashv \widetilde{\mathbb{D}_*}$.}.
%\item\label{lem:findual:pt:assertion2bis} $\tilde{F}\mathsf{A}$ is reflexive (i.e., $\tilde{F}\mathsf{A}\simeq D_{*}(D_{fin}^{\mathsf{op}}(\tilde{F}\mathsf{A}))$).
\item\label{lem:findual:pt:assertion3} $\widetilde{(\mathbb{Top}_{(\mathbb{k},\mathsf{d})}^{\mathbb{d}})}(\mathsf{A})=D^{\mathsf{op}}_{fin}(\tilde{\mathbb{A}}(\mathsf{A}))$.
\end{enumerate}
\end{theorem}

\begin{proof}
%Given an object $\mathsf{B}=(B,\gamma,\zeta)$ in $\mathbf{Alg}_{\mathbb{k}}$, let $B^0$ be the underlying vector space of $D_{fin}\mathsf{B}$. 
%
Let $\mathsf{A}=((A,\sigma),\mu,\eta)$ be a monoid in $\mathbb{TopFreeMod}_{(\mathbb{k},\mathsf{d})}$. Whence its underlying topological vector space is a linearly compact vector space  (p.~\pageref{subsubsec:linearly_compact_tvs}).   Let $\mathsf{C}:=\widetilde{(\mathbb{Top}_{(\mathbb{k},\mathsf{d})}^{\mathbb{d}})}(\mathsf{A})$. Since $\tilde{\mathbb{A}}\circ \widetilde{\mathbb{Alg}_{(\mathbb{k},\mathsf{d})}}=\widetilde{\mathbb{D}_{*}}$ it follows that $\tilde{\mathbb{A}}\simeq \tilde{\mathbb{A}}\circ  \widetilde{\mathbb{Alg}_{(\mathbb{k},\mathsf{d})}}\circ \widetilde{(\mathbb{Top}_{(\mathbb{k},\mathsf{d})}^{\mathbb{d}})}\simeq \widetilde{\mathbb{D}_{*}}\circ\widetilde{(\mathbb{Top}_{(\mathbb{k},\mathsf{d})}^{\mathbb{d}})}$ (naturally  isomorphic). In particular, $\tilde{\mathbb{A}}(\mathsf{A})\simeq  \widetilde{\mathbb{D}_{*}}(\mathsf{C})$. By construction, the underlying topological vector space of $\mathsf{A}$, namely $(A,\sigma)$, is also the underlying topological vector space of $TA(\mathsf{A})$. Also $\mathsf{A},TA(\mathsf{A})$ and $\tilde{\mathbb{A}}(\mathsf{A})$ share the same underlying vector space $A$, which is isomorphic to $C^*$, where $C$ is the underlying vector space of the coalgebra $\mathsf{C}$. Of course, $(A,\sigma)\simeq Alg_{(\mathbb{k},\mathsf{d})}(C)=(C^*,w^*_{(\mathbb{k},\mathsf{d})})$. Therefore, up to such an isomorphism, $(A,\sigma)$ has a fundamental system of neighborhoods of zero consisting of $V^{\dagger}=\{\, \ell\in A\colon \forall v\in V,\ \ell(v)=0\,\}$ where $V$ is a finite-dimensional subspace of $C$ (see Eq.~(\ref{eq:orthogonal_space_for_Cstar})). 

Let $\ell\in (A,\sigma)'$. By continuity of $\ell$, there exists a finite-dimensional subspace $V$ of $C$ such that $V^{\dagger}\subseteq \ker \ell$. Let $B$ be a (finite) basis of $V$, and let $D$ be the (necessarily finite-dimensional, by~\cite[Thm~1.3.2, p.~21]{Lambe}) subcoalgebra of $\mathsf{C}$ it generates. Then, $V\subseteq {D}$, which implies that $D^{\dagger}\subseteq V^{\dagger}\subseteq \ker \ell$. But $D^{\dagger}$ is a finite-codimensional ideal of $\tilde{\mathbb{A}}(\mathsf{A})$ (by Lemma~\ref{lem:codim_dim_1} and~\cite[Theorem~2.3.1, p.~78]{Abe}), whence $\ell\in A^0$.\footnote{$A^0:=\{\, \ell\in A^*\colon \ker\ell\ \mbox{contains a finite-codimensional (two-sided) ideal of $UA(\mathsf{A})$}\,\}\index{A0@$A^0$}$ is the underlying vector space of the finite dual coalgebra $D_{fin}(\tilde{\mathbb{A}}(\mathsf{A}))$.} %Therefore, $(A,\sigma)'\subseteq A^0$.

It remains to check that the above inclusion $incl_{(A,\sigma)'}$ is a  coalgebra map from  $\widetilde{(\mathbb{Top}_{(\mathbb{k},\mathsf{d})}^{\mathbb{d}})}(\mathsf{A})$ to $D^{\mathsf{op}}_{fin}(\tilde{\mathbb{A}}(\mathsf{A}))$, which would equivalently mean that  $(A,\sigma)'$ is a subcoalgebra of $D^{\mathsf{op}}_{fin}(\tilde{\mathbb{A}}(\mathsf{A}))$. One thus needs to make explicit the two coalgebra structures so as to make possible a comparison. By construction the comultiplication of $\widetilde{(\mathbb{Top}_{(\mathbb{k},\mathsf{d})}^{\mathbb{d}})}(\mathsf{A})$ is given by the composition $\Lambda^{-1}_{(A,\sigma)'\otimes_{\mathbb{k}} (A,\sigma)'}\circ \mu'$. So for $\ell\in (A,\sigma)'$, $(\Lambda^{-1}_{(A,\sigma)'\otimes_{\mathbb{k}} (A,\sigma)'}\circ \mu')(\ell)=\sum_{i=1}^{n}\ell_i\otimes r_i$, for some $\ell_i,r_i\in (A,\sigma)'$. Therefore, given $\ell\in (A,\sigma)'$, $u,v\in A$, $\ell(\mu(u\ostar v))=\sum_{i=1}^n\ell_i(u)r_i(v)$. 
%\begin{equation}\label{eq:calcul_comult_in_top_dual_coalg}
%\begin{array}{lll}
%\ell(\mu(u\ostar v))&=&(\mu'(\ell))(u\ostar v)\\
%&=&(\Lambda_{(A,\sigma)'\otimes_{\mathbb{k}}(A,\sigma)'}(\sum_{i=1}^n\ell_i\otimes r_i))(u\ostar v)\\
%&=&(u\ostar v)(\sum_{i=1}^{n}\ell_i\otimes r_i)\\
%&=&\sum_{i=1}^n(u\ostar v)(\ell_i\otimes r_i)\\
%&=&\sum_{i=1}^n\ell_i(u)r_i(v).
%\end{array}
%\end{equation} 
The counit of $\widetilde{(\mathbb{Top}_{(\mathbb{k},\mathsf{d})}^{\mathbb{d}})}(\mathsf{A})$ is $(A,\sigma)'\xrightarrow{\eta'}(\mathbb{k},\mathsf{d})'\xrightarrow{\psi^{-1}}\mathbb{k}$, i.e., $\ell\mapsto \ell(\eta(1_{\mathbb{k}}))$. It follows easily, from the explicit description of $D_{fin}(\mathsf{B})$ provided in~\cite[p.~35]{Dasca}, for a monoid $\mathsf{B}$ in $\mathbb{Vect}_{\mathbb{k}}$, that the above  comultiplication coincides with that of $D_{fin}(\tilde{\mathbb{A}}(A))$, and  because it is patent that the counit of $\widetilde{(\mathbb{Top}_{(\mathbb{k},\mathsf{d})}^{\mathbb{d}})}(\mathsf{A})$ is the restriction of that of $D^{\mathsf{op}}_{fin}(\tilde{\mathbb{A}}(\mathsf{A}))$, $(A,\sigma)'$ is a subcoalgebra of $D^{\mathsf{op}}_{fin}(\tilde{\mathbb{A}}(\mathsf{A}))$.  

%It remains to check the comultiplications are similarly related. This follows from the commutativity of the diagram below, in $\mathbf{Vect}_{\mathbb{k}}$, where the unnamed inclusions are the obvious ones (and $\ostar$ stands for $\ostar_{(\mathbb{k},\mathsf{d})}$, $\otimes$ is $\otimes_{\mathbb{k}}$, $incl$ is an abbreviation of $incl_{(A,\sigma)'}$ and $\Xi^{0}_{(A,\sigma),(A,\sigma)}$ is the co-restriction of $\Xi^{*}_{(A,\sigma),(A,\sigma)}$).
%\begin{footnotesize} 
%\begin{equation}
%\xymatrix@C=0.6pt{
%&((A,\sigma)\ostar(A,\sigma))' \ar@{^{(}->}[d]\ar[rr]^{\Lambda^{-1}_{(A,\sigma)'\otimes(A,\sigma)'}}&& (A,\sigma)'\otimes(A,\sigma)'\ar@{^{(}->}[d]\ar[rdd]^{incl\otimes incl}\\
%\ar@{^{(}->}[d]^{incl}(A,\sigma)'\ar@{^{(}->}[rd]\ar[ru]^{\mu'} &(((A,\sigma)'\otimes (A,\sigma)')^*)^* \ar[r]^{\Xi_{(A,\sigma),(A,\sigma)}^*}& (A\otimes A)^* &\ar[l]_{\Theta_{A,A}}A^*\otimes A^*\\
%A^0\ar@/_1pc/[rr]_{\mu^0}\ar@{^{(}->}[r]&A^* \ar[u]^{\mu^{*}}& \ar@{_{(}->}[ul](((A,\sigma)'\otimes (A,\sigma)')^*)^0\ar[r]_{\Xi^{0}_{(A,\sigma),(A,\sigma)}}&(A\otimes A)^0 \ar@{_{(}->}[ul]\ar[r]_{(\Theta^0)^{-1}}&\ar@{_{(}->}[ul]A^0\otimes A^0
%}
%\end{equation}
%\end{footnotesize}
%
%(Here it suffices to check that the top diagram commutes. Let $\ell_1,\ell_2\in (A,\sigma)'$ and $u,v\in A$. Then, $(\Xi^*_{(A,\sigma),(A,\sigma)}(\Lambda_{(A,\sigma)'\otimes_{\mathbb{k}}(A,\sigma)'}(\ell_1\otimes\ell_2)))(u\otimes v)=(\Lambda_{(A,\sigma)'\otimes_{\mathbb{k}}(A,\sigma)'}(\ell_1\otimes\ell_2))(u\ostar v)=(u\ostar v)(\ell_1\otimes \ell_2)=\ell_1(u)\ell_2(v)=(\Theta_{A,A}(\ell_1\otimes\ell_2))(u\otimes v)$.)

It remains to prove the equivalence of the four assertions given in the statement. \ref{lem:findual:pt:assertion2bis}$\Leftrightarrow$\ref{lem:findual:pt:assertion2} since  finite duality restricts to an equivalence of categories  between the full categories of reflexive algebras and of coreflexive coalgebras (in a standard way; see e.g.,~\cite[Proposition~4.2, p.~16]{Lambek}), and $\tilde{\mathbb{A}}(\mathsf{A})\simeq \widetilde{\mathbb{D}_{*}}(\widetilde{(\mathbb{Top}^{\mathbb{d}}_{(\mathbb{k},\mathsf{d})})}(\mathsf{A}))$. %equivalence of categories (as for every adjunction; see e.g.,~\cite[Proposition~4.2, p.~16]{Lambek}) when one restricts the adjunction to the full subcategories of {\em coreflexive coalgebras}\index{Coreflexive coalgebra}, i.e., those coalgebras $\ring{C}$ such that $\ring{C}\simeq D^{\mathsf{op}}_{fin}(\widetilde{\mathbb{D}_{*}}(\mathsf{C}))$ under the natural inclusion $C\hookrightarrow (C^*)^0$, $u\mapsto (\ell\mapsto \ell(u))$, and of {\em reflexive algebras}\index{Reflexive algebra}, i.e., those monoids $\mathsf{A}$ such that $\mathsf{A}\simeq \widetilde{\mathbb{D}_{*}}(D^{\mathsf{op}}_{fin}(\mathsf{A}))$ under the linear map $A\to (A^0)^*$, $u\mapsto (\ell\mapsto \ell(u))$.

The coalgebra $\mathsf{C}:=\widetilde{(\mathbb{Top}_{(\mathbb{k},\mathsf{d})}^{\mathbb{d}})}(\mathsf{A})$ is coreflexive if, and only if, every finite-codimensional ideal of $\widetilde{\mathbb{D}_{*}}(\mathsf{C})\simeq \tilde{\mathbb{A}}(\mathsf{A})$ is closed in the finite topology of $C^*$ (\cite[Lemma~2.2.15, p.~76]{Abe}), which coincides with our topology $w^*_{(\mathbb{k},\mathsf{d})}$, and thus it turns out that  $(\widetilde{\mathbb{D}_{*}}(\mathsf{C}),w^*_{(\mathbb{k},\mathsf{d})})\simeq TA(\mathsf{A})$ (since $C^*$ under the finite topology is equal to $Alg_{(\mathbb{k},\mathsf{d})}(C)\simeq (A,\sigma)$ by functoriality). Whence \ref{lem:findual:pt:assertion2}$\Leftrightarrow$\ref{lem:findual:pt:assertion1}.

%that \ref{lem:findual:pt:assertion2}$\Leftrightarrow$\ref{lem:findual:pt:assertion1} is a characterization of coreflexive coalgebras\footnote{Actually the characterization of coreflexive coalgebras we have in mind makes reference to ideals of the dual algebra $\widetilde{\mathbb{D}_{*}(\mathsf{C})}$ of a coalgebra $\mathsf{C}$ closed for the finite topology on $C^*$ (see~\cite[Lemma~2.2.15, p.~76]{Abe}). To make the connection with our topologies requires to see that the finite topology on the dual $V^*$ of a vector space $V$ coincides with $w^*_{(\mathbb{k},\mathsf{d})}$, which is actually clear from its definition.} (\cite[Lemma~2.2.15, p.~76]{Abe}).

%The rest of this part of the proof is straightforward (in particular $\widetilde{(D_{top}^{(\mathbb{k},\mathsf{d})})^{\mathsf{op}}}\mathsf{A}\xrightarrow{\|-\|}D_{fin}^{\mathsf{op}}(\tilde{F}\mathsf{A})$ is natural in the monoid object $\mathsf{A}$).

Let us assume that in $TA(\mathsf{A})$ every finite-codimensional ideal is closed. Let $\ell\in D_{fin}(\tilde{\mathbb{A}}(\mathsf{A}))$. By definition $\ker\ell$ contains a finite-codimensional ideal say $I$ of $\tilde{\mathbb{A}}(\mathsf{A})$. Since $I$ is closed, there exists a finite-dimensional subspace $D$ of $C$ such that $D^{\dagger}=I$ (since  the closed subspaces are of the form $D^{\dagger}$ for a subspace $D$ of $C$ and by Lemma~\ref{lem:codim_dim_1}, $\mathrm{codim}(I)=\mathrm{codim}(D^{\dagger})=\dim(D)$), whence $I$ is open, which shows that $(A,\sigma)\xrightarrow{\ell}(\mathbb{k},\mathsf{d})$ is continuous so \ref{lem:findual:pt:assertion1}$\Rightarrow$\ref{lem:findual:pt:assertion3}. 

Let  $\mathsf{C}:=\widetilde{(\mathbb{Top}_{(\mathbb{k},\mathsf{d})}^{\mathbb{d}})}(\mathsf{A})=D^{\mathsf{op}}_{fin}(\tilde{\mathbb{A}}(\mathsf{A}))$, so that $\tilde{\mathbb{A}}(\mathsf{A})\simeq  \widetilde{\mathbb{D}_{*}}(\mathsf{C})$, as above. Whence $\mathsf{C}=D^{\mathsf{op}}_{fin}(\tilde{\mathbb{A}}(\mathsf{A}))\simeq D^{\mathsf{op}}_{fin}(\widetilde{\mathbb{D}_{*}}(\mathsf{C}))$. This is not sufficient to ensure coreflexivity of $\mathsf{C}$, since  there is at this stage no guaranty that the above isomorphism corresponds  to the counit of the adjunction $D_{fin}^{\mathsf{op}}\dashv\widetilde{\mathbb{D}_*}$ (see Footnote~\ref{fn:coreflx}). One knows from the beginning of the proof that $$\tilde{\mathbb{A}}(\tilde{\Gamma}_{\mathsf{A}})\colon \mathbb{A}(\mathsf{A})\simeq \tilde{\mathbb{A}}(\widetilde{\mathbb{Alg}_{(\mathbb{k},\mathsf{d})}}(\widetilde{(\mathbb{Top}_{(\mathbb{k},\mathsf{d})}^{\mathbb{d}})}(\mathsf{A})))$$ which, in this case where $\widetilde{(\mathbb{Top}_{(\mathbb{k},\mathsf{d})}^{\mathbb{d}})}(\mathsf{A})=D^{\mathsf{op}}_{fin}(\tilde{\mathbb{A}}(\mathsf{A}))$, is the isomorphism $\|\Gamma_{(A,\sigma)}\|\colon A\simeq ((A,\sigma)')^*=(A^0)^*$, $u\mapsto (\ell\mapsto \ell(u))$. So $\tilde{\mathbb{A}}(\mathsf{A})$ is reflexive, and thus its finite dual coalgebra $\mathsf{C}$ is coreflexive. Thus \ref{lem:findual:pt:assertion3}$\Rightarrow$\ref{lem:findual:pt:assertion2}.
\end{proof}

\begin{example}\label{ex:hadamard_monoid_3}
Let $\ring{R}$ be a ring. Let $C_{\ring{R}}X=(R^{(X)},d_X,e_X)$ be the group-like coalgebra on $X$, i.e., $d_X(\delta_x)=\delta_x\otimes \delta_x$, and $e_X(\delta_x)=1_{\ring{R}}$, $x\in X$.  The following diagram commutes (this may be checked by hand) for a rigid ring $(\ring{R},\tau)$. %Let us check that the following diagram commutes.
\begin{equation}
\xymatrix@R=0.7pc{
((R,\tau)^X)' \ar[d]_{\lambda_X}\ar[rr]^{\mu_X'}& &((R,\tau)^X\ostar_{(\ring{R},\tau)}(R,\tau)^X)'\\
R^{(X)} \ar[r]_-{d_X}& R^{(X)}\otimes_{\ring{R}}R^{(X)} \ar[r]_-{\lambda_X^{-1}\otimes_{\ring{R}}\lambda_X^{-1}}& ((R,\tau)^X)'\otimes_{\ring{R}}((R,\tau)^X)'\ar[u]_{\Lambda_{((R,\tau)^X)'\otimes_{\ring{R}}((R,\tau)^X)'}}
}
\end{equation}
%Let $\ell\in ((R,\tau)^X)'$, $x,y\in X$. Then, 
%\begin{equation}
%\begin{array}{ll}
%&(\Lambda_{((R,\tau)^X)'\otimes_{\ring{R}}((R,\tau)^X)'}((\lambda_X^{-1}\otimes_{\ring{R}}\lambda_X^{-1})(d_X(\lambda_X(\ell)))))(\delta_x\ostar \delta_y)\\
%=&(\Lambda_{((R,\tau)^X)'\otimes_{\ring{R}}((R,\tau)^X)'}((\lambda_X^{-1}\otimes_{\ring{R}}\lambda_X^{-1})(\sum_{z\in Z}\ell(\delta_z)\delta_z\otimes\delta_z)))(\delta_x\ostar \delta_y)\\
%=&(\delta_x\ostar \delta_y)(\sum_{z\in Z}\ell(\delta_z)\pi_z\otimes \pi_z)\\
%=&\sum_{z\in Z}\ell(\delta_z)\pi_z(\delta_x)\pi_z(\delta_y)\\
%=&\ell(\delta_x)\delta_x(y)\\
%=&\ell(\mu_X(\delta_x\ostar \delta_y))\\
%=&(\mu_X'(\ell))(\delta_x\ostar \delta_y).
%\end{array}
%\end{equation}
Moreover $\eta_X'(\ell)=\psi(e_X(\lambda_X(\ell)))$ for each $\ell\in ((R,\tau)^X)'$. 
%
%since $(\eta_X'(\ell))(1_{\ring{R}})=\ell(\eta_X(1_{\ring{R}}))=\sum_{x\in X}\ell(\delta_x)$ and 
%\begin{equation}
%\begin{array}{lll}
%\psi(e_X(\lambda_X(\ell)))&=&\psi(\sum_{x\in X}\ell(\delta_x))\\
%&=&(\sum_{x\in X}\ell(\delta_x))id_{R}.
%\end{array}
%\end{equation} 
All of this shows that $\lambda_X\colon \widetilde{(\mathbb{Top}^{\mathbb{d}}_{(\ring{R},\tau)})}(M_{(\ring{R},\tau)}(X))\simeq C_{\ring{R}}X$ is an isomorphism of coalgebras. %Example~\ref{ex:explicit_calculus_of_tilde_Gamma} (p.~\pageref{ex:explicit_calculus_of_tilde_Gamma}) has for consequence that  $\widetilde{(\mathbb{Top}^{\mathbb{d}}_{(\ring{R},\tau)})^{\mathbb{op}}}(M_{(\ring{R},\tau)}(X))\simeq C_{\ring{R}}X$, i.e., that the topological dual coalgebra of $M_{(\ring{R},\tau)}(X)$ is the group-like coalgebra $C_{\ring{R}}X$ on $X$ (Appendix~\ref{appendix:coalgebras:subsec:grouplike}, p.~\pageref{appendix:coalgebras:subsec:grouplike}). 

Let $\mathbb{k}$ be a field. It follows from Theorem~\ref{thm:findual} and Example~\ref{ex:hadamard_monoid_2bis} that in $A_{(\mathbb{k},\mathsf{d})}(X)$ every finite-codimensional ideal is closed if, and only if $C_{\mathbb{k}}X$ is coreflexive if, and only if, $\{\, \pi_x\colon x\in X\,\}={}_1\mathbf{Alg}_{\mathbb{k}}(A_{\mathbb{k}}(X),\mathbb{k})$ (\cite[Corollary~3.2, p.~528]{RadfordCoreflCoalg}). This holds in particular if $|X|\leq |\mathbb{k}|$ (see~\cite[Corollary~3.6, p.~529]{RadfordCoreflCoalg}). If $\mathbb{k}$ is a finite field, then $C_{\mathbb{k}}(X)$ is coreflexive if, and only if, $X$ is finite (see~\cite[Remark~3.7, p.~530]{RadfordCoreflCoalg}). 
\end{example}

% Your bilbigraphy           %<-------------------

%\begin{scriptsize}
\appendix

\section{Monoidal categories and functors}\label{appendix:moncat}
 %Monoidal subcategory. Dual monoidal category $\mathbb{C}^{\mathbb{op}}$
 
This appendix contains basic facts about monoidal categories and monoidal functors, and a  part of it is reprinted from~\cite{PoinsotPorstunit}. See~\cite[Chap.~VII]{MacLane} for more details.

Throughout  $\mathbb{C} = (\mathbf{C}, -\otimes -, I, \alpha, \lambda, \rho)=(\mathbf{C}, -\otimes -, I)$\index{C@$\mathbb{C}$} (resp. $\mathbb{C} = (\mathbf{C}, -\otimes -, I, \alpha, \lambda, \rho,\sigma)$) denotes a  (resp. symmetric) monoidal category with $\alpha$ the
associativity and $\lambda$ and $\rho$ the
left and right unit constraints (resp., and $\sigma$ the symmetry), referred to as {\em coherence constraints}. These constraints have to make commute some diagrams to ensure {\em coherence} of the (resp. symmetric) monoidal category (see~\cite[Chap.~VII, p.~165]{MacLane}). If $\mathbb{C}$ is a (symmetric) monoidal category, then so is $\mathbb{C}^{\mathbb{op}}:=(\mathbf{C}^{\mathsf{op}},-\otimes^{\mathsf{op}}-,I,(\alpha^{-1})^{\mathsf{op}},(\varrho^{-1})^{\mathsf{op}},(\lambda^{-1})^{\mathsf{op}})$\index{Copp@$\mathbb{C}^{\mathbb{op}}$}, called the {\em dual}\index{Dual monoidal category} monoidal category of $\mathbb{C}$. 
%When $\mathbf{C}$ has finite products, then with the tensor product given by the categorical product and the unit given by the terminal object of $\mathbf{C}$, the coherence conditions are automatically fulfilled, and the symmetric monoidal category thus provided is referred to as {\em cartesian}\index{Cartesian monoidal category}. 

\subsection{Monoids and comonoids}\label{appendix:mon_and_comon}
A {\em monoid in $\mathbb{C}$}\index{Monoid} is a triple $(C, C\otimes C\overset{m}{\longrightarrow}
C, I\overset{e}{\longrightarrow} C)$ such that the diagrams
\begin{center}
\begin{minipage}{7cm}
\xymatrix@=2em{\ar[d]_{\alpha_{C,C,C}}(C\otimes C)\otimes C\ar[r]^{\ \ m\otimes
id_C}&C\otimes C\ar[dd]^m\\ C\otimes (C\otimes C)\ar[d]_{id_C\otimes m} &\\
C\otimes C\ar[r]_m&C}
\end{minipage}
\hspace{1cm}
\begin{minipage}{7cm}\xymatrix@=2em{C\otimes
I\ar[r]^{id_C\otimes e}\ar[dr]_{\rho_C}&C\otimes
C\ar[d]^m&I\otimes C\ar[l]_{e\otimes
id_C}\ar[dl]^{\lambda_C}\\&C&}\end{minipage}
\end{center}commute, while a \em monoid morphism \em $(C, m, e)\longrightarrow (C', m', e')$
 is any $f\colon C\rightarrow C'$ making the diagrams
\begin{center}
\begin{minipage}{7cm}
\xymatrix@=2em{C\otimes C\ar[r]^m\ar[d]_{f\otimes f}
&C\ar[d]^f\\C'\otimes C'\ar[r]_{m'}&C'}
\end{minipage}\hspace{2cm}\begin{minipage}{7cm}\xymatrix@=2em{I\ar[r]^e \ar[dr]_{e'}&
C\ar[d]^f\\ &C'}\end{minipage}
\end{center}
commutative. This defines the category $\mathbf{Mon}\mathbb{C}$\index{Mon@$\mathbf{Mon}\mathbb{C}$} of monoids in $\mathbb{C}$.\\
The category $\mathbf{Comon}\mathbb{C}$\index{Comon@$\mathbf{Comon}\mathbb{C}$} of {\em comonoids}\index{Comonoid} over $\mathbb{C}$ is defined to be
$(\mathbf{Mon}\mathbb{C}^\mathbb{op})^\mathsf{op}$, the opposite of the category of monoids in $\mathbb{C}^\mathbb{op}$.

A monoid $(C, m, e)$ is called {\em commutative} if, and only if, $m= m\circ \sigma_{C,C}$
with $\sigma_{C,C}\colon C\otimes C\rightarrow C\otimes C$ the symmetry; dually, a comonoid $(C, \mu,
\epsilon)$ is called {\em cocommutative}\index{Cocommutative}, provided that $\mu = \sigma_{C,C}\circ\mu$.
By $_c\mathbf{Mon}\mathbb{C}$\index{Mon2@$_c\mathbf{Mon}\mathbb{C}$} and $_{coc}\mathbf{Comon}\mathbb{C}$\index{Comon2@$_{coc}\mathbf{Comon}\mathbb{C}$} we
denote the categories of commutative monoids and cocommutative
comonoids  respectively, with all (co)monoid morphisms as morphisms. Of course, one has $_{coc}\mathbf{Comon}\mathbb{C} =
(_c\mathbf{Mon} \mathbb{C}^\mathbb{op})^\mathsf{op}$. %\\
%Omitting the unit $e$ of a monoid as well as the respective axioms one obtains the categories  $\mathbf{Sgr}\mathbb{C}$ and  $_c\mathbf{Sgr}\mathbb{C}$ of (commutative)  semigroups $(S, s)$ in $\mathbb{C}$. Dually one has the categories $\mathbf{Cosgr}\mathbb{C}$ and $_{coc}\mathbf{Cosgr}\mathbb{C}$ of (cocommutative) cosemigroups.
%The following are well-known examples of these concepts. 
\begin{example}\label{ex:cats_of_mon}
\begin{enumerate}
\item  $\mathbf{Mon}(\mathbb{Set})$ is (isomorphic to) the category of monoids, where $\mathbb{Set}=(\mathbf{Set},\times,1)$ ($1:=\{\,\emptyset\,\}$).%, while $\mathbf{Sgr}\mathbb{C}$ is (isomorphic to) the category of ordinary semigroups.
\item Let $\mathbb{Mod}_{\ring{R}}$ be the monoidal category of $\ring{R}$-modules and $\ring{R}$-linear maps for a commutative unital ring $\ring{R}$ with its usual tensor product $\otimes_{\ring{R}}$.
\begin{enumerate}
\item  $\mathbf{Mon}(\mathbb{Mod}_{\ring{R}})$ is isomorphic to the category {$_1\mathbf{Alg}_{\ring{R}}$}  of ``ordinary'' unital $\ring{R}$-algebras under the functor $O$, concrete over $\mathbf{Mod}_{\ring{R}}$, such that $O(\ring{A}):=(A,m_{\ring{A}},1_{\ring{A}})$\index{O@$O$}, with $m_{\ring{A}}(x,y):=\mu(x\otimes y)$, $x,y\in A$, and $1_{\ring{A}}:=\eta(1_{\ring{R}})$, where $\ring{A}=(A,\mu_{\ring{A}},\eta_{\ring{A}})$ is a monoid in $\mathbb{Mod}_{\ring{R}}$. %Indeed, a unital $\ring{R}$-algebra $\mathsf{A}=(A,m_{\mathsf{A}},1_{\mathsf{A}})$ corresponds to the monoid $(A,\mu_{\mathsf{A}},\eta_{\mathsf{A}})$, with $A\otimes_{\ring{R}}A\xrightarrow{\mu_{\mathsf{A}}}A$ the unique $\ring{R}$-linear map such that $\mu_{\mathsf{A}}(x\otimes y)=m_{\mathsf{A}}(x,y)$, $x,y\in A$, and $R\xrightarrow{\eta_{\mathsf{A}}}A$ is given by $\eta_{\mathsf{A}}(1_{\ring{R}})=1_{\mathsf{A}}$. Conversely from  a monoid $(A,\mu,\eta)$ one obtains a unital algebra $(A,\mu_{bil},\eta(1_{\ring{R}}))$, where $A\times A\xrightarrow{\mu_{bil}}A$ is the unique $\ring{R}$-bilinear map associated with $A\otimes_{\ring{R}}A\xrightarrow{\mu}A$. \\
%The correspondence $O\colon (A,\mu,\eta)\mapsto (A,\mu_{bil},\eta(1_{\ring{R}}))$\index{O@$O$} defines a functorial isomorphism $\mathbf{Mon}(\mathbb{Mod}_{\ring{R}})\simeq {}_1\mathbf{Alg}_{\ring{R}}$, which is the identity on arrows, and is concrete over $\mathbf{Mod}_{\ring{R}}$ (for the obvious forgetful functors). 
\item Likewise ${}_{c}\mathbf{Mon}(\mathbb{Mod}_{\ring{R}})\simeq {}_{1,c}\mathbf{Alg}_{\ring{R}}$ under the (co-)restriction of the above functor $O$. 
\item $\mathbf{Comon}(\mathbb{Mod}_{\ring{R}})$ is the category {$_\epsilon\mathbf{Coalg}_{\ring{R}}$}  of counital $\ring{R}$-coalgebras (\cite{Abe, Dasca}), and the category of cocommutative coalgebras ${}_{\epsilon,coc}\mathbf{Coalg}_{\ring{R}}$ is equal to ${}_{coc}\mathbf{Comon}(\mathbb{Mod}_{\ring{R}})$.
\end{enumerate}
\end{enumerate}
\end{example}

%\begin{remark}\label{rem:appendix:mon_and_comon:diag_comonoid}
%$\mathbf{Set}\simeq {}_{coc}\mathbf{Comon}(\mathbb{Set})$ (where $\mathbb{Set}$ is the cartesian monoidal category of sets) under the {\em diagonal comonoid functor} $\Delta\colon X\mapsto (X,\Delta_X,!_X)$  (see~\cite[p.~5]{Porst}), with inverse the forgetful functor, where $\Delta_X(x):=(x,x)$ and $!_X(x):=\emptyset$, $x\in X$ (one uses $1:=\{\,\emptyset\,\}$ for terminal object in $\mathbf{Set}$).
%\end{remark}

\subsection{Monoidal functors and their induced functors}\label{appendix:moncat:mon_fun_and_ind_fun}

We briefly recall the following	definitions and facts, too,  which are fundamental for this note. See e.g. \cite{Aguiar,Porst, Str} for a more detailed treatment and for the missing proofs.

 \begin{definition}\label{def:monfct}
Let $\mathbb{C} = (\mathbf{C}, -\otimes -,I)$ and  $\mathbb{C}' = (\mathbf{C}', -\otimes' -,I')$ be monoidal categories.  A \em (lax) monoidal functor\index{Monoidal functor (lax, strong, strict)} from $\mathbb{C}$ to $\mathbb{C}'$ \em is a triple $\mathbb{F}:=(F,\Phi,\phi)$\index{F@$\mathbb{F}$}, where  $F\colon\mathbf{C}\rightarrow\mathbf{C}'$  is a functor,   $\Phi_{C_1,C_2}\colon FC_1\otimes' FC_2\rightarrow F(C_1\otimes C_2)$ is a natural transformation and $\phi\colon I' \rightarrow  FI$ is a $\mathbf{C}$-morphism, 
subject to certain coherence conditions.
%, in particular to commutativity of the following diagram, where we omit the unital constraints as usual. 
%$$\begin{minipage}{6cm}
%%\ar@{.>}[dl]^d
%%\ar@/^2pc/@{->}[rr]^m
%\xymatrix@=3em{
%    I'\otimes FC=FC=FC\otimes I'  \ar[r]^{\id\otimes\phi  }\ar[d]_{ \phi\otimes\id }\ar[dr]^\id&{ FC\otimes FI} \ar[d]^{\Phi_{C,I} }  \\
%  FI\otimes FC   \ar[r]_{ \Phi_{I,C}}          &   F(I\otimes C)=FC=F(C\otimes I)
%}
%\end{minipage}$$
A lax monoidal functor is called \em strong monoidal \em (resp. \em strict monoidal \em), if $\Phi$ and $\phi$ are isomorphisms (resp. identities). $\Phi,\phi$ are the {\em coherence constraints}\index{Coherence constraints} of $\mathbb{F}$.
\end{definition}

%\begin{enumerate}
%\item Associativity coherence: for each $\mathbf{C}$-objects $A,B,C$ the following diagram commutes.
%\begin{equation}\label{diag:associativity_coherence_monfun}
%\xymatrix{
%\ar[d]_{\Phi_{A,B}\otimes' id_{FC}}(FA\otimes' FB)\otimes' FC \ar[r]^{\alpha'_{FA,FB,FC}} & FA\otimes'(FB\otimes' FC)\ar[d]^{id_A\otimes'\Phi_{B,C}}\\
%F(A\otimes B)\otimes' FC \ar[d]_{\Phi_{A\otimes B,C}}&FA\otimes'F(B\otimes C)\ar[d]^{\Phi_{A,B\otimes C}}\\
%F((A\otimes B)\otimes C)\ar[r]_{F(\alpha_{A,B,C})} &F(A\otimes (B\otimes C))
%}
%\end{equation}
%\item Left unit coherence: for each $\mathbf{C}$-object $B$, the following diagram commutes.
%\begin{equation}\label{diag:left_unit_coherence_monfun}
%\xymatrix{
%I'\otimes' FB \ar[r]^{\phi\otimes'id_{FB}}\ar[d]_{\lambda_{FB}'} & FI\otimes'FB\ar[d]^{\Phi_{I,B}}\\
%FB &\ar[l]^{F\lambda_{B}} F(I\otimes B)
%}
%\end{equation}
%\item Right unit coherence: for each $\mathbf{C}$-object $A$, the following diagram commutes.
%\begin{equation}\label{diag:right_unit_coherence_monfun}
%\xymatrix{
%FA\otimes'I' \ar[r]^{id_{FA}\otimes' \phi}\ar[d]_{\rho_{FA}'} & FA\otimes'FI\ar[d]^{\Phi_{A,I}}\\
%FB &\ar[l]^{F\rho_{A}} F(B\otimes I)
%}
%\end{equation}
%\end{enumerate}

Let $\mathbb{C},\mathbb{C}'$ be symmetric. A monoidal functor $\mathbb{F}\colon \mathbb{C}\to\mathbb{C}'$ is said to be {\em symmetric} when furthermore it satisfies yet another coherence condition relative to the symmetry constraints. 
\begin{Facts}\label{facts:composing_mon_fun}
Let $\mathbb{C}:=(\mathbf{C},\otimes,I)$, $\mathbb{C}':=(\mathbf{C}',\otimes',I')$ and $\mathbb{C}'':=(\mathbf{C}'',\otimes'',I'')$ be  (symmetric) monoidal categories. 
\begin{enumerate}
\item $\mathbb{id}_{\mathbb{C}}:=(id_{\mathbf{C}},id_{-\otimes-},id_I)$\index{id@$\mathbb{id}$}, or simply $\mathbb{id}$, is a monoidal functor from $\mathbb{C}$ to itself, and serves as a unit for the composition of monoidal functors given below. %If no ambiguities arise, one let $\mathbb{id}:=\mathbb{id}_{\mathbb{C}}$.
\item Given monoidal functors $\mathbb{F}=(F,\Phi,\phi)\colon \mathbb{C}\to\mathbb{D}$ and $\mathbb{G}=(G,\Psi,\psi)\colon \mathbb{D}\to\mathbb{E}$, one defines a  monoidal functor $\mathbb{G}\circ\mathbb{F}:=\mathbb{H}=(H,\Theta,\theta)\colon \mathbb{C}\to\mathbb{E}$ with
\begin{enumerate}
\item $H=G\circ F$.
\item Given objects $C_1,C_2$ of $\mathbf{C}$, $\Theta_{C_1,C_2}$ is the composite $\mathbf{C}''$-morphism $G(F(C_1))\otimes''GF(C_2)\xrightarrow{\Psi_{F(C_1),F(C_2)}}G(F(C_1)\otimes' F(C_2))\xrightarrow{G(\Phi_{C_1,C_2})}G(F(C_1\otimes C_2))$.
\item $\theta:=I''\xrightarrow{\psi}G(I')\xrightarrow{G(\phi)}G(F(I))$.
\end{enumerate}
$H$ is strong (resp. symmetric) when $F,G$ so are.
\end{enumerate}
\end{Facts}
%By a {\em (symmetric) monoidal subcategory}\index{Monoidal subcategory} of a (symmetric) monoidal category $\mathbb{C}=(\mathbf{C},-\otimes-,I)$ we mean a subcategory $\mathbf{C}'$ of $\mathbf{C}$, closed under tensor products, containing $I$, and the coherence constraints of $\mathbb{C}$ between $\mathbf{C}'$-objects. (The last condition is automatically fulfilled when $\mathbf{C}'$ is a full subcategory.) The embedding $E_{\mathbf{C}'}$ of $\mathbf{C}'$ into $\mathbf{C}$ then is a strict monoidal functor $\mathbb{E}_{\mathbb{C}'}$.   
%
%Let $\mathbb{F}=(F,\Phi,\phi)\colon \mathbb{C}\to\mathbb{D}$ be a monoidal functor, and let $\mathbb{C}'$ and $\mathbb{D}'$ be  monoidal subcategories of $\mathbb{C}$ and $\mathbb{D}$, respectively. By a {\em restriction} of $\mathbb{F}$ to these subcategories is meant a monoidal functor $\mathbb{F}'=(F',\Phi',\phi')\colon \mathbb{C}'\to\mathbb{D}'$ such that (a) $F'C=FC$ for all $\mathbf{C}'$-object $C$, (b) $\Phi_{A,B}'=\Phi_{A,B}$ for all $\mathbf{C}'$-objects $A,B$, and (c) $\phi'=\phi$. This is equivalent to saying that $\mathbb{F}\circ\mathbb{E}_{\mathbb{C}'}=\mathbb{E}_{\mathbb{D}'}\circ\mathbb{F}'$.
%

\begin{proposition}\label{satz:mon_ind}
Let  $\mathbb{F}=(F,\Phi,\phi)\colon \mathbb{C}\rightarrow\mathbb{C}'$ be a monoidal functor.

%\mbox{ 
$\tilde{\mathbb{F}}(M,m,e)  = (FM, FM\otimes FM\xrightarrow{\Phi_{M,M}}F(M\otimes M)\xrightarrow{Fm}FM, I'\xrightarrow{\phi}FI\xrightarrow{Fe}FM)$ and $\tilde{\mathbb{F}} f = Ff$
define 
an \em induced  functor\index{Induced functor} \em $\tilde{\mathbb{F}}\colon\mathbf{Mon}\mathbb{C}\rightarrow\mathbf{Mon}\mathbb{C}'$\index{F2@$\tilde{\mathbb{F}}$}, such that the diagram below commutes   (with forgetful functors $U_m$\index{Um@$U_m$} and $U'_m$). 
 %\begin{center}
%\begin{minipage}{7cm}
%\xymatrix@=3em@C=1pc@R=1pc
\begin{equation}
\xymatrix@R=1pc{\mathbf{Mon}\mathbb{C}\ar[r]^{\tilde{\mathbb{F}}}\ar[d]_{U_m}
&\mathbf{Mon}\mathbb{C}'\ar[d]^{U'_m}\\ \mathbf{C}\ar[r]_{F}&\mathbf{C}'}
\end{equation}
%\end{minipage}
%\end{center}
%commutes
\end{proposition}

%The same obviously holds for the categories of semigroups over $\mathbb{C}$ and $\mathbb{C}'$ respectively. Write then $\hat{\mathbb{F}}$ instead	of $\tilde{\mathbb{F}}$.
%We thus have the following commutative diagram (with forgetful functors $V$, $V'$, $U_s$ and $U'_s$).

%$$\begin{minipage}{6cm}
%\xymatrix@=3em{
%    \mathbf{Mon}\mathbb{C} \ar[r]^{\tilde{\mathbb{F}}  }\ar[d]_{ V }\ar@/_2pc/@{->}[dd]_{U_m}&{ \mathbf{Mon}\mathbb{C}'} \ar[d]^{V' }\ar@/^2pc/@{->}[dd]^{U'_m}  \\
%  \mathbf{Sgr}\mathbb{C}   \ar[r]_{\hat{\mathbb{F}}}\ar[d]_{U_s}          &    \mathbf{Sgr}\mathbb{C}' \ar[d]^{U'_s}\\
%  \mathbf{C}\ar[r]^F & \mathbb{C}'
%}
%\end{minipage}$$

\begin{remark}\label{rem:induced_functor_symmetric_case}
\begin{enumerate}
\item When $\mathbb{F}$ is symmetric, then $\tilde{\mathbb{F}}$  also provides a functor ${}_c\mathbf{Mon}\mathbb{C}$ to ${}_{c}\mathbf{Mon}\mathbb{C}'$  with similar properties as above.
\item $\widetilde{\mathbb{id}_{\mathbb{C}}}=id_{\mathbf{Mon}(\mathbb{C})}$ and $\widetilde{\mathbb{G}\circ \mathbb{F}}=\tilde{\mathbb{G}}\circ\tilde{\mathbb{F}}$.% and $\widehat{\mathbb{G}\circ \mathbb{F}}=\hat{\mathbb{G}}\circ\hat{\mathbb{F}}$
\end{enumerate}
\end{remark}

A strong monoidal functor $\mathbb{F}=(F,\Phi,\phi)\colon \mathbb{C}\to\mathbb{C}'$ may be considered as a strong monoidal functor $\mathbb{F}^{\mathbb{d}}:=(F^{\mathsf{op}},(\Phi^{-1})^{\mathsf{op}},(\phi^{-1})^{\mathsf{op}})\colon \mathbb{C}^{\mathbb{op}}\to(\mathbb{C}')^{\mathbb{op}}$\index{F3@$\mathbb{F}^{\mathbb{d}}$}, the {\em dual of $\mathbb{F}$}\index{Dual of $\mathbb{F}$} (\cite[Proposition~17, p.~639]{Porst}).    %Thus it also induces a functor $\mathbb{F}^{\mathbb{d}}_m\colon\mathbf{Comon}\mathbb{C}\rightarrow\mathbf{Comon}\mathbb{C}'$.% and $\mathbb{F}^{\mathbb{d}}_s\colon \mathbf{Cosgr}\mathbb{C}\rightarrow\mathbf{Cosgr}\mathbb{C}'$. 
%
%Of course, one may also define the {\em dual} of a strong opmonoidal functor $\mathbb{F}=(F,\Phi,\phi)\colon \mathbb{C}\to\mathbb{C}'$ as $\mathbb{F}^{\mathbb{d}}:=(F,\Phi^{-1},\phi^{-1})\colon \mathbb{C}\to\mathbb{C}'$ which is a strong monoidal functor.
%
%Observe that 
%\begin{enumerate}
%\item $\mathbb{id}_{\mathbb{C}}^{\mathbb{d}}=\mathbb{id}_{\mathbb{C}}$. 
%
%\item If $\mathbb{G}\colon \mathbb{C}'\to \mathbb{C}''$ is also a strong monoidal functor, then $\mathbb{G}\circ\mathbb{F}$ is strong, and $(\mathbb{G}\circ \mathbb{F})^{\mathbb{d}}=\mathbb{F}^{\mathbb{d}}\circ\mathbb{G}^{\mathbb{d}}$.
%
%\item $(\mathbb{F}^{\mathbb{d}})^{\mathbb{d}}=\mathbb{F}$.
%
%\item $(\mathbb{F}^{\mathbb{op}})^{\mathbb{d}}=(\mathbb{F}^{\mathbb{d}})^{\mathbb{op}}$.
%\end{enumerate}

\begin{definition}\label{def:mon_transfo}
Let $\mathbb{F}=(F,\Phi,\phi)$ and $\mathbb{G}=(G,\Psi,\psi)$ be monoidal functors from $\mathbb{C}$ to $\mathbb{C}'$. 
A natural transformation $\alpha\colon F\Rightarrow G\colon \mathbf{C}\to\mathbf{C}'$ is a {\em monoidal transformation}\index{Monoidal transformation} $\alpha\colon \mathbb{F}\Rightarrow\mathbb{G}\colon \mathbb{C}\to\mathbb{C}'$ when the following diagrams commute, for every $\mathbf{C}$-objects $C_1,C_2$.
\begin{equation}\label{diag:eq1_mon_transfo}
\xymatrix@R=1pc{
FC_1\otimes' FC_2 \ar[r]^{\alpha_{C_1}\otimes'\alpha_{C_2}}\ar[d]_{\Phi_{C_1,C_2}}& GC_1\otimes'GC_2\ar[d]^{\Psi_{C_1,C_2}}&&I' \ar[r]^{\phi}\ar[rd]_{\psi}& FI\ar[d]^{\alpha_I}\\
F(C_1\otimes C_2) \ar[r]_{\alpha_{C_1\otimes C_2}}& G(C_1\otimes C_2)&&&GI
}
\end{equation}
%and
%\begin{equation}\label{diag:eq2_mon_transfo}
%\xymatrix{
%I' \ar[r]^{\phi}\ar[rd]_{\psi}& FI\ar[d]^{\alpha_I}\\
%&GI
%}
%\end{equation}
\end{definition}

\begin{remark}\label{rem:induced_nat_transfo}
\begin{enumerate}
\item Let $\alpha\colon \mathbb{F}\Rightarrow\mathbb{G}\colon \mathbb{C}\to\mathbb{C}'$ be a monoidal transformation. It induces $\tilde{\alpha}\colon\tilde{\mathbb{F}}\Rightarrow \tilde{\mathbb{G}}\colon \mathbf{Mon}\mathbb{C}\to\mathbf{Mon}\mathbb{C}'$ with $\tilde{\alpha}_{(C,m,e)}:=\alpha_C$\index{$\tilde{\alpha}$}. %The same holds for the respective categories of semigroups (with $\hat{\alpha}$ instead of $\tilde{\alpha}$). $\tilde{\alpha}$ (resp. $\hat{\alpha}$) is the natural transformation {\em induced}\index{Induced natural transformation} by the monoidal transformation $\alpha$.
\item When the monoidal functors and categories are symmetric, then $\alpha$ also induces  $\tilde{\alpha}\colon \tilde{\mathbb{F}}\Rightarrow\tilde{\mathbb{G}}\colon {}_{c}\mathbf{Mon}\mathbb{C}\to{}_c\mathbf{Mon}\mathbb{C}'$~\cite[Prop.~3.38]{Aguiar}. %In other words, a monoidal transformation between symmetric monoidal functors, is automatically ``symmetric''. %The same holds for the respective categories of commutative semigroups (with $\hat{\alpha}$ instead of $\tilde{\alpha}$).
%\item By duality, when $\alpha\colon \mathbb{F}^{\mathbb{op}}\Rightarrow\mathbb{G}^{\mathbb{op}}\colon \mathbb{C}^{\mathbb{op}}\to\mathbb{C}'^{\mathbb{op}}$ is a monoidal transformation, between opmonoidal functors $\mathbb{F},\mathbb{G}$ from $\mathbb{C}$ to $\mathbb{C}'$, one defines $(\alpha_m)_{(C,\mu,\epsilon)}:=\alpha_C^{\mathsf{op}}$ for a comonoid $(C,\mu,\epsilon)$ in $\mathbb{C}$, and then $\alpha_m\colon \mathbb{G}_m\Rightarrow \mathbb{F}_m\colon \mathbf{Comon}\mathbb{C}\to\mathbf{Comon}\mathbb{C}'$. When furthermore $\mathbb{F}$ and $\mathbb{G}$ are symmetric, $\alpha_m\colon \mathbb{G}_m\Rightarrow\mathbb{F}_m\colon {}_{coc}\mathbf{Comon}\mathbb{C}\to{}_{coc}\mathbf{Comon}\mathbb{C}'$. %The same holds for the respective categories of cocommutative cosemigroups (with $\alpha_s$ instead of $\alpha_m$).
\item Let $\alpha\colon \mathbb{F}\Rightarrow\mathbb{G}\colon \mathbb{C}\to\mathbb{C}'$ be a monoidal transformation between strong monoidal functors, then $\alpha^{\mathsf{op}}\colon G^{\mathsf{op}}\to F^{\mathsf{op}}\colon \mathbf{C}^{\mathsf{op}}\to\mathbf{C}'^{\mathsf{op}}$ also provides a monoidal transformation $\alpha^{\mathbb{d}}\colon \mathbb{G}^{\mathbb{d}}\to\mathbb{F}^{\mathbb{d}}\colon \mathbb{C}^{\mathbb{op}}\to\mathbb{C}'^{\mathbb{op}}$\index{$\alpha^{\mathbb{d}}$}, called the {\em dual}\index{Dual of a monoidal transformation} of $\alpha$.
\end{enumerate}
\end{remark}

A {\em monoidal isomorphism}\index{Monoidal isomorphism} is a monoidal transformation which is a natural isomorphism. 

\begin{remark}\label{rem:inv_of_mon_iso_is_mon_iso}
$\alpha^{-1}\colon \mathbb{G}\Rightarrow\mathbb{F}\colon \mathbb{C}\to\mathbb{C}'$ is a monoidal isomorphism when so is $\alpha$. Accordingly, the induced natural transformation $\tilde{\alpha}\colon \tilde{\mathbb{F}}\Rightarrow\tilde{\mathbb{G}}\colon \mathbf{Mon}(\mathbb{C})\to\mathbf{Mon}(\mathbb{C}')$ is a natural isomorphism with inverse $(\tilde{\alpha})^{-1}=\widetilde{(\alpha^{-1})}$.
 \end{remark}

A {\em monoidal equivalence}\index{Monoidal equivalence} of monoidal categories is given by a monoidal functor $\mathbb{F}\colon \mathbb{C}\to\mathbb{C}'$ such that there are a monoidal functor $\mathbb{G}$ and monoidal isomorphisms $\eta\colon \mathbb{id}\Rightarrow \mathbb{G}\circ\mathbb{F}$ and $\epsilon\colon \mathbb{F}\circ\mathbb{G}\Rightarrow\mathbb{id}$. $\mathbb{C},\mathbb{C}'$ are said {\em monoidally equivalent}\index{Monoidally equivalent}.

\begin{remark}
If $\mathbb{F}$ is a monoidal equivalence, then $\tilde{\mathbb{F}}$  is an equivalence between the corresponding categories of monoids.
\end{remark}

\begin{remark}
Let $\mathbb{F},\mathbb{G}\colon \mathbb{C}'\to\mathbb{C}$ be strong monoidal functors. 
Let $\eta\colon \mathbb{id}\Rightarrow \mathbb{G}\circ\mathbb{F}$ and $\epsilon\colon \mathbb{F}\circ\mathbb{G}\Rightarrow\mathbb{id}$ be monoidal isomorphisms. Whence $\eta^{\mathbb{d}}\colon \mathbb{F}^{\mathbb{d}}\circ\mathbb{G}^{\mathbb{d}}=(\mathbb{G}\circ\mathbb{F})^{\mathbb{d}}\Rightarrow \mathbb{id}_{\mathbb{C}}^{\mathbb{d}}=\mathbb{id}_{\mathbb{C}^{\mathbb{op}}}$ and $\epsilon^{\mathbb{d}} \colon \mathbb{id}_{\mathbb{C}'^{\mathbb{op}}}=(\mathbb{id}_{\mathbb{C}'})^{\mathbb{d}}\Rightarrow \mathbb{G}^{\mathbb{d}}\circ \mathbb{F}^{\mathbb{d}}$ are also monoidal isomorphisms, and this provides a  monoidal equivalence between $\mathbb{C}^{\mathbb{op}}$ and $\mathbb{C}'^{\mathbb{op}}$.
%
%As a consequence, 
\end{remark}

\end{document}